%% file: main.tex
\documentclass[leqno]{amsart}
\usepackage[foot]{amsaddr}
\usepackage{color}
\usepackage{amsmath, amssymb, amsfonts, amsthm}
\usepackage{bm}
\usepackage{array}
\usepackage{graphicx}
\usepackage{placeins}
\usepackage{booktabs}
\usepackage{mathtools}
\usepackage{subfig}
\usepackage{subfiles}
\usepackage{hyperref}
\usepackage{xcolor}
\usepackage[maxnames=99]{biblatex}
\newtheorem{Remark}{Remark}

\newtheorem{Assumption}{Assumption}
\newtheorem{Theorem}{Theorem}

\newtheorem{Lemma}{Lemma}
\DeclareMathOperator*{\argmax}{arg\,max}

\newcommand{\bs}[1]{\boldsymbol{#1}}
\newcommand{\di}[1]{\,\mathrm{d}{#1}}

\begin{document}

\title[Median RQMC for nested integration]{Median randomized quasi-Monte Carlo for nested integration}
\author{Arved Bartuska$^{\#,1}$, Yang Liu$^{1}$, \& Ra\'{u}l Tempone$^{1,2,3}$}
\address{$^1$King Abdullah University of Science \& Technology (KAUST), Computer, Electrical and Mathematical Sciences \& Engineering Division (CEMSE), Thuwal 23955-6900, Saudi Arabia}
\address{$^2$Department of Mathematics, RWTH Aachen University, Geb\"{a}ude-1953 1.OG, Pontdriesch 14-16, 161, 52062 Aachen, Germany}
\address{$^3$Alexander von Humboldt Professor in Mathematics for Uncertainty Quantification, RWTH Aachen University, Germany}
\email{$^\#$arved.bartuska@kaust.edu.sa}

\begin{abstract}
\noindent
\subfile{./abstract.tex}
\end{abstract}

\maketitle


\section{Introduction}

Nested integration arises in various fields of high practical relevance, such as Bayesian optimal experimental design (BOED)~\cite{Cha95, Sha48, Kul59, Lin56, Rya16, Rai24, Hua13}, mathematical finance~\cite{Gor10}, and medical decision-making~\cite{Gil19}. In this work, we focus on the BOED application. The nested structure renders traditional estimation methods such as Monte Carlo (MC) or randomized quasi-Monte Carlo (RQMC) significantly more costly than for the non-nested case. The primary example of a nested integral we consider in this work is the expected information gain (EIG) of an experiment, where the outer integrand is the logarithm and the inner integrand is the likelihood of data observations. An estimator for the EIG based on a double-loop Monte Carlo (DLMC) method was first developed in \cite{Rya03}. This estimator required computational work of $\mathcal{O}(\mathrm{TOL}^{-3})$ to achieve an error tolerance of $\mathrm{TOL}>0$~\cite{Bec18}. This is a significant increase compared to MC estimators for standard, that is, non-nested, integrals with cost $\mathcal{O}(\mathrm{TOL}^{-2})$. Later, \cite{Rai18} provided a more general analysis of nested MC estimators. If the experiment must be modeled via finite elements or a similar discretization scheme, the cost of the DLMC estimator also includes a term relating to the cost and the weak convergence rate of the discretization method. This setting was analyzed in~\cite{Bec18}. A multilevel MC (MLMC) estimator for the EIG without discretization of the model was developed in \cite{God20}, where the inner number of samples determines the level of accuracy, leading to a computational cost of $\mathcal{O}(\mathrm{TOL}^{-2})$, similar to the non-nested case. The work \cite{Bec20} developed an MLMC estimator for the EIG where the level of accuracy was determined by the inner number of samples and the finite element discretization of the model. An importance sampling scheme based on the Laplace approximation ensured a convergence rate of $\mathcal{O}(\mathrm{TOL}^{-2})$ for reasonably large tolerances and $\mathcal{O}(\mathrm{TOL}^{-3})$ as $\mathrm{TOL}\to 0$ under certain assumptions on the discretization scheme. In \cite{Du24}, the level of accuracy was determined by both the discretization and the inner number of samples to create a multi-index MC (MIMC) estimator with cost $\mathcal{O}(\mathrm{TOL}^{-2})$ to estimate the EIG. In all of these estimators, the limiting factor is the rate $\mathcal{O}(\mathrm{TOL}^{-2})$ incurred by the outer MC estimation. To break this barrier, \cite{Bar24} and \cite{Kaa24} proposed estimators based on the randomized quasi-MC (RQMC) method, achieving computational costs of almost $\mathcal{O}(\mathrm{TOL}^{-10/9})$ and almost $\mathcal{O}(\mathrm{TOL}^{-1})$, respectively. Later, \cite{Bar24b} proposed a multilevel double-loop estimator employing both randomized and deterministic QMC methods to achieve computational costs of almost $\mathcal{O}(\mathrm{TOL}^{-8/9})$.

The quasi-Monte Carlo (QMC) method is a widely used numerical integration technique that employs deterministic low-discrepancy sequences to improve convergence over standard Monte Carlo. One popular example for such a low-discrepancy sequence is the Sobol' sequence \cite{Sob67}. Estimators using the first $N$ points in the Sobol' sequence in dimension $d$ can achieve an asymptotic convergence rate of $\mathcal{O}(N^{-1} (\log(N))^{d-1})$ for integrands with finite Hardy--Krause variation~\cite{dick2010digital}. The cost of such estimators is proportional to $N$, and thus the overall computational cost to achieve an error tolerance $\mathrm{TOL}>0$ is $\mathcal{O}(\mathrm{TOL}^{-1/(1-\epsilon)})$, where $\epsilon>0$ is used to hide log terms. The RQMC method yields unbiased mean estimators and variance estimates by randomizing the low-discrepancy sequences. The variance of RQMC estimators has been extensively explored in the previous literature, for instance in~\cite{lecuyer2000variance, owen1997monte, dick2010digital, liu2025randomized}. Under suitable regularity conditions, the convergence rate of RQMC estimators based on the Sobol' sequence with random scrambling can reach $\mathcal{O}(N^{-3/2} (\log(N))^{(d-1)/2})$~\cite{liu2025weak}. More recently, it has been shown that even faster and smoothness-adaptive convergence rates can be attained by taking the median of estimators based on random linear scrambled randomized sequences~\cite{Pan23, Pan24}, thus achieving convergence rates of $\mathcal{O}(N^{-\frac{c}{d}\log(N)})$ for any $c<3\log(2)/\pi^2$, where $\log$ denotes the natural logarithm. It was noted in~\cite{God26} that replacing the mean with the median over fully nested scramblings does not provide similar improvements. Such scramblings will thus not be considered in this work. Moreover, the term ``nested'' in this context is conceptually unrelated to nested integrals and should not be confused.

In this work, we extend the median RQMC (mRQMC) approach to the nested integration setting. Not only do we demonstrate that the improved efficiency of mRQMC directly translates from the standard to the nested case for both the inner and outer integrals, we also show that the overall estimator cost does not significantly increase for nested integration problems. Thus, the mRQMC method is particularly suited for the nested setting. We focus on applications from BED, where nested integrals appear in the EIG functional, quantifying the information increase (i.e., reduction of uncertainty) for possible data observations from an experiment. The EIG is equivalent to the expected Kullback--Leibler divergence between the prior and posterior distributions of the model parameters and, moreover, equivalent to the mutual information between parameters and data observations. For the case of additive Gaussian observation noise in the experiment, the likelihood is a Gaussian density. As the likelihood approaches zero, the logarithm in the outer integral displays singular behavior. This singularity deteriorates the convergence rates of RQMC estimators for the outer integral. 

We propose to remedy this issue using an importance sampling scheme. A further remedy, which we introduced in~\cite{Bar24} and applied successfully in~\cite{Bar24b}, is to truncate the Gaussian observation noise, or equivalently to restrict the argument of the inverse cumulative distribution function (CDF) away from the boundary of the unit hypercube. Truncation likewise removes the boundary singularity, but it does so at the price of a bias that must be balanced against the remaining error contributions. The importance sampling scheme considered in the present work instead leaves the value of the integral unchanged, and therefore we do not pursue truncation further here. The main contributions of this work are summarized as follows:
\begin{enumerate}
    \item[(i)] we introduce the DLmRQMC estimator for general nested integrals and analyze its estimation error under certain regularity conditions;
    \item[(ii)] we propose an importance sampling method and derive regularity conditions required to bound the DLmRQMC estimator error for the EIG application;
    \item[(iii)] we present numerical experiments to demonstrate the improved efficiency of the proposed estimator compared to established methods and conclude with an EIG example to verify its practical applicability.
\end{enumerate}

This paper is organized as follows. Section~\ref{sec:background} provides background information and introduces notation related to the RQMC method. Section~\ref{sec:median_dl_estimator} introduces the DLmRQMC estimator. Section~\ref{sec:error.bounds} analyzes the error of the DLmRQMC estimator for analytic integrands. In Section~\ref{sec:eig}, the EIG formulation is introduced and an importance sampling method is presented. In Section~\ref{sec:error.bounds.EIG}, we analyze the error of the DLmRQMC estimator for the EIG setting. Section~\ref{sec:numex} presents numerical results. Finally, Section~\ref{sec:conclusion} concludes the paper and discusses possible further directions.

\section{Background}
\label{sec:background}
We start with a standard non-nested integration problem. Let
\begin{equation}\label{def:g.bar}
    I[\varphi]\coloneqq\int_{[0,1]^d}\varphi(\bs{z})\di{}\bs{z},
\end{equation}
be a standard (non-nested) integral, where $\varphi:[0,1]^d\to\mathbb{R}$ and $d$ is a positive integer. Throughout this work, we assume that $\varphi$ is square-integrable and that the integration domain is the unit hypercube $[0,1]^d$ in dimension $d$. Moreover, let $M\in\mathbb{N}$ and $(\bs{z}^{(m)})_{m=1}^M$ be a sequence of points in $[0,1]^d$, and
\begin{equation}
    \mu_M[\varphi]\coloneqq \frac{1}{M}\sum_{m=1}^M\varphi(\bs{z}^{(m)})
\end{equation}
be the sample mean of $\varphi$ evaluated at points $\bs{z}^{(m)}$, where $1\leq m\leq M$. 
\subsection{Randomized Quasi-Monte Carlo estimator}
Let 
\begin{equation}
    \bs{z}^{(r,m)}\coloneqq\bs{\rho}^{(r)}(\bs{t}^{(m)}), \quad 1\leq r\leq R, \, 1\leq m\leq M,
\end{equation}
where $\bs{t}^{(1)},\ldots,\bs{t}^{(M)}$ are points from the deterministic Sobol' sequence in dimension $d$ and $\bs{\rho}^{(1)},\ldots,\bs{\rho}^{(R)}$ are independent random linear scramblings~\cite{Mat98} for $R\in\mathbb{N}$. The randomized Quasi-Monte Carlo (RQMC) estimator is defined as follows:
\begin{align}
 \mu_{R}\left[\mu_M\left[\varphi\right]\right] {}&\coloneqq \frac{1}{R}\sum_{r=1}^{R}\frac{1}{M}\sum_{m=1}^M\varphi(\bs{z}^{(r,m)}),\nonumber\\
 {}&=\frac{1}{R}\sum_{r=1}^{R}\mu_{M}^{(r)}\left[\varphi\right].
\end{align}
For simplicity, we define the following:
\begin{equation}
    \mu_{RM}[\varphi]\coloneqq\mu_{R}\left[\mu_M\left[\varphi\right]\right].
\end{equation}
 We consider the number of randomizations $R$ to be fixed and adjust the number of Sobol' points $M=2^{\tilde{m}}$ for some $\tilde{m}$ to meet a user-defined error tolerance $\mathrm{TOL}>0$. The particular construction of the Sobol' sequence necessitates that the number of points are chosen to be a power of two. An odd number of randomizations $R=2\tilde{r}-1$ for some $\tilde{r}$ is selected for comparison with median-based estimators where an odd number of randomizations is typically used.  When the dominating mixed derivative of $\varphi$ in the weak sense is square-integrable, the root mean squared error (RMSE) of the RQMC estimator converges almost at a rate of $3/2$ in $M$, that is,
\begin{equation}
    \mathbb{E}\left[\left(\mu_{RM}[\varphi]-I[\varphi]\right)^2\right]^{\frac{1}{2}}\leq \frac{C_{d,\epsilon}}{R^{\frac{1}{2}}M^{\frac{3}{2}-\epsilon}},
\end{equation}
for any $\epsilon>0$, where $C_{d,\epsilon}$ is independent of $M$ and $C_{d,\epsilon}\to\infty$ as $\epsilon\to 0$~\cite{Owe97, dick2010digital}. The MC convergence rate of $1/2$ in $R$ implies that increasing $M$ is more efficient, and $R$ should typically be selected as small as possible as the overall cost of this estimator is proportional to $R\times M$.

\subsection{Median Randomized Quasi-Monte Carlo estimator}
The work \cite{Pan23} suggests replacing the mean over the randomizations in the RQMC estimator with the median to obtain the median RQMC (mRQMC) estimator. Let $Z=(z^{(r)})_{r=1}^{R}$ be a sequence of real numbers where $z^{(r)}\in\mathbb{R}$ for $R=2\tilde{r}-1$. Moreover, let $\sigma_{\rm{asc}}\in\mathcal{S}$ be a permutation that rearranges $Z$ in ascending order and $\mathcal{S}$ is the set of all permutations of $R$ elements of a sequence. That is,
\begin{equation}
    \hat{z}^{(r)}\coloneqq \sigma_{\rm{asc}}(z^{(i_r)}), \quad \text{for $1\leq r\leq R$}, 
\end{equation}
for appropriately chosen indices $i_{1},\ldots,i_{R}$ such that
\begin{equation}
    \hat{z}^{(1)}\leq\ldots\leq \hat{z}^{(\tilde{r})}\leq \ldots\leq \hat{z}^{(R)}.
\end{equation}
If all elements in $Z$ are unique, then $\sigma_{\rm{asc}}$ is also unique. Otherwise, we simply pick one permutation that sorts the elements of $Z$ in ascending order. The median of $Z$ is then defined as follows:
\begin{equation}\label{def:median}
    \mathrm{med}\{Z\}\coloneqq \hat{z}^{(\tilde{r})}.
\end{equation}
The mRQMC estimator is subsequently defined as follows:
\begin{align}
 \hat{\mu}_{M}^{(\tilde{r})}[\varphi]\coloneqq\mathrm{med}\left\{(\hat{\mu}_{M}^{(r)}[\varphi])_{r=1}^R\right\},
\end{align}
where
\begin{equation}
    \hat{\mu}_{M}^{(r)}[\varphi]\coloneqq \sigma_{\rm{asc}}(\mu_{M}^{(i_{r})}[\varphi]), \quad \text{for $1\leq r\leq R$}
\end{equation}
such that $\hat{\mu}_{M}^{(1)}[\varphi]\leq\ldots\leq \hat{\mu}_{M}^{(\tilde{r})}[\varphi]\leq\ldots \leq\hat{\mu}_{M}^{(R)}[\varphi]$ are in ascending order. The motivation to replace the mean over the randomizations with the median is that the RMSE of the RQMC estimator is often dominated by large outliers and the median thus achieves smaller error for those cases. The work \cite{Pan24} demonstrated that the probability of the mRQMC absolute error being larger than $M^{-\frac{c}{d}\log(M)}$ for all $c<3\log(2)/\pi^2$ converges super-polynomially as long as $R=\Theta((\log_2(M))^2)$. In particular, this holds if the integrand $\varphi$ is analytic.

\section{Double-loop median RQMC estimator definition and error splitting}
\label{sec:median_dl_estimator}
We aim to estimate nested integrals of the form
\begin{align}\label{eq:double.loop}
  I\left[f\left(I[g]\right)\right] {}&\coloneqq \int_{[0,1]^{d_1}}f\left(\int_{[0,1]^{d_2}}g(\bs{y},\bs{x})\di{}\bs{x}\right)\di{}\bs{y},
\end{align}
where $f$ is a nonlinear function and $d_1, d_2$ are positive integers. 
\subsection{Estimator definition}
The double-loop RQMC (DLRQMC) estimator \cite{Bar24} uses the RQMC method to estimate the inner and outer integrals in the nested integral~\eqref{eq:double.loop}. This estimator has the following shape:
\begin{equation}\label{eq:dlqmc.S.R}
\mu_{SN}\left[f\left(\mu_{RM}\left[g\right]\right)\right] \coloneqq \frac{1}{S}\sum_{s=1}^{S}\frac{1}{N}\sum_{n=1}^Nf\left(\frac{1}{R}\sum_{r=1}^{R}\frac{1}{M}\sum_{m=1}^Mg(\bs{y}^{(s,n)},\bs{x}^{(s,n,r,m)})\right),
\end{equation}
where
\begin{equation}
    \bs{y}^{(s,n)}\coloneqq\bs{\tau}^{(s)}(\bs{u}^{(n)}), \quad 1\leq s\leq S, \, 1\leq n\leq N
\end{equation}
and $\bs{\tau}^{(s)}$, with $1\leq s\leq S$, $S\in\mathbb{N}$ are independent scramblings of the Sobol' sequence $\bs{u}^{(n)}$ in $d_1$, where $1\leq n\leq N$, $N\in\mathbb{N}$, and
\begin{equation}\label{eq:x.nrm}
    \bs{x}^{(s,n,r,m)}\coloneqq\bs{\rho}^{(s,n,r)}(\bs{t}^{(m)}), \quad 1\leq s\leq S, \, 1\leq n\leq N, \, 1\leq r\leq R, \, 1\leq m\leq M,
\end{equation}
where $\bs{\rho}^{(s,n,r)}$, with $1\leq s\leq S$, $1\leq n\leq N$, $N=2^{\tilde{n}}$ and $1\leq r\leq R$ are independent randomizations of the Sobol' sequence $\bs{t}^{(m)}$ in $d_2$, where $1\leq m\leq M$ and $M=2^{\tilde{m}}$.

In this work, we propose a double-loop mRQMC (DLmRQMC) estimator, given as follows:
\begin{align} {}&\hat{\mu}_{N}^{(\tilde{s})}\left[f\left(\hat{\mu}_{M}^{(\tilde{r})}[g]\right)\right]\nonumber\\
{}&\coloneqq\mathrm{med}\left\{\left(\frac{1}{N}\sum_{n=1}^Nf\left(\mathrm{med}\left\{\left(\frac{1}{M}\sum_{m=1}^Mg(\bs{y}^{(s,n)},\bs{x}^{(s,n,r,m)})\right)_{r=1}^R\right\}\right)\right)_{s=1}^S\right\},
\end{align}
where separate permutations are applied for each $s,n$, where $1\leq s\leq S$, $1\leq n\leq N$. We decompose the total integration error into an inner and an outer error as follows:
\begin{align}\label{eq:total.error}
    \mathcal{E}_{\rm{tot}}{}&\coloneqq\hat{\mu}_{N}^{(\tilde{s})}\left[f\left(\hat{\mu}_{M}^{(\tilde{r})}[g]\right)\right] - I\left[f\left(I[g]\right)\right]\nonumber\\
    {}&= \underbrace{\hat{\mu}_{N}^{(\tilde{s})}\left[f\left(\hat{\mu}_{M}^{(\tilde{r})}[g]\right)\right] - \hat{\mu}_{N}^{(\tilde{s})}\left[f\left(I[g]\right)\right]}_{\eqqcolon\mathcal{E}_{\rm{in}}} + \underbrace{\hat{\mu}_{N}^{(\tilde{s})}\left[f\left(I[g]\right)\right] - I\left[f\left(I[g]\right)\right]}_{\eqqcolon\mathcal{E}_{\rm{out}}}.
\end{align}
Notice that in the inner error term $\mathcal{E}_{\rm{in}}$, the outer median may be attained at different original indices $s$ for both terms, where $1\leq s\leq S$. We need the following lemmas to proceed.
\begin{Lemma}[Maximum difference between two sequences]\label{lem:max.diff}
   Let $A = (A^{(1)}, \dotsc, A^{(R)})$ and $B = (B^{(1)}, \dotsc, B^{(R)})$ be two sequences with size $R=2\tilde{r}-1$. Given two permutations $(\sigma_{A,\mathrm{asc}},\sigma_{B,\mathrm{asc}})\in\mathcal{S}\times\mathcal{S}$ that sort $A$ and $B$ in ascending order, respectively, it holds that $\sigma_{A,\mathrm{asc}}$ and $\sigma_{B,\mathrm{asc}}$ minimize the maximum element-wise differences of the two sequences among all pairs of permutations in $\mathcal{S}\times\mathcal{S}$. That is,
   \begin{equation}
       \max_{r=1,\ldots,R} \left\lvert \sigma_{A,\mathrm{asc}}(A^{(r)}) -  \sigma_{B,\mathrm{asc}}(B^{(r)}) \right\rvert\leq \max_{r=1,\ldots,R} \left\lvert \sigma_{A}(A^{(r)}) -  \sigma_{B}(B^{(r)}) \right\rvert
   \end{equation}
   for any $\{\sigma_{A},\sigma_{B}\}\in\mathcal{S}\times\mathcal{S}$.
\end{Lemma}
This problem is related to finding the infinite Wasserstein distance between two empirical measures, see~\cite{Cha08}. The proof is provided in Appendix~\ref{app:max.diff}.
\begin{Lemma}\label{lem:med.bound}
    Given two sequences $A = (A^{(1)}, \dotsc, A^{(R)})$ and $B = (B^{(1)}, \dotsc, B^{(R)})$ with size $R=2\tilde{r}-1$, we have that
    \begin{equation}
        \left\lvert \mathrm{med} \{A\} -  \mathrm{med} \{B\} \right\rvert \leq \max_{r=1,\ldots,R} \left\lvert A^{(r)} -  B^{(r)} \right\rvert.
    \end{equation}
\end{Lemma}
\begin{proof}
    Let $\sigma_{A,\mathrm{asc}}$ be a permutation that sorts $A$ in ascending order such that
    \begin{equation}
        \hat{A}^{(r)}\coloneqq \sigma_{A,\mathrm{asc}}\left(A^{(i_r)}\right), \quad \text{for $1\leq r\leq R$},
    \end{equation}
    for appropriately chosen indices $i_1,\ldots,i_R$ such that $(\hat{A}^{(r)})_{r=1}^R$ is sorted in ascending order and $\sigma_{B,\mathrm{asc}}$ be a permutation that sorts $B$ in ascending order, that is,
        \begin{equation}
        \hat{B}^{(r)}\coloneqq \sigma_{B,\mathrm{asc}}\left(B^{(j_r)}\right), \quad \text{for $1\leq r\leq R$},
    \end{equation}
    for appropriately chosen indices $j_1,\ldots,j_R$. From the definition of the median in~\eqref{def:median}, it follows that 
    \begin{align}\label{eq:med.max}
            \left\lvert \mathrm{med} \{A\} -  \mathrm{med} \{B\} \right\rvert &= \left\lvert \hat{A}^{(\tilde{r})} -  \hat{B}^{(\tilde{r})} \right\rvert\nonumber\\
            &\leq \max_{r=1,\ldots,R} \left\lvert \hat{A}^{(r)} -  \hat{B}^{(r)} \right\rvert\nonumber\\
            &\leq \max_{r=1,\ldots,R} \left\lvert A^{(r)} -  B^{(r)} \right\rvert,
    \end{align}
    where the last step in~\eqref{eq:med.max} follows from Lemma~\ref{lem:max.diff}.
\end{proof}
We close this section with an elementary result that converts a probability bound for a single randomization into a probability bound for the median over several randomizations and subsequently into a moment bound.
\begin{Lemma}[Median concentration and moment bound]\label{lem:median.moment}
Let $R=2\tilde{r}-1$, where $\tilde{r}\in\mathbb{N}$, let $Z^{(1)},\ldots,Z^{(R)}$ be independent real-valued random variables, and let $\mu\in\mathbb{R}$. Assume that there exist $b\geq 0$, $0<\delta<1$, and $\Delta<\infty$ such that
\begin{equation}\label{eq:med.lem.hyp}
   \mathbb{P}\left(\left|Z^{(r)}-\mu\right|\geq b\right)\leq\delta\quad\text{for }1\leq r\leq R,
   \qquad\text{and}\qquad \left|Z^{(r)}-\mu\right|\leq\Delta\ \text{almost surely}.
\end{equation}
Then the sample median satisfies
\begin{equation}\label{eq:med.lem.prob}
   \mathbb{P}\left(\left|\mathrm{med}\left\{(Z^{(r)})_{r=1}^{R}\right\}-\mu\right|\geq b\right)\leq\left(4\delta\right)^{\tilde{r}},
\end{equation}
and, for every $p\geq 1$,
\begin{equation}\label{eq:med.lem.moment}
   \mathbb{E}\left[\left|\mathrm{med}\left\{(Z^{(r)})_{r=1}^{R}\right\}-\mu\right|^{p}\right]^{\frac{1}{p}}\leq b+\Delta\left(4\delta\right)^{\frac{\tilde{r}}{p}}.
\end{equation}
If $\mathcal{Z}_1,\ldots,\mathcal{Z}_J$ are $J$ such medians, with respective centers $\mu_1,\ldots,\mu_J$, each formed from $R$ independent variables satisfying~\eqref{eq:med.lem.hyp} with the same $b$, $\delta$, and $\Delta$ and the corresponding center, and not necessarily independent of one another, then
\begin{equation}\label{eq:med.lem.max}
   \mathbb{E}\left[\max_{1\leq j\leq J}\left\{\left|\mathcal{Z}_j-\mu_j\right|^{p}\right\}\right]^{\frac{1}{p}}\leq b+\Delta\left(J\left(4\delta\right)^{\tilde{r}}\right)^{\frac{1}{p}}.
\end{equation}
\end{Lemma}

\begin{proof}
The sample median of $R=2\tilde{r}-1$ values deviates from $\mu$ by at least $b$ only if at least $\tilde{r}$ of the individual values deviate from $\mu$ by at least $b$. By independence, the probability that a prescribed set of $\tilde{r}$ indices all deviate is at most $\delta^{\tilde{r}}$, and a union bound over the $\binom{2\tilde{r}-1}{\tilde{r}}\leq 2^{2\tilde{r}-1}\leq 4^{\tilde{r}}$ such sets gives~\eqref{eq:med.lem.prob}.

For~\eqref{eq:med.lem.moment}, write $X\coloneqq|\mathrm{med}\{(Z^{(r)})_{r=1}^{R}\}-\mu|$ and split the expectation at the level $b$, that is,
\begin{align}
   \mathbb{E}\left[X^p\right]=\mathbb{E}\left[X^p\mathbf{1}_{\{X<b\}}\right]+\mathbb{E}\left[X^p\mathbf{1}_{\{X\geq b\}}\right]{}&\leq b^p+\Delta^p\,\mathbb{P}\left(X\geq b\right)\nonumber\\
   {}&\leq b^p+\Delta^p\left(4\delta\right)^{\tilde{r}},
\end{align}
where $\mathbf{1}_{\{\cdot\}}$ denotes the indicator function. Here, the almost sure bound $X\leq\Delta$ follows from~\eqref{eq:med.lem.hyp}, since the median of $R$ values lies between the smallest and the largest of them. Taking $p$th roots and using subadditivity yields~\eqref{eq:med.lem.moment}. The bound~\eqref{eq:med.lem.max} follows in the same way after a union bound over the $J$ medians, where no independence is required.
\end{proof}
In the following section, we state the assumptions required to bound the DLmRQMC estimator error for the case that both integrands are analytic and provide error bound for this setting.

\subsection{\texorpdfstring{$f$}{Outer integrand} Lipschitz, \texorpdfstring{$g$}{inner integrand} analytic, \texorpdfstring{$f\circ I[g]$}{combined integrand} analytic}\label{sec:error.bounds}
We state the following assumptions before proceeding with the first central result of this study in Theorem~\ref{thm:moment} below.
\begin{Assumption}[Lipschitz continuity of $f$]\label{asu:Lipschitz.f}
    Let $f$ be Lipschitz continuous with constant $L$ on an interval $\mathcal{I}\subseteq\mathbb{R}$, that is, there exist $L<\infty$ and $M_{\mathcal{I}}<\infty$ such that
\begin{equation}
    \left|f(z)-f(z')\right|\leq L\left|z-z'\right|
\end{equation}
for all $z,z'\in\mathcal{I}$, where $\mathcal{I}\subseteq\mathbb{R}$ is an interval containing both the range $\{I[g(\bs{y},\cdot)]:\bs{y}\in[0,1]^{d_1}\}$ and, almost surely, the range of the inner estimator $\hat{\mu}_M^{(\tilde{r})}[g(\bs{y},\cdot)]$ for all $\bs{y}\in[0,1]^{d_1}$ and all $M\geq M_{\mathcal{I}}$.
\end{Assumption}

\begin{Assumption}[Uniform analyticity of $g$]\label{asu:uniform.analytic}
Let $g(\bs{y},\cdot):[0,1]^{d_2}\to\mathbb{R}$ be analytic uniformly in $\bs{y}$, that is, there exist constants $A_g<\infty$ and $\eta_g>0$, both independent of $\bs{y}$, such that
 \begin{equation}
   \left|\frac{\partial^{\lVert\bs{\alpha}\rVert_1}}{\partial x_1^{\alpha_1}\cdots\partial x_{d_2}^{\alpha_{d_2}}}g(\bs{y},\bs{x})\right|\leq \frac{A_g\,\lVert\bs{\alpha}\rVert_1!}{\eta_g^{\lVert\bs{\alpha}\rVert_1}}
 \end{equation}
 for all $\bs{x}\in[0,1]^{d_2}$, $\bs{y}\in[0,1]^{d_1}$, and $\bs{\alpha}\in\mathbb{N}_0^{d_2}$.
\end{Assumption}
\begin{Assumption}[Analyticity of \texorpdfstring{$f\circ I[g]$}{combined integrand}]\label{asu:analytic.h}
Let
\begin{equation}\label{def:outer.integrand}
h(\bs{y})\coloneqq f\left(I\left[g(\bs{y},\cdot)\right]\right)    
\end{equation}
be analytic on $[0,1]^{d_1}$, that is, there exist constants $A_h<\infty$ and $\eta_h>0$, such that
 \begin{equation}
   \left|\frac{\partial^{\lVert\bs{\alpha}\rVert_1}}{\partial y_1^{\alpha_1}\cdots\partial y_{d_1}^{\alpha_{d_1}}}h(\bs{y})\right|\leq \frac{A_h\,\lVert\bs{\alpha}\rVert_1!}{\eta_h^{\lVert\bs{\alpha}\rVert_1}}
 \end{equation}
 for all $\bs{y}\in[0,1]^{d_1}$ and $\bs{\alpha}\in\mathbb{N}_0^{d_1}$.
\end{Assumption}
\begin{Lemma}[Probability error bound for inner and outer errors]\label{lem:total.error}
    Let $f$ satisfy Assumption~\ref{asu:Lipschitz.f} with Lipschitz constant $L$ and let $g(\bs{y},\cdot):[0,1]^{d_2}\to\mathbb{R}$ satisfy Assumption~\ref{asu:uniform.analytic}. Moreover, let the composite integrand $h$ as defined in~\eqref{def:outer.integrand} satisfy Assumption~\ref{asu:analytic.h}. Then, for every $c<3\log(2)/\pi^2$, we have the following probability bounds:
 \begin{align}
\mathbb{P}\left(\left|\mathcal{E}_{\mathrm{in}}\right| \geq M^{-\frac{c}{d_2}\log(M)}\right) &\leq SN\left(\frac{4B_{d_2}}{\sqrt{\tilde{m}}}\right)^{\tilde{r}},
\label{eq:inner.error.prob.bound}\\
\mathbb{P}\left(\left|\mathcal{E}_{\mathrm{out}} \right| \geq N^{-\frac{c}{d_1}\log(N)}\right) &\leq \left(\frac{4B_{d_1}}{\sqrt{\tilde{n}}}\right)^{\tilde{s}}.
\label{eq:outer.error.prob.bound}
\end{align}
Here, the probability is taken with respect to the joint law of the independent random linear scramblings $\bs{\tau}^{(s)}$, $1\leq s\leq S$, and $\bs{\rho}^{(s,n,r)}$, $1\leq s\leq S$, $1\leq n\leq N$, $1\leq r\leq R$, that define the DLmRQMC estimator.
\end{Lemma}
\begin{proof}
We start by bounding the absolute inner error as follows using Lemma~\ref{lem:med.bound}:
\begin{align}\label{eq:inner.bound.max}
    \left|\mathcal{E}_{\rm{in}}\right|{}&=\left|\hat{\mu}_{N}^{(\tilde{s})}\left[f\left(\hat{\mu}_{M}^{(\tilde{r})}[g]\right)\right] - \hat{\mu}_{N}^{(\tilde{s})}\left[f\left(I[g]\right)\right]\right|\nonumber\\
    {}&\leq \max_{s=1,\ldots,S}\left|\frac{1}{N}\sum_{n=1}^{N}f\left(\hat{\mu}_{M}^{(\tilde{r})}\left[g(\bs{y}^{(s,n)},\cdot)\right]\right) - \frac{1}{N}\sum_{n=1}^{N}f\left(I\left[g(\bs{y}^{(s,n)},\cdot)\right]\right)\right|.
\end{align}
By the triangle inequality and Assumption~\ref{asu:Lipschitz.f}, it follows that
\begin{align}\label{eq:inner.error.bound}
{}&\max_{s=1,\ldots,S}\left|\frac{1}{N}\sum_{n=1}^{N}f\left(\hat{\mu}_{M}^{(\tilde{r})}\left[g(\bs{y}^{(s,n)},\cdot)\right]\right) - \frac{1}{N}\sum_{n=1}^{N}f\left(I\left[g(\bs{y}^{(s,n)},\cdot)\right]\right)\right|\nonumber\\
{}&\leq\max_{s=1,\ldots,S}\frac{1}{N}\sum_{n=1}^{N}\left|f\left(\hat{\mu}_{M}^{(\tilde{r})}\left[g(\bs{y}^{(s,n)},\cdot)\right]\right) - f\left(I\left[g(\bs{y}^{(s,n)},\cdot)\right]\right)\right|\nonumber\\
{}&\leq\max_{s=1,\ldots,S}\frac{L}{N}\sum_{n=1}^{N}\left|\hat{\mu}_{M}^{(\tilde{r})}\left[g(\bs{y}^{(s,n)},\cdot)\right] - I\left[g(\bs{y}^{(s,n)},\cdot)\right]\right|\nonumber\\
{}&\leq\max_{\substack{s=1,\ldots,S\\ n=1,\ldots,N}}L\left|\underbrace{\hat{\mu}_{M}^{(\tilde{r})}\left[g(\bs{y}^{(s,n)},\cdot)\right] - I\left[g(\bs{y}^{(s,n)},\cdot)\right]}_{\eqqcolon\mathcal{E}_{R,M}(\bs{y}^{(s,n)})}\right|.
\end{align}
From Assumption~\ref{asu:uniform.analytic} and~\cite[Theorem 2]{Pan24}, it follows for every $c<3\log(2)/\pi^2$ that there exists $B_{d_2}<\infty$ depending on $c$ and $d_2$ and there exists $M_0<\infty$ depending on $c$, $d_2$, $A_{g}$, and $\eta_{g}$ such that
\begin{equation}\label{eq:delta.inner.s.n}
    \mathbb{P}\left(\left|\frac{1}{M}\sum_{m=1}^Mg(\bs{y},\bs{x}^{(r,m)})- I\left[g(\bs{y},\cdot)\right] \right|\geq M^{-\frac{c}{d_2}\log(M)}\Bigg|\bs{y}\right)\leq \frac{B_{d_2}}{\sqrt{\tilde{m}}}
\end{equation}
for all $\bs{y}\in[0,1]^{d_1}$, all $r=1,\ldots,R$, and for all $2^{\tilde{m}}=M\geq M_0$. Applying Lemma~\ref{lem:median.moment}\eqref{eq:med.lem.prob} with $Z^{(r)}=\frac{1}{M}\sum_{m}g(\bs{y},\bs{x}^{(r,m)})$, $\mu=I[g(\bs{y},\cdot)]$, and $b=M^{-c\log(M)/d_2}$ yields
\begin{equation}\label{eq:delta.inner.y}
    \mathbb{P}\left(\left|\mathcal{E}_{R,M}(\bs{y}) \right|\geq M^{-\frac{c}{d_2}\log(M)}\Bigg|\bs{y}\right)\leq \left(\frac{4B_{d_2}}{\sqrt{\tilde{m}}}\right)^{\tilde{r}},
\end{equation}
because the inner scramblings $\bs{\rho}^{(s,n,r)}$ are iid and independent of the outer scramblings $\bs{\tau}^{(s)}$, where the right-hand side is independent of $\bs{y}\in[0,1]^{d_1}$, and thus it follows that
\begin{equation}
    \mathbb{P}\left(\left|\mathcal{E}_{R,M}(\bs{y}) \right|\geq M^{-\frac{c}{d_2}\log(M)}\right)\leq \left(\frac{4B_{d_2}}{\sqrt{\tilde{m}}}\right)^{\tilde{r}}
\end{equation}
unconditionally for any $\bs{y}\in[0,1]^{d_1}$. It follows that
\begin{equation}\label{eq:delta.inner.max}
    \mathbb{P}\left(\max_{\substack{s=1,\ldots,S\\ n=1,\ldots,N}}\left|\mathcal{E}_{R,M}(\bs{y}^{(s,n)}) \right|\geq M^{-\frac{c}{d_2}\log(M)}\right)\leq SN\left(\frac{4B_{d_2}}{\sqrt{\tilde{m}}}\right)^{\tilde{r}}
\end{equation}
and hence, by~\eqref{eq:inner.error.bound}, it follows that
\begin{equation}
    \mathbb{P}\left(\left|\mathcal{E}_{\mathrm{in}}\right|\geq LM^{-\frac{c}{d_2}\log(M)}\right)\leq SN\left(\frac{4B_{d_2}}{\sqrt{\tilde{m}}}\right)^{\tilde{r}}.
\end{equation}
To absorb $L$, let $c<c'<3\log(2)/\pi^2$ and replace $c$ with $c'$ in the above derivations. Then, let $M_0$ depend on $c$, $d_2$, $A_g$, $\eta_g$, and $L$ such that $LM^{-c'\log(M)/d_2}\leq M^{-c\log(M)/d_2}$ for all $M\geq M_0$. It follows that~\eqref{eq:inner.error.prob.bound} holds for all $M\geq M_0$ now also depending on $L$.

From Assumption~\ref{asu:analytic.h} and~\cite[Theorem 2]{Pan24}, it follows for every $c<3\log(2)/\pi^2$ that there exists $B_{d_1}<\infty$ depending on $c$ and $d_1$ and there exists $N_0<\infty$ depending on $c$, $d_1$, $A_{h}$, and $\eta_{h}$ such that
\begin{equation}
    \mathbb{P}\left(\left|\frac{1}{N}\sum_{n=1}^{N}h(\bs{y}^{(s,n)})-I[h]\right|\geq N^{-\frac{c}{d_1}\log(N)}\right)\leq \frac{B_{d_1}}{\sqrt{\tilde{n}}}
\end{equation}
for all $2^{\tilde{n}}=N\geq N_0$. Applying Lemma~\ref{lem:median.moment}\eqref{eq:med.lem.prob} with $Z^{(s)}=\frac{1}{N}\sum_{n}h(\bs{y}^{(s,n)})$, $\mu=I[h]$, and $b=N^{-c\log(N)/d_1}$ yields~\eqref{eq:outer.error.prob.bound}. 
\end{proof}
\begin{Theorem}[Moment bound]\label{thm:moment}
Given the Assumptions of Lemma~\ref{lem:total.error}, for every $c<3\log(2)/\pi^2$, there exists $N_0<\infty$ depending on $c$, $d_1$, $A_h$, and $\eta_h$, and there exists $M_0<\infty$ depending on $c$, $d_2$, $A_g$, $\eta_g$, and $L$, such that
\begin{align}\label{eq:cor}
  \mathbb{E}\left[\left|\mathcal{E}_{\mathrm{tot}}\right|^p\right]^{\frac{1}{p}}\leq{}& N^{-\frac{c}{d_1}\log(N)}+\left(\frac{4B_{d_1}}{\sqrt{\tilde{n}}}\right)^{\frac{\tilde{s}}{p}}2\left\lVert h\right\rVert_{L^{\infty}([0,1]^{d_1})}
  \nonumber\\
  {}&+LM^{-\frac{c}{d_2}\log(M)}+L\left(S\left(\frac{4B_{d_2}}{\sqrt{\tilde{m}}}\right)^{\tilde{r}}\right)^{\frac{1}{p}}2\max_{\bs{y}\in[0,1]^{d_1}}\left\lVert g(\bs{y},\cdot)\right\rVert_{L^{\infty}([0,1]^{d_2})}
\end{align}
for every $p\geq 1$, where $\mathcal{E}_{\rm{tot}}$ is as in~\eqref{eq:total.error}, for all $N\geq N_0$, and $M\geq M_0$. The bound in~\eqref{eq:cor} tends to zero super-polynomially provided that $S=\Theta((\tilde{n})^2)$, $R=\Theta((\tilde{m})^2)$, and $\tilde{n}=\Theta((\tilde{m})^2)$.
\end{Theorem}
In particular, taking $p=2$ in Theorem~\ref{thm:moment} bounds the RMSE of the DLmRQMC estimator. Moreover, in contrast to Lemma~\ref{lem:total.error}, no relation between $\log_2(N)$ and $\log_2(M)$ is required.
\begin{proof}
    It follows from~\eqref{eq:total.error} and Minkowski's inequality that
    \begin{equation}
         \mathbb{E}\left[\left|\mathcal{E}_{\mathrm{tot}}\right|^p\right]^{\frac{1}{p}}\leq \mathbb{E}\left[\left|\mathcal{E}_{\mathrm{in}}\right|^p\right]^{\frac{1}{p}}+\mathbb{E}\left[\left|\mathcal{E}_{\mathrm{out}}\right|^p\right]^{\frac{1}{p}}
    \end{equation}
    for every $p\geq 1$. The result in~\eqref{eq:outer.error.prob.bound} together with Lemma~\ref{lem:median.moment}\eqref{eq:med.lem.moment} applied with 
    \begin{equation}
        \begin{cases}
            Z^{(s)}=\frac{1}{N}\sum_{n=1}^{N}h(\bs{y}^{(s,n)}),\\
            \mu=I[h], \\
            b=N^{-c\log(N)/d_1},\\
            \delta=B_{d_1}/\sqrt{\tilde{n}},\\
            \Delta=2\left\lVert h\right\rVert_{L^{\infty}([0,1]^{d_1})}
        \end{cases}
    \end{equation}
    yields
    \begin{equation}
        \mathbb{E}\left[\left|\mathcal{E}_{\mathrm{out}}\right|^p\right]^{\frac{1}{p}}\leq N^{-\frac{c}{d_1}\log(N)}+\left(\frac{4B_{d_1}}{\sqrt{\tilde{n}}}\right)^{\frac{\tilde{s}}{p}}2\left\lVert h\right\rVert_{L^{\infty}([0,1]^{d_1})}.
    \end{equation}
    We proceed as in~\eqref{eq:inner.bound.max} and~\eqref{eq:inner.error.bound} to obtain
    \begin{align}\label{eq:max.sum}
        \left|\mathcal{E}_{\mathrm{in}}\right|{}&\leq\max_{s=1,\ldots,S}\frac{L}{N}\sum_{n=1}^{N}\left|\mathcal{E}_{R,M}(\bs{y}^{(s,n)})\right|\nonumber\\
        {}&\leq\frac{L}{N}\sum_{n=1}^{N}\max_{s=1,\ldots,S}\left|\mathcal{E}_{R,M}(\bs{y}^{(s,n)})\right|.
    \end{align}
    From Minkowski's inequality, it follows that
    \begin{align}\label{eq:Minkowski.n}
        \mathbb{E}\left[\left|\mathcal{E}_{\mathrm{in}}\right|^{p}\right]^{\frac{1}{p}}
        {}&\leq\frac{L}{N}\sum_{n=1}^{N}\mathbb{E}\left[\left(\max_{s=1,\ldots,S}\left|\mathcal{E}_{R,M}(\bs{y}^{(s,n)})\right|\right)^p\right]^{\frac{1}{p}}\nonumber\\
        {}&\leq L \max_{n=1,\ldots,N}\mathbb{E}\left[\left(\max_{s=1,\ldots,S}\left|\mathcal{E}_{R,M}(\bs{y}^{(s,n)})\right|\right)^p\right]^{\frac{1}{p}}.
    \end{align}
Moreover, it follows almost surely that
\begin{align}\label{eq:almost.surely}
    \left|\mathcal{E}_{R,M}(\bs{y})\right|{}&\leq\left|\mathrm{med}\left\{\left(\frac{1}{M}\sum_{m=1}^{M}g(\bs{y},\bs{x}^{(r,m)})\right)_{r=1}^{R}\right\}\right|+\left|I\left[g(\bs{y},\cdot)\right]\right|\nonumber\\
    {}&\leq 2\max_{\bs{y}\in[0,1]^{d_1}}\left\lVert g(\bs{y},\cdot)\right\rVert_{L^{\infty}([0,1]^{d_2})}.
\end{align}
Applying Lemma~\ref{lem:median.moment}\eqref{eq:med.lem.max} with 
\begin{align}
\begin{cases}
J{}&=S,\nonumber\\ 
Z_s^{(r)}{}&=\frac{1}{M}\sum_{m=1}^{M}g(\bs{y}^{(s,n)},\bs{x}^{(s,n,r,m)}),\nonumber\\ 
\mathcal{Z}_s{}&=\mathrm{med}\{(Z_s^{(r)})_{r=1}^R\},\nonumber\\  \mu_s{}&=I\left[g(\bs{y}^{(s,n)},\cdot)\right],\nonumber\\ 
b{}&=M^{-\frac{c}{d_2}\log(M)},\quad \text{and}\nonumber\\ 
\Delta{}&=2\max_{\bs{y}\in[0,1]^{d_1}}\left\lVert g(\bs{y},\cdot)\right\rVert_{L^{\infty}([0,1]^{d_2})}\nonumber\\ 
\end{cases}
\end{align}

together with~\eqref{eq:delta.inner.y} and~\eqref{eq:almost.surely} yields
\begin{align}
    {}&\mathbb{E}\left[\left(\max_{s=1,\ldots,S}\left|\mathcal{E}_{R,M}(\bs{y}^{(s,n)})\right|\right)^p\Bigg|\left(\bs{y}^{(s,n)}\right)_{s=1}^S\right]^{\frac{1}{p}}\nonumber\\
    {}&\leq M^{-\frac{c}{d_2}\log(M)}+2\max_{\bs{y}\in[0,1]^{d_1}}\left\lVert g(\bs{y},\cdot)\right\rVert_{L^{\infty}([0,1]^{d_2})}\left(S\left(\frac{4B_{d_2}}{\sqrt{\tilde{m}}}\right)^{\tilde{r}}\right)^{\frac{1}{p}},
\end{align}
where the right-hand side is independent of $\bs{y}^{(s,n)}\in[0,1]^{d_1}$ for $1\leq s\leq S$, $1\leq n\leq N$. Thus, it follows that
\begin{align}
    {}&\mathbb{E}\left[\left(\max_{s=1,\ldots,S}\left|\mathcal{E}_{R,M}(\bs{y}^{(s,n)})\right|\right)^p\right]\nonumber\\
    {}&\leq \left(M^{-\frac{c}{d_2}\log(M)}+2\max_{\bs{y}\in[0,1]^{d_1}}\left\lVert g(\bs{y},\cdot)\right\rVert_{L^{\infty}([0,1]^{d_2})}\left(S\left(\frac{4B_{d_2}}{\sqrt{\tilde{m}}}\right)^{\tilde{r}}\right)^{\frac{1}{p}}\right)^p.
\end{align}
Now, it follows that
\begin{align}
    {}&\max_{n=1\ldots,N}\mathbb{E}\left[\left(\max_{s=1,\ldots,S}\left|\mathcal{E}_{R,M}(\bs{y}^{(s,n)})\right|\right)^p\right]^{\frac{1}{p}}\nonumber\\
    {}&\leq M^{-\frac{c}{d_2}\log(M)}+2\max_{\bs{y}\in[0,1]^{d_1}}\left\lVert g(\bs{y},\cdot)\right\rVert_{L^{\infty}([0,1]^{d_2})}\left(S\left(\frac{4B_{d_2}}{\sqrt{\tilde{m}}}\right)^{\tilde{r}}\right)^{\frac{1}{p}},
\end{align}
and finally, substituting into~\eqref{eq:Minkowski.n} yields
\begin{align}
    \mathbb{E}\left[\left|\mathcal{E}_{\mathrm{in}}\right|^{p}\right]^{\frac{1}{p}}
        {}&\leq LM^{-\frac{c}{d_2}\log(M)}+L\left(S\left(\frac{4B_{d_2}}{\sqrt{\tilde{m}}}\right)^{\tilde{r}}\right)^{\frac{1}{p}}2\max_{\bs{y}\in[0,1]^{d_1}}\left\lVert g(\bs{y},\cdot)\right\rVert_{L^{\infty}([0,1]^{d_2})}.
\end{align}
\end{proof}
\begin{Remark}[Pre-asymptotic behavior of the DLmRQMC estimator]
    The bounds on the probabilities in Lemma~\ref{lem:total.error} only apply for sample sizes $N\geq N_0$ and $M\geq M_0$. These quantities tend to infinity as $c$ approaches $3\log(2)/\pi^2$ from below and are used to absorb the lower-order factor $2^{B_1\tilde{m}\log(\tilde{m})}$ appearing in the proof of~\cite[Corollary 3.]{Pan24}. The result is understood to be asymptotic in $N$ and $M$, and in particular the constants $B_{d_1}$ and $B_{d_2}$, corresponding to $B_2$ in the notation of~\cite{Pan24}, may not be quantitatively meaningful for practical sample sizes. We demonstrate in Section~\ref{sec:numex} that the super-polynomial convergence guaranteed by Lemma~\ref{lem:total.error} is already apparent at practically feasible sizes, in contrast with the RQMC estimator, whose observed rate remains $3/2$ over the same range.
\end{Remark}

Next, we consider the case where the composite outer integrand $h$ as defined in~\eqref{def:outer.integrand} is not analytic and instead only satisfies certain smoothness assumptions outlined below. We demonstrate in the following section that this setting applies to the EIG of an experiment after an appropriate importance sampling scheme is applied.

\section{Expected information gain estimation}
\label{sec:eig}
\subsection{Definition and notation}
The EIG is a popular example of a nested integral. We provide a brief derivation of the EIG to motivate its relevance in BOED. Let
\begin{equation}\label{eq:data.model}
    \bs{Y}=\bs{G}(\bs{\theta},\bs{\xi})+\bs{\varepsilon}
\end{equation}
be a model for observation data $\bs{Y}\in\mathbb{R}^{d_y}$. Here, $\bs{\varepsilon}\in\mathbb{R}^{d_y}$ denotes additive observation noise from a Gaussian distribution $\mathcal{N}(\bs{0},\bs{\Sigma}_{\bs{\varepsilon}})$, where $\bs{0}$ is the zero vector in $\mathbb{R}^{d_y}$ and $\bs{\Sigma}_{\bs{\epsilon}}\in\mathbb{R}^{d_y\times d_y}$ is the noise covariance. We assume that $\bs{\Sigma}_{\bs{\epsilon}}$ is known and that its smallest eigenvalue is bounded away from zero. Moreover, $\bs{\theta}\in\bs{\Theta}\subset\mathbb{R}^{d_{\theta}}$ are unknown model parameters and $\bs{\xi}\in\bs{\Xi}\subset\mathbb{R}^{d_{\xi}}$ are design parameters. Finally,  $\bs{G}:\bs{\Theta}\times\bs{\Xi}\to \mathbb{R}^{d_y}$ is a model function that relates the parameters and design to the data. Throughout this work, we consider the design fixed and omit it in the notation. 

In the Bayesian setting, a prior distribution $\pi(\cdot)$ is prescribed to $\bs{\theta}$, encoding knowledge about the model parameters before observing any data. Finding parameters $\bs{\theta}$ for given observation data $\bs{Y}$ such that the equation in~\eqref{eq:data.model} is satisfied is known as an inverse problem. In the Bayesian setting, the goal is to find a probability distribution $\pi(\bs{\theta}|\bs{Y})$, called the posterior distribution. This distribution encompasses our knowledge regarding the parameters after observing data $\bs{Y}$. The field of BOED is concerned with the task of finding experiment designs $\bs{\xi}$ such that the Bayesian inverse problem can be solved easily or optimally. The design must be chosen before data observations are made, and thus we choose an objective function that involves an expectation over $\bs{Y}$. One natural choice for such an objective function is the mutual information between $\bs{\theta}$ and $\bs{Y}$, which is equivalent to the EIG, or also called the expected Kullback--Leibler divergence between the posterior and prior distributions. The posterior is typically inaccessible and Bayes' Theorem is used to rewrite the EIG in terms of the likelihood $\pi(\bs{Y}|\bs{\theta})$ and the evidence $\pi(\bs{Y})$ as follows:
\begin{align}
    EIG\coloneqq{}&\mathbb{E}_{\bs{Y}}[D_{\rm{KL}}(\pi(\bs{\theta}|\bs{Y})||\pi(\bs{\theta}))],\nonumber\\ 
    {}&=\int_{\mathcal{Y}}\int_{\bs{\Theta}}\left[\log\left(\pi(\bs{\theta}|\bs{Y})\right)-\log(\pi(\bs{\theta}))\right]\pi(\bs{\theta}|\bs{Y})\mathrm{d}\bs{\theta}\pi(\bs{Y})\mathrm{d}\bs{Y}.
\end{align}
We typically do not have access to the evidence $\pi(\bs{Y})$ nor the posterior $\pi(\bs{\theta}|\bs{Y})$, so we write the evidence as follows:
\begin{equation}
    \pi(\bs{Y})=\int_{\bs{\Theta}}\pi(\bs{Y}|\bs{\theta})\pi(\bs{\theta})\mathrm{d}\bs{\theta},
\end{equation}
and write the posterior as follows, using Bayes' Theorem:
\begin{equation}
    \pi(\bs{\theta}|\bs{Y})=\frac{\pi(\bs{Y}|\bs{\theta})\pi(\bs{\theta})}{\int_{\bs{\Theta}}\pi(\bs{Y}|\bs{\theta}')\pi(\bs{\theta}')\mathrm{d}\bs{\theta}'}.
\end{equation}
It follows from~\eqref{eq:data.model} for additive Gaussian noise that the likelihood term is also Gaussian, that is,
\begin{align}
    \pi(\bs{Y}|\bs{\theta}'){}&=\frac{e^{-\frac{1}{2}\left\lVert \bs{Y}-\bs{G}(\bs{\theta}')\right\rVert_{\bs{\Sigma}_{\bs{\varepsilon}}^{-1}}^2}}{\det(2\pi\bs{\Sigma}_{\bs{\varepsilon}})^{\frac{1}{2}}},\nonumber\\
    {}&=\frac{e^{-\frac{1}{2}\left\lVert \bs{G}(\bs{\theta})+\bs{\varepsilon}-\bs{G}(\bs{\theta}')\right\rVert_{\bs{\Sigma}_{\bs{\varepsilon}}^{-1}}^2}}{\det(2\pi\bs{\Sigma}_{\bs{\varepsilon}})^{\frac{1}{2}}},
\end{align}
and
\begin{equation}
    \pi(\bs{Y}|\bs{\theta})=\frac{e^{-\frac{1}{2}\left\lVert \bs{\varepsilon}\right\rVert_{\bs{\Sigma}_{\bs{\varepsilon}}^{-1}}^2}}{\det(2\pi\bs{\Sigma}_{\bs{\varepsilon}})^{\frac{1}{2}}},
\end{equation}
where 
\begin{equation}
\left\lVert\cdot\right\rVert_{\bs{\Sigma}_{\bs{\varepsilon}}^{-1}}^2\coloneqq \left\lVert {\bs{\Sigma}_{\bs{\varepsilon}}^{-\frac{1}{2}}}\cdot\right\rVert^2.
\end{equation}
Thus, the EIG is written as follows:
\begin{align}\label{eq:EIG}
    EIG{}&= \int_{\bs{\Theta}}\int_{\mathcal{Y}}\left[\log\left(\frac{e^{-\frac{1}{2}\left\lVert \bs{\varepsilon}\right\rVert_{\bs{\Sigma}_{\bs{\varepsilon}}^{-1}}^2}}{\det(2\pi\bs{\Sigma}_{\bs{\varepsilon}})^{\frac{1}{2}}}\right)\right.\nonumber\\
    {}&\quad\left.-\log\left(\int_{\bs{\Theta}}\frac{e^{-\frac{1}{2}\left\lVert \bs{G}(\bs{\theta})+\bs{\varepsilon}-\bs{G}(\bs{\theta}')\right\rVert_{\bs{\Sigma}_{\bs{\varepsilon}}^{-1}}^2}}{\det(2\pi\bs{\Sigma}_{\bs{\varepsilon}})^{\frac{1}{2}}}\pi(\bs{\theta}')\di{}\bs{\theta}'\right)\right]\frac{e^{-\frac{1}{2}\left\lVert \bs{\varepsilon}\right\rVert_{\bs{\Sigma}_{\bs{\varepsilon}}^{-1}}^2}}{\det(2\pi\bs{\Sigma}_{\bs{\varepsilon}})^{\frac{1}{2}}}\di{}\bs{\varepsilon}\pi(\bs{\theta})\di{}\bs{\theta}.
\end{align}
Notice that the normalization constants inside the logarithms cancel with each other. Moreover, the first term is simply a Gaussian integral with a closed-form expression, and we thus focus on the second term, which constitutes a nested integral. We assume throughout this work that the prior distribution $\pi(\cdot)$ is the uniform distribution in dimension $d_2$, and thus we assume that $\bs{\Theta}=[0,1]^{d_2}$. Moreover, we assume that $\bs{\varepsilon}$ is a centered Gaussian random variable in dimension $d_1-d_2$ with $d_1>d_2$. It follows that $\mathcal{Y}=\mathbb{R}^{d_1-d_2}$. Employing the QMC method to estimate the outer integral in the EIG expression is not straightforward, as the integration domain of the Gaussian distribution, that is, $\mathbb{R}^{d_1-d_2}$, must be mapped back to $[0,1]^{d_1-d_2}$. The standard approach is to apply the inverse CDF of the standard normal distribution, thus introducing singularities at the boundaries. One obtains the following nested integral:
\begin{align}\label{eq:EIG.uh}
    I[f(I[g])]= \int_{[0,1]^{d_1}}\log\left(\int_{[0,1]^{d_2}}e^{-\frac{1}{2}\left\lVert \bs{G}(\bs{y}_1)+\bs{\Sigma}_{\bs{\varepsilon}}^{\frac{1}{2}}\bs{\Phi}^{-1}(\bs{y}_2)-\bs{G}(\bs{x})\right\rVert_{\bs{\Sigma}_{\bs{\varepsilon}}^{-1}}^2}\di{}\bs{x}\right)\di{}\bs{y},
\end{align}
where $\bs{y}_1\coloneqq \bs{\theta}$, $\bs{x}\coloneqq\bs{\theta}'$, $\bs{y}=(\bs{y}_1,\bs{y}_2)$, and
\begin{equation}
    \bs{\Sigma}_{\bs{\varepsilon}}=\bs{\Sigma}_{\bs{\varepsilon}}^{\frac{1}{2}}\bs{\Sigma}_{\bs{\varepsilon}}^{\frac{1}{2}}.
\end{equation}
It follows that $f(z)=\log(z)$ and
\begin{equation}\label{def:g}
    g(\bs{y},\bs{x})=\exp\left(-\frac{1}{2}\left\lVert \bs{G}(\bs{y}_1)+\bs{\Sigma}_{\bs{\varepsilon}}^{\frac{1}{2}}\bs{\Phi}^{-1}(\bs{y}_2)-\bs{G}(\bs{x})\right\rVert_{\bs{\Sigma}_{\bs{\varepsilon}}^{-1}}^2\right)
\end{equation}
for the EIG application. It follows from~\eqref{eq:EIG} that
\begin{equation}
    EIG=-\frac{d_1-d_2}{2}-I[f(I[g])].
\end{equation}
We introduce the following notation:
\begin{align}\label{def:Lambda}
    \Lambda(\bs{y},\bs{x}){}&\coloneqq\frac{1}{2}\left\lVert \bs{\Sigma}_{\bs{\varepsilon}}^{-\frac{1}{2}}\left(\bs{G}(\bs{y}_1)-\bs{G}(\bs{x})\right)+\bs{\Phi}^{-1}(\bs{y}_2)\right\rVert^2\nonumber\\
    {}&=\frac{1}{2}\left\lVert \bs{G}(\bs{y}_1)+\bs{\Sigma}_{\bs{\varepsilon}}^{\frac{1}{2}}\bs{\Phi}^{-1}(\bs{y}_2)-\bs{G}(\bs{x})\right\rVert_{\bs{\Sigma}_{\bs{\varepsilon}}^{-1}}^2,
\end{align}
and thus it follows that $g(\cdot,\cdot)=e^{-\Lambda(\cdot,\cdot)}$.
The setting where both the inner integrand $g$ and the composite integrand $h=f\circ I[g]$ are analytic does not typically apply to the EIG setting, where $f$ is the logarithm and $g$ is the likelihood of data observations. The reason for this is that additive Gaussian noise in the data model induces a singularity in the inverse CDF of the parameters of interest once they are mapped to the unit hypercube. This singularity causes the likelihood to go to zero as the integration variable approaches the edges of the unit hypercube, and as a result, the logarithm approaches negative infinity. This setting was analyzed in \cite{Bar24}, where it was shown that a convergence rate of almost one is possible in the RMSE of the double-loop RQMC estimator and no faster error convergence was observed. This was shown using Owen's boundary growth condition, as the Hardy--Krause variation of the integrand $h$ is infinite in this case. In~\cite[Remark 2.]{Pan25}, it is demonstrated that a similar rate is also possible for the mRQMC estimator based on the analysis in~\cite{liu2025randomized}. This is consistent with the behavior observed in Figure~\ref{fig:EIG.outer} (left column), where the RQMC and the mRQMC estimators converge at the same rate of one. We therefore modify the integrand so that stronger guarantees become available to the median RQMC estimator. Moreover, we assume that the model $\bs{G}$ satisfies the following regularity assumption in $\bs{x}\in [0,1]^{d_2}$.
\begin{Assumption}[Analytic model $\bs{G}$]\label{asu:model}
There exist constants $A_{\bs{G}}<\infty$ and $\eta_{\bs{G}}>0$ such that
\begin{equation}
   \left\lVert\frac{\partial^{\lVert\bs{\alpha}\rVert_1}}{\partial x_1^{\alpha_1}\cdots\partial x_{d_{2}}^{\alpha_{d_{2}}}}\bs{G}(\bs{z})\right\rVert\leq\frac{A_{\bs{G}}\lVert\bs{\alpha}\rVert_1!}{\eta_{\bs{G}}^{\lVert\bs{\alpha}\rVert_1}}
\end{equation}
for all $\bs{z}\in[0,1]^{d_2}$ and $\bs{\alpha}\in\mathbb{N}_0^{d_2}$.
\end{Assumption}
It follows from Assumption~\ref{asu:model} that the inner integrand $g$ in~\eqref{def:g} is analytic uniformly in $\bs{y}\in[0,1]^{d_1}$; that is, Assumption~\ref{asu:uniform.analytic} holds for the EIG application. We formalize this result in the following lemma.
\begin{Lemma}[Uniform analyticity of $g$]\label{Lem:asu.2}
    Let Assumption~\ref{asu:model} hold for $\bs{G}$ as in~\eqref{eq:data.model}. Then, Assumption~\ref{asu:uniform.analytic} holds for $g$ as in~\eqref{def:g}. That is, there exist constants $A_g<\infty$ and $\eta_g>0$ depending on $A_G$, $\eta_G$, and $\bs{\Sigma}_{\bs{\varepsilon}}$ such that
 \begin{equation}
   \left|\frac{\partial^{\lVert\bs{\alpha}\rVert_1}}{\partial x_1^{\alpha_1}\cdots\partial x_{d_2}^{\alpha_{d_2}}}g(\bs{y},\bs{x})\right|\leq \frac{A_g\,\lVert\bs{\alpha}\rVert_1!}{\eta_g^{\lVert\bs{\alpha}\rVert_1}}
 \end{equation}
 for all $\bs{x}\in[0,1]^{d_2}$, $\bs{y}\in[0,1]^{d_1}$, and $\bs{\alpha}\in\mathbb{N}_0^{d_2}$.
\end{Lemma}
For the proof, see Appendix~\ref{app:uniform.analyticity}. We note that although customary in the EIG literature, an observation noise model that is bounded could be applied, for which case the EIG may have sufficient regularity to satisfy the assumptions stated in the previous section; however, we will not discuss such cases in the present work and focus on the additive Gaussian noise case. 

\subsection{Importance sampling}
 An importance sampling scheme was introduced in~\cite{Liu23} to regularize an integrand and avoid boundary singularities. We propose a similar importance sampling method here to transform the integrand $h$ so that the boundary singularity from the logarithm is removed. Although this ensures bounded Hardy--Krause variation, the composite integrand is still not analytic. We start again with the nested integral in~\eqref{eq:EIG} and apply a transformation of variables $\bs{\varepsilon}=\bs{\Sigma}_{\bs{\varepsilon}}^{\frac{1}{2}}\bs{z}$ first for convenience, where $\bs{z}$ is standard Gaussian:
\begin{align}
    {}&\int_{\bs{\Theta}}\int_{\mathcal{Y}}\log\left(\int_{\bs{\Theta}}e^{-\frac{1}{2}\left\lVert \bs{G}(\bs{\theta})+\bs{\varepsilon}-\bs{G}(\bs{\theta}')\right\rVert_{\bs{\Sigma}_{\bs{\varepsilon}}^{-1}}^2}\di{}\bs{\theta}'\right)\frac{e^{-\frac{1}{2}\left\lVert \bs{\varepsilon}\right\rVert_{\bs{\Sigma}_{\bs{\varepsilon}}^{-1}}^2}}{\det(2\pi\bs{\Sigma}_{\bs{\varepsilon}})^{\frac{1}{2}}}\di{}\bs{\varepsilon}\di{}\bs{\theta}\nonumber\\
    {}&=\int_{\bs{\Theta}}\int_{\mathcal{Y}}\log\left(\int_{\bs{\Theta}}e^{-\frac{1}{2}\left\lVert \bs{G}(\bs{\theta})+\bs{\Sigma}_{\bs{\varepsilon}}^{\frac{1}{2}}\bs{z}-\bs{G}(\bs{\theta}')\right\rVert_{\bs{\Sigma}_{\bs{\varepsilon}}^{-1}}^2}\di{}\bs{\theta}'\right)\frac{e^{-\frac{1}{2}\left\lVert \bs{z}\right\rVert^2}}{(2\pi)^{\frac{d_1-d_2}{2}}}\di{}\bs{z}\di{}\bs{\theta}\nonumber\\
    {}&=\int_{\bs{\Theta}}\int_{\mathcal{Y}}\log\left(\int_{\bs{\Theta}}e^{-\frac{1}{2}\left\lVert \bs{\Sigma}_{\bs{\varepsilon}}^{-\frac{1}{2}}\left(\bs{G}(\bs{\theta})-\bs{G}(\bs{\theta}')\right)+\bs{z}\right\rVert^2}\di{}\bs{\theta}'\right)\frac{e^{-\frac{1}{2}\left\lVert \bs{z}\right\rVert^2}}{(2\pi)^{\frac{d_1-d_2}{2}}}\di{}\bs{z}\di{}\bs{\theta}.
\end{align}
Next, we apply the importance sampling scheme, resulting in the following expression:
\begin{align}
    {}&\int_{\bs{\Theta}}\int_{\mathcal{Y}}\log\left(\int_{\bs{\Theta}}e^{-\frac{1}{2}\left\lVert \bs{\Sigma}_{\bs{\varepsilon}}^{-\frac{1}{2}}\left(\bs{G}(\bs{\theta})-\bs{G}(\bs{\theta}')\right)+\bs{z}\right\rVert^2}\di{}\bs{\theta}'\right)\frac{e^{-\frac{1}{2}\left\lVert \bs{z}\right\rVert^2}}{(2\pi)^{\frac{d_1-d_2}{2}}}\di{}\bs{z}\di{}\bs{\theta}\nonumber\\
    {}&=\int_{\bs{\Theta}}\int_{\mathcal{Y}}\log\left(\int_{\bs{\Theta}}e^{-\frac{1}{2}\left\lVert \bs{\Sigma}_{\bs{\varepsilon}}^{-\frac{1}{2}}\left(\bs{G}(\bs{\theta})-\bs{G}(\bs{\theta}')\right)+\bs{D}_{\bs{\gamma}}\bs{z}\right\rVert^2}\di{}\bs{\theta}'\right)\frac{e^{-\frac{1}{2}\left\lVert \bs{D}_{\bs{\gamma}}\bs{z}\right\rVert^2}}{(2\pi)^{\frac{d_1-d_2}{2}}}\di{}\bs{z}\di{}\bs{\theta}\prod_{i=1}^{d_1-d_2}\gamma_i,
\end{align}
where $\bs{D}_{\bs{\gamma}}\coloneqq\mathrm{diag}(\gamma_1,\ldots,\gamma_{d_1-d_2})\in\mathbb{R}^{(d_1-d_2)\times(d_1-d_2)}$. Applying the inverse CDF now results in the following integral
\begin{align}\label{eq:EIG.bh}
    {}&\int_{\bs{\Theta}}\int_{\mathcal{Y}}\log\left(\int_{\bs{\Theta}}e^{-\frac{1}{2}\left\lVert \bs{\Sigma}_{\bs{\varepsilon}}^{-\frac{1}{2}}\left(\bs{G}(\bs{\theta})-\bs{G}(\bs{\theta}')\right)+\bs{D}_{\bs{\gamma}}\bs{z}\right\rVert^2}\di{}\bs{\theta}'\right)\frac{e^{-\frac{1}{2}\left\lVert \bs{D}_{\bs{\gamma}}\bs{z}\right\rVert^2}}{(2\pi)^{\frac{d_1-d_2}{2}}}\di{}\bs{z}\di{}\bs{\theta}\prod_{i=1}^{d_1-d_2}\gamma_i\nonumber\\
    {}&=\int_{[0,1]^{d_1}}\log\left(\int_{[0,1]^{d_2}}e^{-\frac{1}{2}\left\lVert \bs{\Sigma}_{\bs{\varepsilon}}^{-\frac{1}{2}}\left(\bs{G}(\bs{y}_1)-\bs{G}(\bs{x})\right)+\bs{D}_{\bs{\gamma}}\bs{\Phi}^{-1}(\bs{y}_2)\right\rVert^2}\di{}\bs{x}\right)w_{\bs{\gamma}}(\bs{y}_2)\di{}\bs{y}\nonumber\\
    {}&=\int_{[0,1]^{d_1}}\log\left(\int_{[0,1]^{d_2}}e^{-\Lambda_{\bs{\gamma}}(\bs{y},\bs{x})}\di{}\bs{x}\right)w_{\bs{\gamma}}(\bs{y}_2)\di{}\bs{y},
\end{align}
where 
\begin{align}\label{def:w}
w_{\bs{\gamma}}(\bs{y}_2){}&\coloneqq \prod_{i=1}^{d_1-d_2}\gamma_ie^{-\frac{1}{2}\left\lVert \bs{D}_{\bs{\gamma}}\bs{\Phi}^{-1}(\bs{y}_2)\right\rVert^2+\frac{1}{2}\left\lVert \bs{\Phi}^{-1}(\bs{y}_2)\right\rVert^2}\nonumber\\
{}&=\prod_{i=1}^{d_1-d_2}\gamma_ie^{-\frac{1}{2}(\bs{\Phi}^{-1}(y_{2,i}))^2(\gamma_i^2-1)},
\end{align}
and
\begin{align}\label{def:Lambda.gamma}
    \Lambda_{\bs{\gamma}}(\bs{y},\bs{x}){}&\coloneqq\frac{1}{2}\left\lVert \bs{\Sigma}_{\bs{\varepsilon}}^{-\frac{1}{2}}\left(\bs{G}(\bs{y}_1)-\bs{G}(\bs{x})\right)+\bs{D}_{\gamma}\bs{\Phi}^{-1}(\bs{y}_2)\right\rVert^2.
\end{align}
   Let the composite integrand with importance sampling be defined as follows:
\begin{align}\label{def:importance.sampling.integrand}
    h_{\rm{IS}}(\bs{y}){}&\coloneqq \log\left(\int_{[0,1]^{d_2}}e^{-\Lambda_{\bs{\gamma}}(\bs{y},\bs{x})}\di{}\bs{x}\right)w_{\bs{\gamma}}(\bs{y}_2),
\end{align}
for $\bs{y}\in[0,1]^{d_1}$, where $\Lambda_{\bs{\gamma}}$ is as in~\eqref{def:Lambda.gamma} and $w_{\bs{\gamma}}$ is as in~\eqref{def:w}. The outer integrand in~\eqref{eq:EIG.bh} does not have singularities at the boundaries as long as $\gamma_i>1$ for $1\leq i\leq d_1-d_2$. However, the integrand becomes increasingly peaked inside the domain as $\gamma_i\to\infty$ for any $1\leq i\leq d_1-d_2$. This effect limits the size $\gamma_i$ that can be used in practice. Moreover, mixed derivatives of the integrand in~\eqref{eq:EIG.bh} might still display boundary singularities, implying that this integrand is not analytic. We state the following Theorem~\ref{thm:moment.EIG} and related results again in the general notation used in the previous section and later verify that this setting applies to the EIG case with importance sampling.

\subsection{\texorpdfstring{$g$}{inner integrand} analytic, \texorpdfstring{$f\circ I[g]$}{combined integrand} smooth}\label{sec:error.bounds.EIG}
\begin{Assumption}[Weighted compatibility of $f$ and $g$]\label{asu:weighted.compatibility}
   Let $g:[0,1]^{d_1}\times[0,1]^{d_2}\to\mathbb{R}$ and
      \begin{equation}
       g^{\ast}(\bs{y})\coloneqq \inf_{\bs{x}\in[0,1]^{d_2}}g(\bs{y},\bs{x})
   \end{equation}
   for $\bs{y}\in[0,1]^{d_1}$. Moreover, let
   \begin{equation}\label{def:Omega}
       \Omega_{>0}\coloneqq\left\{\bs{y}\in[0,1]^{d_1}:g^{\ast}(\bs{y})>0\right\}.
   \end{equation}
   Assume that $I[g(\bs{y},\cdot)]\leq 1$ for every $\bs{y}\in\Omega_{>0}\subseteq [0,1]^{d_1}$. 
   Moreover, let $f:(0,\infty)\to\mathbb{R}$ be such that there exists $L_{f}<\infty$ so that
   \begin{equation}
       \left|f(z)-f(z')\right|\leq L_{f}\left|\log(z)-\log(z')\right|
   \end{equation}
   for all $z,z'>0$. Finally, let $w:\Omega_{>0}\to(0,\infty)$ be such that there exists $0<\rho<1$ and $A_{\rho}<\infty$ so that
   \begin{equation}
       w(\bs{y})\leq A_{\rho}\left(g^{\ast}(\bs{y})\right)^{\rho}
   \end{equation}
   for all $\bs{y}\in\Omega_{>0}$.
\end{Assumption}
Assumption~\ref{asu:weighted.compatibility} is used in the proof of Theorem~\ref{thm:moment.EIG} below in place of Assumption~\ref{asu:Lipschitz.f}. We provide the following Lemma.

\begin{Lemma}[Weighted inner-error bound]\label{lem:weighted}
    Let Assumption~\ref{asu:weighted.compatibility} hold, let $\bs{y}\in\Omega_{>0}$ be fixed, and let $K\in\mathbb{R}$ and $0<b<1/2$ be such that
    \begin{equation}
        K\geq g^{\ast}(\bs{y})
    \end{equation}
    and
    \begin{equation}
        \left|K-I[g(\bs{y},\cdot)]\right|\leq b.
    \end{equation}
Then, there exists $C_{\rho}\coloneqq L_{f}A_{\rho}\max\{2^{\rho},3^{\rho}/(e\rho)\}<\infty$ such that
    \begin{equation}
        w(\bs{y})\left|f(K)-f(I[g(\bs{y},\cdot)])\right|\leq C_{\rho}b^{\rho}.
    \end{equation}
\end{Lemma}
For the proof, see Appendix~\ref{app:weighted}. We introduce the following notation. Let $f$, $g$, and $w$ be as in Assumption~\ref{asu:weighted.compatibility}, and let
\begin{equation}\label{def:weighted.integrand}
    h_{w}(\bs{y})\coloneqq f\left(I[g(\bs{y},\cdot)]\right)w(\bs{y})
\end{equation}
and
\begin{equation}
    \hat{h}_{w}(\bs{y})\coloneqq f\left(\hat{\mu}_M^{(\tilde{r})}[g(\bs{y},\cdot)]\right)w(\bs{y})
\end{equation}
for every $\bs{y}\in\Omega_{>0}$ with $\Omega_{>0}$ as in~\eqref{def:Omega} and let $h_w=\hat{h}_w=0$ on the complement of $\Omega_{>0}$. Then, the total weighted error is defined as follows:
\begin{align}\label{eq:total.error.is}
    \mathcal{E}_{{\rm{tot}},w}{}&\coloneqq\hat{\mu}_{N}^{(\tilde{s})}\left[\hat{h}_w\right] - I\left[h_w\right]\nonumber\\
    {}&= \underbrace{\hat{\mu}_{N}^{(\tilde{s})}\left[\hat{h}_w\right] - \hat{\mu}_{N}^{(\tilde{s})}\left[h_w\right]}_{\eqqcolon\mathcal{E}_{{\rm{in}},w}} + \underbrace{\hat{\mu}_{N}^{(\tilde{s})}\left[h_w\right] - I\left[h_w\right]}_{\eqqcolon\mathcal{E}_{{\rm{out}},w}}.
\end{align}

\begin{Assumption}[Regularity of the weighted outer integrand]\label{asu:outer}
Let $h_{w}$ be as in~\eqref{def:weighted.integrand}. For a multi-index $\bs{\beta}\in\mathbb{N}_0^{d_1}$, let $h_{w}^{(\bs{\beta})}$ denote the corresponding mixed partial derivative of $h_{w}$. Moreover, let $\alpha\in\mathbb{N}\cup\{0\}$. For every $\bs{\beta}\in\{0,1,\ldots,\alpha+1\}^{d_1}$, the mixed partial derivative $h_{w}^{(\bs{\beta})}$ exists on $\Omega_{>0}$ for $\Omega_{>0}$ as in~\eqref{def:Omega} and extends continuously to the complement of $\Omega_{>0}$.
\end{Assumption}
Assumption~\ref{asu:analytic.h} does not apply to the EIG integrand with importance sampling and Assumption~\ref{asu:outer} is used in the proof of Theorem~\ref{thm:moment.EIG} below instead. 

\begin{Lemma}\label{lem:total.error.eig} Let $g(\bs{y},\cdot):[0,1]^{d_2}\to\mathbb{R}$ satisfy Assumption~\ref{asu:uniform.analytic}. Let $c'$ be such that $c'<3\log(2)/\pi^2$ and let $M_1<\infty$, depending on $c'$, $d_2$, $A_g$, and $\eta_g$, be such that $M^{-c'\log(M)/d_2}<1/2$ for all $M\geq M_1$ and define the event $\mathcal{A}^{(s,n)}$ for $1\leq s\leq S$, $1\leq n\leq N$ such that
\begin{equation}\label{def:A.s.n}
    \mathcal{A}^{(s,n)}\coloneqq\left\{\left|\hat{\mu}_M^{(\tilde{r})}[g(\bs{y}^{(s,n)},\cdot)]-I[g(\bs{y}^{(s,n)},\cdot)]\right|\geq M^{-\frac{c'}{d_2}\log(M)}\right\}
\end{equation}
for all $M\geq M_1$. We have the following probability bound
\begin{equation}
\mathbb{P}\left(\mathcal{A}^{(s,n)}\Big|\bs{\tau}\right)\leq \left(\frac{4B_{d_2}}{\sqrt{\tilde{m}}}\right)^{\tilde{r}},
\end{equation}
where $\bs{\tau}$ is the outer scrambling. 
\end{Lemma}
\begin{proof}
The result follows from~\cite[Theorem 2.]{Pan24}.
\end{proof}

\begin{Theorem}[Root mean squared error for the weighted outer integrand]\label{thm:moment.EIG}
Let $f$, $g$, and $w$ satisfy the Assumption~\ref{asu:weighted.compatibility} for $L_f<\infty$, $0<\rho<1$, and $A_{\rho}<\infty$. Let the Assumptions of Lemma~\ref{lem:total.error.eig} hold for $c'<3\log(2)/\pi^2$. Let $h_w$ as in~\eqref{def:weighted.integrand} satisfy Assumption~\ref{asu:outer} for $\alpha\in\mathbb{N}\cup\{0\}$. Moreover, let $g(\bs{y},\bs{x})\leq 1$ for all $\bs{x}\in[0,1]^{d_2}$ and all $\bs{y}\in\Omega_{>0}$ and let $S$ and $N$ be such that $S=\Theta(\log_2(N))$. Then, for every $\epsilon>0$ there exists $C_{\epsilon},N_0<\infty$ depending on $\epsilon$, $d_1$, the norm of $h_w$ in the space of~\cite[Section 3.]{Pan25} or the Sobolev space of~\cite[Definition 4., Remark 2.]{goda2026quasi}, and $\alpha$, and for every $c<c'$ there exists $M_0<\infty$ depending on $c$, $d_2$, $A_g$, $\eta_g$, $L_{f}$, $\rho$, and $A_{\rho}$ such that
\begin{align}\label{eq:cor.eig}
  \mathbb{E}\left[\left|\mathcal{E}_{{\mathrm{tot}},w}\right|^2\right]^{\frac{1}{2}}\leq{}& C_{\epsilon}N^{-\alpha-\frac{3}{2}+\epsilon}+M^{-\frac{\rho c}{d_2}\log(M)}+\frac{L_{f}A_{\rho}}{e\rho}\left(S\left(\frac{4B_{d_2}}{\sqrt{\tilde{m}}}\right)^{\tilde r}\right)^{\frac{1}{2}}
\end{align}
 for all $N\geq N_0$, and $M\geq M_0$. The second and third terms in~\eqref{eq:cor.eig} tend to zero super-polynomially in $M$ provided that $R=\Theta((\tilde{m})^2)$.
\end{Theorem}
\begin{proof}
    It follows from Minkowski's inequality that
    \begin{equation}\label{eq:Minkowski.w}
         \mathbb{E}\left[\left|\mathcal{E}_{{\mathrm{tot}},w}\right|^2\right]^{\frac{1}{2}}\leq \mathbb{E}\left[\left|\mathcal{E}_{{\mathrm{in}},w}\right|^2\right]^{\frac{1}{2}}+\mathbb{E}\left[\left|\mathcal{E}_{{\mathrm{out}},w}\right|^2\right]^{\frac{1}{2}}.
    \end{equation}
    It follows from~\cite[Theorem 6.]{Pan25} or~\cite[Theorem 1.]{goda2026quasi} that there exists $C_{\epsilon},N_0<\infty$ such that
    \begin{equation}\label{eq:epsilon.out.w}
        \mathbb{E}\left[\left|\mathcal{E}_{{\mathrm{out}},w}\right|^2\right]^{\frac{1}{2}}\leq C_{\epsilon}N^{-\alpha-\frac{3}{2}+\epsilon}
    \end{equation}
    for all $N\geq N_0$. To bound the inner error, we proceed as in~\eqref{eq:max.sum}. It follows that
    \begin{align}
        \left|\mathcal{E}_{{\mathrm{in}},w}\right|{}&\leq \max_{s=1,\ldots,S}\frac{1}{N}\sum_{n=1}^{N}w\left(\bs{y}^{(s,n)}\right)\left|f\left(\hat{\mu}_{M}^{(\tilde{r})}[g(\bs{y}^{(s,n)},\cdot)]\right)-f\left(I[g(\bs{y}^{(s,n)},\cdot)]\right)\right|\nonumber\\
        {}&\leq \frac{1}{N}\sum_{n=1}^{N}\underbrace{\max_{s=1,\ldots,S}w\left(\bs{y}^{(s,n)}\right)\left|f\left(\hat{\mu}_{M}^{(\tilde{r})}[g(\bs{y}^{(s,n)},\cdot)]\right)-f\left(I[g(\bs{y}^{(s,n)},\cdot)]\right)\right|}_{\eqqcolon D^{(n)}}.
    \end{align}
    From Minkowski's inequality, it follows that
    \begin{align}
        \mathbb{E}\left[\left|\mathcal{E}_{{\mathrm{in}},w}\right|^2\right]^{\frac{1}{2}}{}&\leq \frac{1}{N}\sum_{n=1}^{N}\mathbb{E}\left[\left(D^{(n)}\right)^2\right]^{\frac{1}{2}}\nonumber\\
        {}&\leq \max_{n=1,\ldots,N}\mathbb{E}\left[\left(D^{(n)}\right)^2\right]^{\frac{1}{2}}.
    \end{align}
    For $1\leq n\leq N$ fixed, let $\mathcal{A}^{(s,n)}$ be as in~\eqref{def:A.s.n} and define
    \begin{equation}
        \mathcal{A}_n\coloneqq \bigcup_{s=1}^{S}\mathcal{A}^{(s,n)}.
    \end{equation}
    Lemma~\ref{lem:median.moment}\eqref{eq:med.lem.prob} and a union bound over the $S$ outer indices yield
    \begin{equation}\label{eq:union.bound.S}
        \mathbb{P}\left(\mathcal{A}_n\Big|\left(\bs{y}^{(s,n)}\right)_{s=1}^{S}\right)\leq S\left(\frac{4B_{d_2}}{\sqrt{\tilde{m}}}\right)^{\tilde{r}},
    \end{equation}
    where we used that the inner scramblings are independent across $1\leq r\leq R$. Next, we write
    \begin{align}\label{eq:exp.A.n.split}
        \mathbb{E}\left[\left(D^{(n)}\right)^2\Big|\left(\bs{y}^{(s,n)}\right)_{s=1}^{S}\right]={}&\mathbb{E}\left[\left(D^{(n)}\right)^2\mathbf{1}_{\mathcal{A}_n^{\rm{c}}}\Big|\left(\bs{y}^{(s,n)}\right)_{s=1}^{S}\right]\nonumber\\
        {}&+\mathbb{E}\left[\left(D^{(n)}\right)^2\mathbf{1}_{\mathcal{A}_n}\Big|\left(\bs{y}^{(s,n)}\right)_{s=1}^{S}\right].
    \end{align}
    From Lemma~\ref{lem:weighted}, it follows that there exists $M_1<\infty$ such that
    \begin{equation}
       \mathbb{E}\left[\left(D^{(n)}\right)^2\mathbf{1}_{\mathcal{A}_n^{\rm{c}}}\Big|\left(\bs{y}^{(s,n)}\right)_{s=1}^{S}\right]\leq C_{\rho}^2M^{-\frac{2\rho c'}{d_2}\log(M)} 
    \end{equation}
    for all $M\geq M_1$. Moreover, it follows that there exists $M_0<\infty$, $M_0\geq M_1$ such that
        \begin{equation}\label{eq:exp.A.n.split.1}
       \mathbb{E}\left[\left(D^{(n)}\right)^2\mathbf{1}_{\mathcal{A}_n^{\rm{c}}}\Big|\left(\bs{y}^{(s,n)}\right)_{s=1}^{S}\right]\leq M^{-\frac{2\rho c}{d_2}\log(M)}
    \end{equation}
    for all $M\geq M_0$. Next, we derive a bound for the second term in~\eqref{eq:exp.A.n.split}. For $\bs{y}^{(s,n)}\in\Omega_{>0}$ for $1\leq s\leq S$, $1\leq n\leq N$, it follows that
    \begin{equation}
        g^{\ast}(\bs{y}^{(s,n)})\leq \hat{\mu}_M^{(\tilde{r})}[g(\bs{y}^{(s,n)},\cdot)]\leq 1
    \end{equation}
    and also that
        \begin{equation}
        g^{\ast}(\bs{y}^{(s,n)})\leq I[g(\bs{y}^{(s,n)},\cdot)]\leq 1.
    \end{equation}
    From Assumption~\ref{asu:weighted.compatibility}, it follows that
       \begin{align} {}&w\left(\bs{y}^{(s,n)}\right)\left|f\left(\hat{\mu}_{M}^{(\tilde{r})}[g(\bs{y}^{(s,n)},\cdot)]\right)-f\left(I[g(\bs{y}^{(s,n)},\cdot)]\right)\right|\nonumber\\
       {}&\leq w\left(\bs{y}^{(s,n)}\right)L_{f}\left|\log\left(\hat{\mu}_{M}^{(\tilde{r})}[g(\bs{y}^{(s,n)},\cdot)]\right)-\log\left(I[g(\bs{y}^{(s,n)},\cdot)]\right)\right|\nonumber\\
       {}&=w\left(\bs{y}^{(s,n)}\right)L_{f}\left|\log\left(\frac{\hat{\mu}_{M}^{(\tilde{r})}[g(\bs{y}^{(s,n)},\cdot)]}{I[g(\bs{y}^{(s,n)},\cdot)]}\right)\right|\nonumber\\
       {}&\leq w\left(\bs{y}^{(s,n)}\right)L_{f}\log\left(\frac{1}{g^{\ast}(\bs{y}^{(s,n)})}\right)\nonumber\\
        {}&\leq L_{f}A_{\rho}\left(g^{\ast}(\bs{y}^{(s,n)})\right)^{\rho}\log\left(\frac{1}{g^{\ast}(\bs{y}^{(s,n)})}\right)\nonumber\\
        {}&\leq \frac{L_{f}A_{\rho}}{e\rho},
    \end{align} 
    independent of $\bs{y}^{(s,n)}$, where the last step follows from $\max_{u\in(0,1]}u^{\rho}\log(1/u)=1/(e\rho)$ by a similar derivation as in~\eqref{eq:u.log.1} to~\eqref{eq:u.log.2}. For $\bs{y}^{(s,n)}$ in the complement of $\Omega_{>0}$,
    \begin{equation}
        w\left(\bs{y}^{(s,n)}\right)\left|f\left(\hat{\mu}_{M}^{(\tilde{r})}[g(\bs{y}^{(s,n)},\cdot)]\right)-f\left(I[g(\bs{y}^{(s,n)},\cdot)]\right)\right|=0,
    \end{equation}
    and thus, it follows that
    \begin{equation}
        \left|\mathcal{E}_{{\mathrm{in}},w}\right|\leq\frac{L_{f}A_{\rho}}{e\rho}
    \end{equation}
    almost surely, independent of $1\leq s\leq S$, $1\leq n\leq N$, and thus also that
        \begin{equation}
        D^{(n)}\leq\frac{L_{f}A_{\rho}}{e\rho}
    \end{equation}
    almost surely. It follows from~\eqref{eq:union.bound.S} that
    \begin{align}\label{eq:exp.A.n.split.2}
        \mathbb{E}\left[\left(D^{(n)}\right)^2\mathbf{1}_{\mathcal{A}_n}\Big|\left(\bs{y}^{(s,n)}\right)_{s=1}^{S}\right]{}&\leq \left(\frac{L_{f}A_{\rho}}{e\rho}\right)^2\mathbb{P}\left(\mathcal{A}_n\Big|\left(\bs{y}^{(s,n)}\right)_{s=1}^{S}\right)\nonumber\\
        {}&\leq \left(\frac{L_{f}A_{\rho}}{e\rho}\right)^2S\left(\frac{4B_{d_2}}{\sqrt{\tilde{m}}}\right)^{\tilde{r}}.
    \end{align}
    Combining the results in~\eqref{eq:exp.A.n.split}, \eqref{eq:exp.A.n.split.1}, and~\eqref{eq:exp.A.n.split.2} yields
    \begin{equation}\label{eq:exp.Dn.y}
        \mathbb{E}\left[\left(D^{(n)}\right)^2\Big|\left(\bs{y}^{(s,n)}\right)_{s=1}^{S}\right]\leq M^{-\frac{2\rho c}{d_2}\log(M)}+\left(\frac{L_{f}A_{\rho}}{e\rho}\right)^2S\left(\frac{4B_{d_2}}{\sqrt{\tilde{m}}}\right)^{\tilde{r}}.
    \end{equation}
    The right-hand side of~\eqref{eq:exp.Dn.y} is independent of $(\bs{y}^{(s,n)})_{s=1}^{S}$ and independent of $n$; thus, it follows that
        \begin{equation}\label{eq:exp.Dn}
        \mathbb{E}\left[\left(D^{(n)}\right)^2\right]\leq M^{-\frac{2\rho c}{d_2}\log(M)}+\left(\frac{L_{f}A_{\rho}}{e\rho}\right)^2S\left(\frac{4B_{d_2}}{\sqrt{\tilde{m}}}\right)^{\tilde{r}}
    \end{equation}
    and moreover that
    \begin{equation}
        \mathbb{E}\left[\left|\mathcal{E}_{{\mathrm{in}},w}\right|^2\right]\leq M^{-\frac{2\rho c}{d_2}\log(M)}+\left(\frac{L_{f}A_{\rho}}{e\rho}\right)^2S\left(\frac{4B_{d_2}}{\sqrt{\tilde{m}}}\right)^{\tilde{r}}.
    \end{equation}
    From the subadditivity of the square root, it follows that
        \begin{equation}\label{eq:Minkowski.inner.w}
        \mathbb{E}\left[\left|\mathcal{E}_{{\mathrm{in}},w}\right|^2\right]^{\frac{1}{2}}\leq M^{-\frac{\rho c}{d_2}\log(M)}+\frac{L_{f}A_{\rho}}{e\rho}\left(S\left(\frac{4B_{d_2}}{\sqrt{\tilde{m}}}\right)^{\tilde{r}}\right)^{\frac{1}{2}}
    \end{equation}
for all $M\geq M_0$. Finally, combining the result in~\eqref{eq:Minkowski.w} with the results in~\eqref{eq:epsilon.out.w} and~\eqref{eq:Minkowski.inner.w} completes the proof.
\end{proof}

\begin{Lemma}[Weighted compatibility of the EIG integrands]\label{lem:weighted.compatibility.EIG}
    Let $\bs{G}$ be as in Assumption~\ref{asu:model}. Moreover, let $\bs{\gamma}=(\gamma_1,\ldots,\gamma_{d_1-d_2})$ with $\gamma_{i}>1$ for $1\leq i\leq d_1-d_2$ and set
    \begin{equation}
        U\coloneqq\sup_{\bs{x},\bs{y}_1\in[0,1]^{d_2}}\left\lVert \bs{\Sigma}_{\bs{\varepsilon}}^{-\frac{1}{2}}(\bs{G}(\bs{y}_1)-\bs{G}(\bs{x}))\right\rVert.
    \end{equation}
    It follows from Assumption~\ref{asu:model} that $U<\infty$. Let $g=e^{-\Lambda_{\bs{\gamma}}}$ with $\Lambda_{\bs{\gamma}}$ as in~\eqref{def:Lambda.gamma} and $w_{\bs{\gamma}}:(0,1)^{d_1-d_2}\to(0,\infty)$ be as in~\eqref{def:w}. Moreover, let $\gamma_{\rm{min}}\coloneqq\min_i\gamma_i$, $\gamma_{\rm{max}}\coloneqq\max_i\gamma_i$ and
    \begin{equation}
        \rho^{\ast}\coloneqq \frac{\gamma_{\rm{min}}^2-1}{\gamma_{\rm{max}}^2}
    \end{equation}
    and thus $0<\rho^{\ast}<1$. Then, for every $0<\rho<\rho^{\ast}$, it follows that there exists $A_{\rho}<\infty$ such that
    \begin{equation}
        w(\bs{y}_2)\leq A_{\rho}(g^{\ast}(\bs{y}))^{\rho}
    \end{equation}
    for all $\bs{y}_2\in(0,1)^{d_1-d_2}$.
\end{Lemma}
For the proof, see Appendix~\ref{app:weighted.EIG}. It follows from Lemma~\ref{lem:weighted.compatibility.EIG} that $h_{\rm{IS}}$ satisfies Assumption~\ref{asu:weighted.compatibility} with $w(\bs{y})=w_{\bs{\gamma}}(\bs{y}_2)$ and $\Omega_{>0}=[0,1]^{d_2}\times(0,1)^{d_1-d_2}$.

\begin{Lemma}[Regularity of the weighted EIG integrand]\label{lem:mixed.derivatives.is}
Let $h_{\rm{IS}}$  be as in~\eqref{def:importance.sampling.integrand}. Moreover, let $\alpha\in\mathbb{N}\cup\{0\}$. Given Assumption~\ref{asu:model}, if $\gamma_i>\sqrt{\alpha+2}$ for every $1\leq i\leq d_1-d_2$, then for every $\bs{\beta}\in\{0,1,\ldots,\alpha+1\}^{d_1}$, the mixed partial derivative $h_{\rm{IS}}^{(\bs{\beta})}$ exists on the open cube $(0,1)^{d_1}$ and extends continuously to $[0,1]^{d_1}$.
\end{Lemma}
For the proof, see Appendix~\ref{app:der.int.is}. It follows from Lemma~\ref{lem:mixed.derivatives.is} that $h_{\rm{IS}}$ satisfies Assumption~\ref{asu:outer} with parameter $\alpha$ and thus the conditions stated in Lemma~\ref{lem:total.error.eig} are satisfied for the EIG application with uniform prior distribution for $\alpha\in\mathbb{N}\cup\{0\}$ as long as $\gamma_i>\sqrt{\alpha+2}$ for $1\leq i\leq d_1-d_2$. In principle, this importance sampling scheme enables a convergence rate of $\alpha+3/2-\epsilon$ for any $\alpha>0$ and any $\epsilon>0$ in the DLmRQMC estimator. In practice, as $\alpha\to\infty$, this rate is only observed for a very large number of Sobol' points $N$, and numerical issues occur. Another issue is the finite precision used in the scrambling method, see~\cite[Lemma 2.1.]{Pan23}, which we do not address in this study.

\begin{Remark}[Cotangent transform]
    Another method of mapping $\mathbb{R}^{d_1-d_2}$ to $[0,1]^{d_1-d_2}$ is the cotangent transform. This transform could be used instead of the importance sampling scheme introduced in this work. However, our numerical results did not reveal a clear advantage of either method over the other, and thus we focus on the importance sampling scheme for concision.
\end{Remark}

\section{Numerical Results}
\label{sec:numex}
\subsection{Example 1: Polynomial example}
We begin with a simple polynomial example, where
\begin{equation}\label{eq:pol.f}
    f(z)\coloneqq z^2+z, \quad \text{where } z\coloneqq I[g]
\end{equation}
and
\begin{equation}\label{eq:pol.g}
    g(\bs{y},\bs{x})\coloneqq \sum_{i=1}^{d}\left[y_i + x_i +  y_ix_i\right], \quad \text{where } \bs{y},\bs{x}\in[0,1]^{d}.
\end{equation}
This example was used in~\cite{Bar24} to study the effects of the integrand dimension on the convergence behavior of the DLRQMC estimator.
It follows immediately that $g$ and the composition $f\circ I[g]$ are analytic. Moreover, $f$ is Lipschitz on the image of $I[g]$ and the image of $\hat{\mu}_M^{(\tilde{r})}[g(\bs{y},\cdot)]$ for all $\bs{y}\in[0,1]^{d_1}$ for all $M\geq 1$. We first analyze the inner and outer error as defined in~\eqref{eq:total.error} separately and then move to the total error of the nested estimator. To compare the RQMC and the mRQMC estimators for the inner integral, we fix $\bs{y}$ at the arbitrary value $\bs{\bar{y}}\coloneqq(1/2,\ldots,1/2)^{\scriptscriptstyle\mskip-1mu\top\mskip-2mu}\in\mathbb{R}^{d}$.
Figure~\ref{fig:pol} (A) shows that the RMSE over $J=100$ independent replicates of the mRQMC estimator  displays behavior consistent with the super-polynomial convergence rate for dimension $d=10$, while for the RQMC estimator, the RMSE converges at rate $3/2$, as expected. We present the RMSE along with the largest and smallest absolute error of the $J$ replicates as well as the 10th and 90th percentiles as a box plot. The median (50th percentile) of the $J$ replicates is also depicted, although we note that this is not related to the median over the $R$ scramblings of the Sobol' points, which is the main focus of this study. Here, we chose $M=2^{\tilde{m}}$ and $R=2\tilde{m}-1$, where $4\leq\tilde{m}\leq 9$. This linear dependence of $R$ on $\tilde{m}$ was sufficient for this example, see~\cite[Remark 2]{Pan24} for a discussion. The same behavior is also observed for $d=100$, although only for larger $M$ (Figure~\ref{fig:pol} (D)).

To compare the RQMC and the mRQMC estimators for the outer integral, we chose $N=2^{\tilde{n}}$ and $S=2\tilde{n}-1$, where $8\leq\tilde{n}\leq 14$. Panel (B) shows that the RMSE of the mRQMC estimator again displays behavior consistent with the super-polynomial convergence rate, whereas for the RQMC estimator, the RMSE converges at rate $3/2$. The same holds for $d=100$ (Panel E) for $11\leq \tilde{n}\leq 19$.

Finally, to compare the DLRQMC and the DLmRQMC estimators for the nested integral, we present the RMSE
as a function of $K=NM$ samples for $d=10$ in Panel (C). Here, we choose $9\leq \tilde{n}\leq 14$ and $\tilde{m}=\tilde{n}-5$. The RMSE of the mRQMC estimator again displays results consistent with the super-polynomial convergence rate, whereas for the RQMC estimator, the RMSE converges at rate $3/4$. The primary motivation for replacing the mean with the median in the RQMC estimator is that the median removes large outliers in the RMSE. In Figure~\ref{fig:pol.ccdf} we demonstrate that increasing the number of Sobol' points leads to more extreme outliers for both the inner loop (Panel (A)) and the outer loop (Panel (B)). For this figure, we used $J=2000$ replicates of the outer and inner RQMC estimators with $R=S=1$ and evaluated the complementary CDF (CCDF) for each replicate. We computed the standardized absolute errors for the inner estimator as follows:
\begin{equation}
    \hat{I}_{{\mathrm{SAE},\mathrm{inner}}}^{(j)}\coloneqq\left|\frac{\frac{1}{M}\sum_{m=1}^M g(\bs{\bar{y}},\bs{x}^{(j,m)})-I[g(\bs{\bar{y}},\cdot)]}{\tilde{\sigma}_{\mathrm{inner}}}\right|, \quad 1\leq j\leq J,
\end{equation}
where $\tilde{\sigma}_{\mathrm{inner}}$ is the sample standard deviation over the $J=2000$ replicates. Moreover, we computed the standardized absolute errors for the outer estimator as follows:
\begin{equation}
    \hat{I}_{{\mathrm{SAE},\mathrm{outer}}}^{(j)}\coloneqq\left|\frac{\frac{1}{N}\sum_{n=1}^N f\left(I\left[g(\bs{y}^{(j,n)},\cdot)\right]\right)-I\left[f\left(I[g]\right)\right]}{\tilde{\sigma}_{\mathrm{outer}}}\right|, \quad 1\leq j\leq J.
\end{equation}
We present the probability of a standard normal random variable having outliers of a given size $t$ as a reference, that is, two times the survival function of the standard normal. We observe that increasing $M$ and $N$, respectively, leads to fewer outliers; however, these rare outliers increase dramatically in size as $M$ and $N$ grow.

Although the inner, and in particular, the outer, mRQMC estimators are affected by the dimension $d$, this effect is not as dramatic as the curse of dimensionality in the error bounds in Lemma~\ref{lem:total.error} would indicate. The polynomial example described by~\eqref{eq:pol.f} and~\eqref{eq:pol.g} exhibits a very simple additive structure. A much stronger effect of the integrand dimension was observed in our next example from BOED where no such simple structure exists.

\begin{figure}[ht]
	\subfloat[Inner loop $d=10$]{%
		\includegraphics[width=0.43\textwidth]{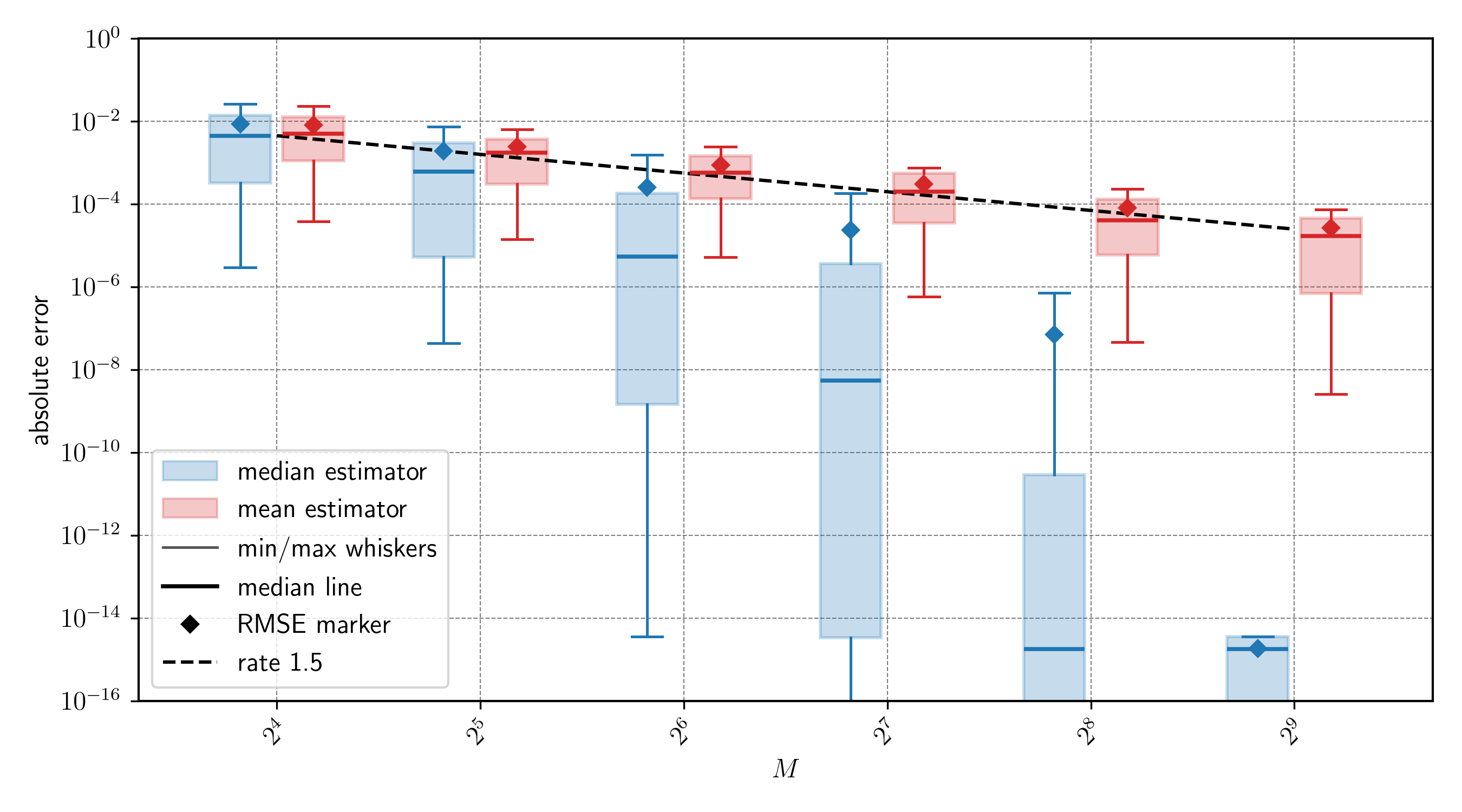}
	}
	\subfloat[Outer loop $d=10$]{%
		\includegraphics[width=0.43\textwidth]{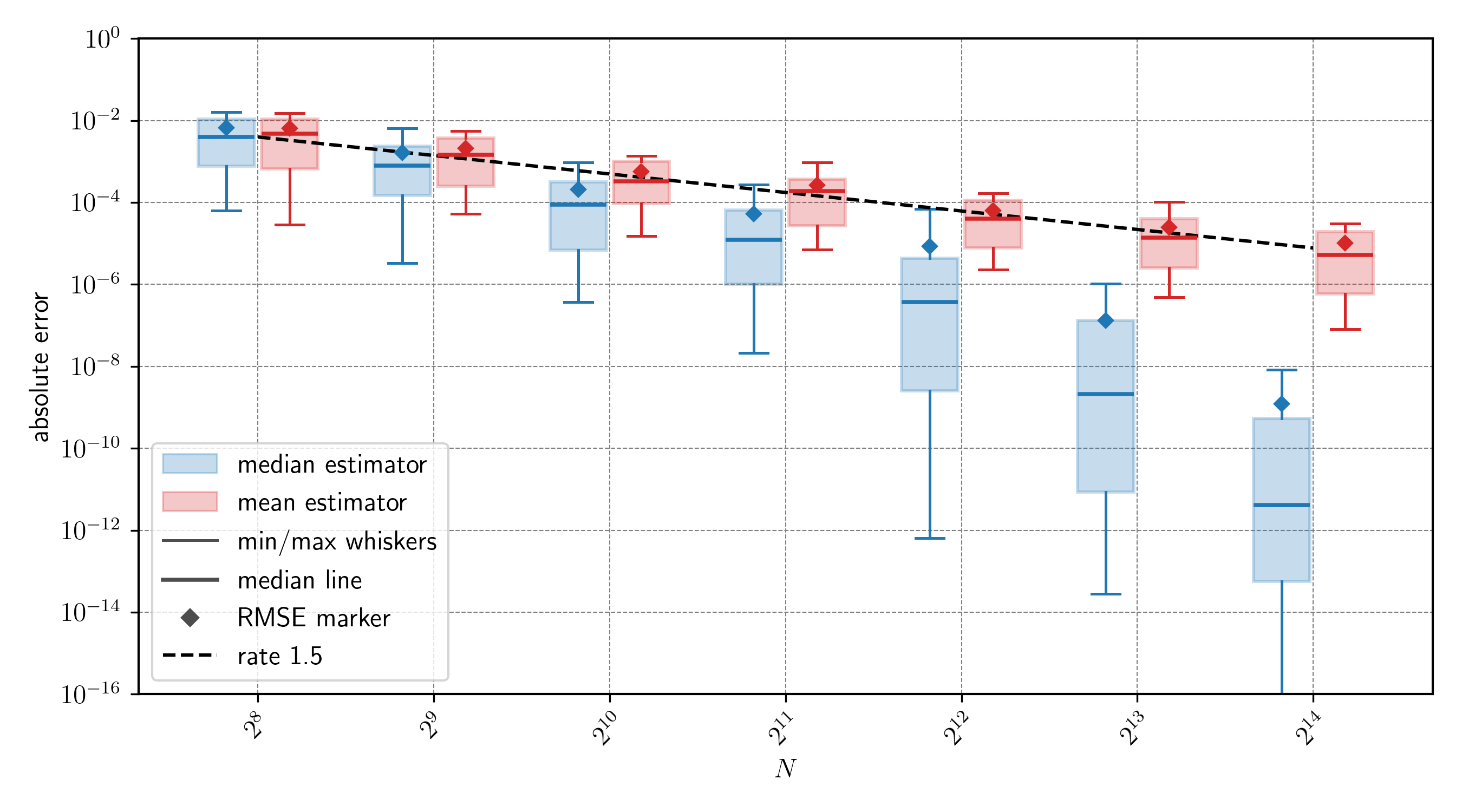}
	}\\
    	\subfloat[Double-loop $d=10$]{%
		\includegraphics[width=0.43\textwidth]{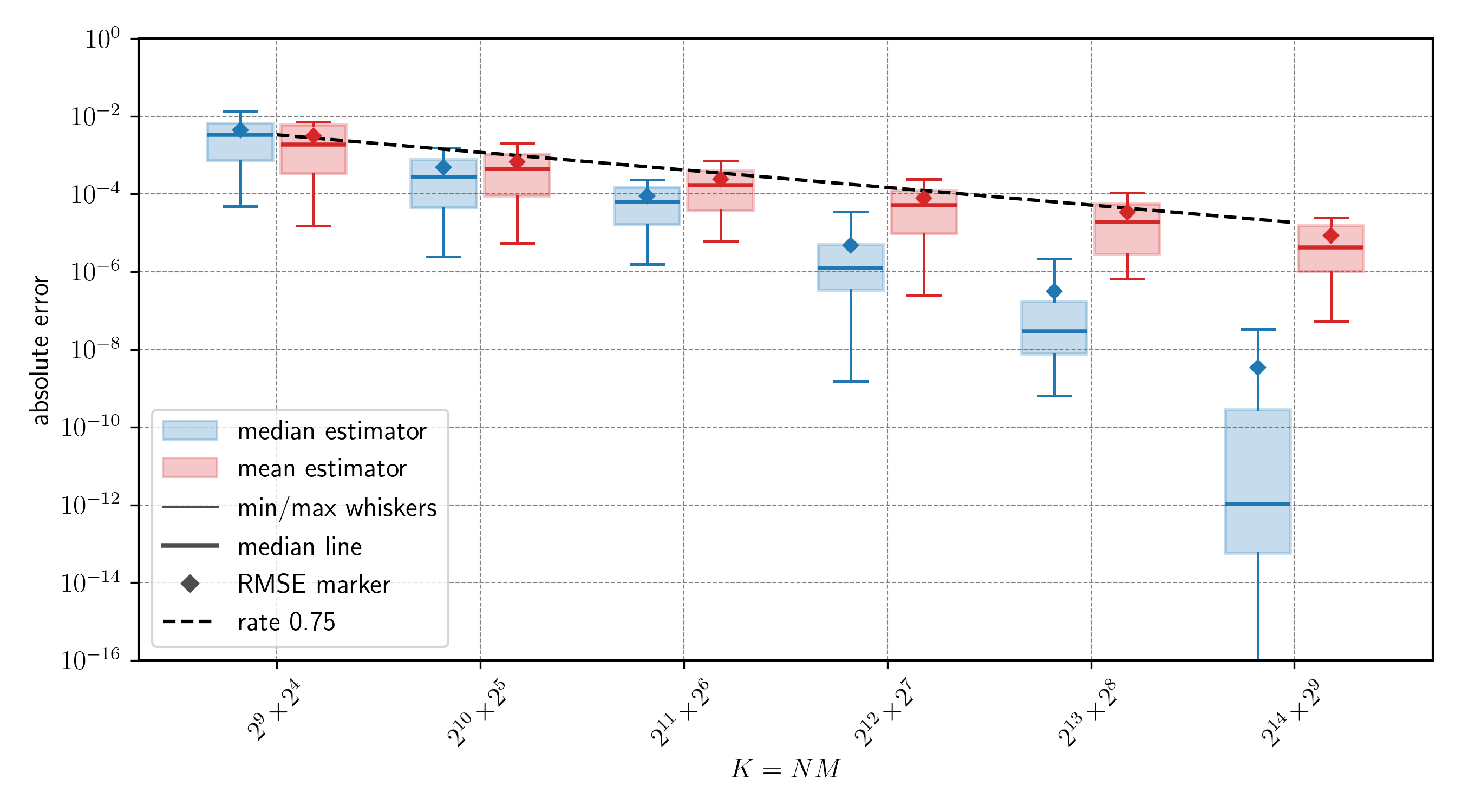}
	}\\
    	\subfloat[Inner loop $d=100$]{%
		\includegraphics[width=0.43\textwidth]{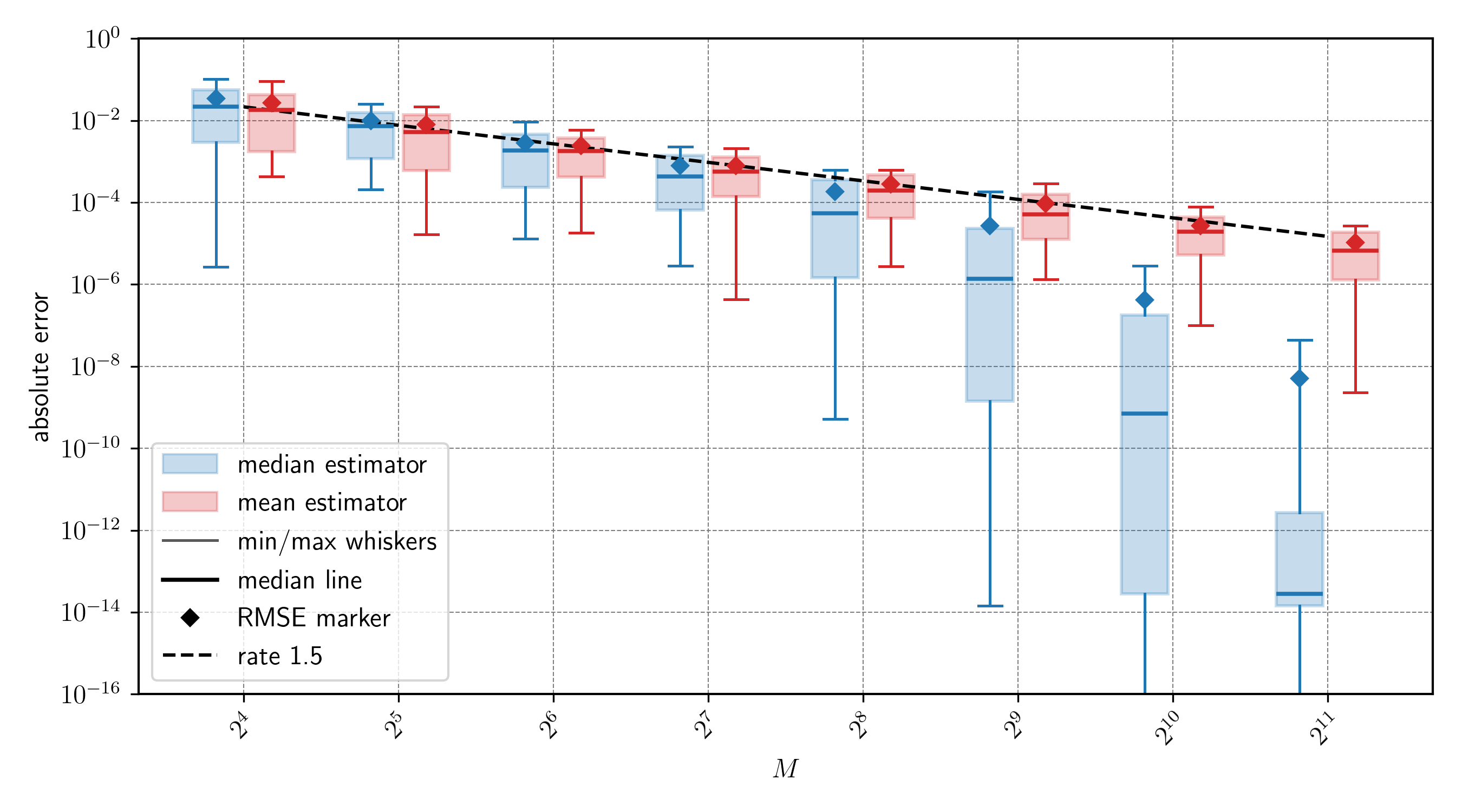}
	}
	\subfloat[Outer loop $d=100$]{%
		\includegraphics[width=0.43\textwidth]{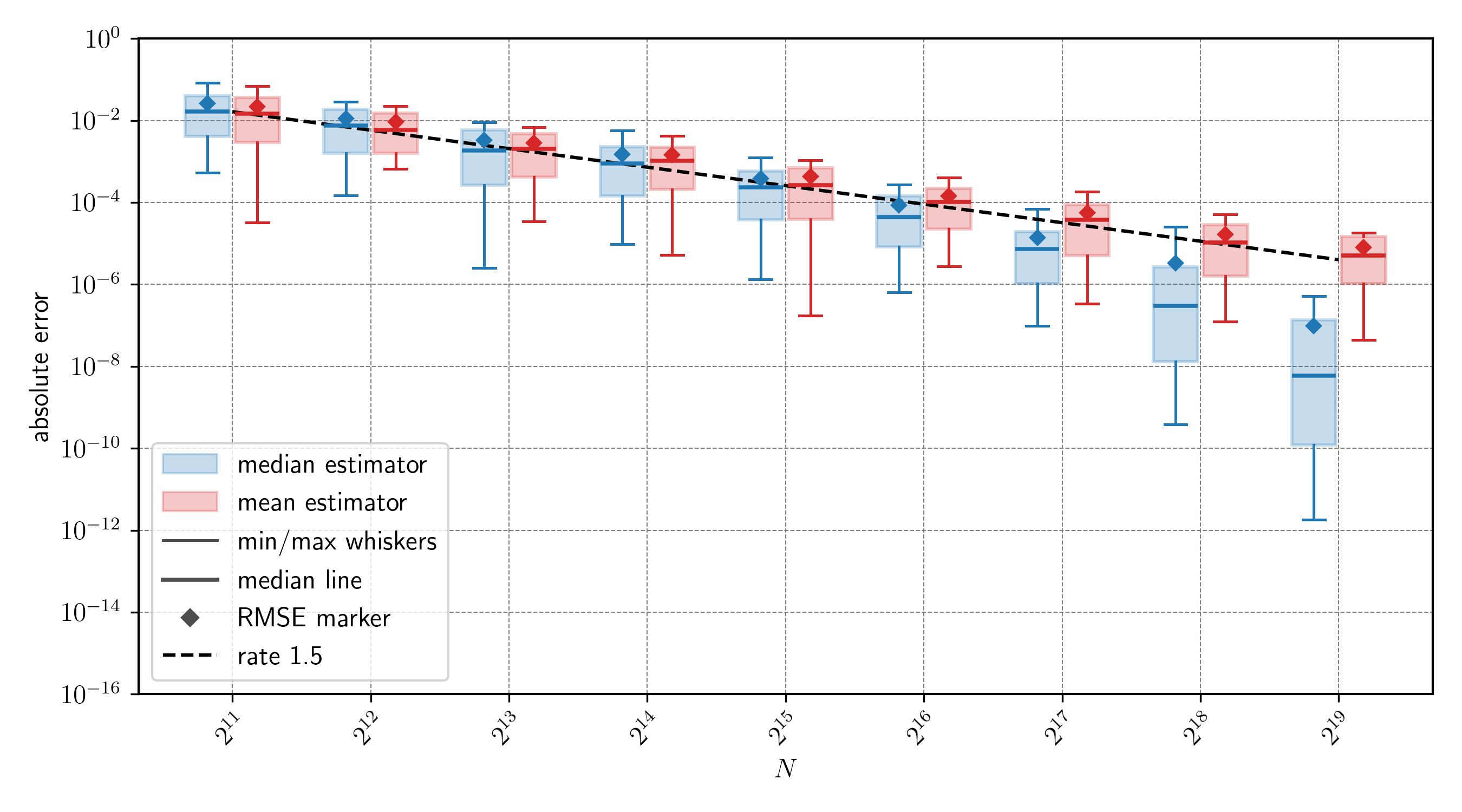}
	}
	\caption{Polynomial Example: Boxplot of $J=100$ replicates of the median RQMC and the standard RQMC estimator in dimension 10 for the inner loop  with fixed outer sample $\bs{\bar{y}}\coloneqq(1/2,\ldots,1/2)^{\scriptscriptstyle\mskip-1mu\top\mskip-2mu}$ (A), for the outer loop with exact inner integration (B), and the double-loop estimator (C). The RMSE of the RQMC estimator converges at rate 1.5 as a function of the inner number of Sobol' points $M$ (A) and as a function of the outer number of Sobol' points $N$ (B). For the double-loop estimator, the RMSE converges at rate 0.75 as a function of the total number of Sobol' points $K=NM$ (C). The RMSE of the median RQMC estimator displays behavior consistent with super-polynomial convergence for all three cases. This behavior emerges for larger $M$ (D) and $N$ (E), respectively, in dimension 100.}
    \label{fig:pol}
\end{figure}

\begin{figure}[ht]
	\subfloat[Inner loop]{%
		\includegraphics[width=0.4\textwidth]{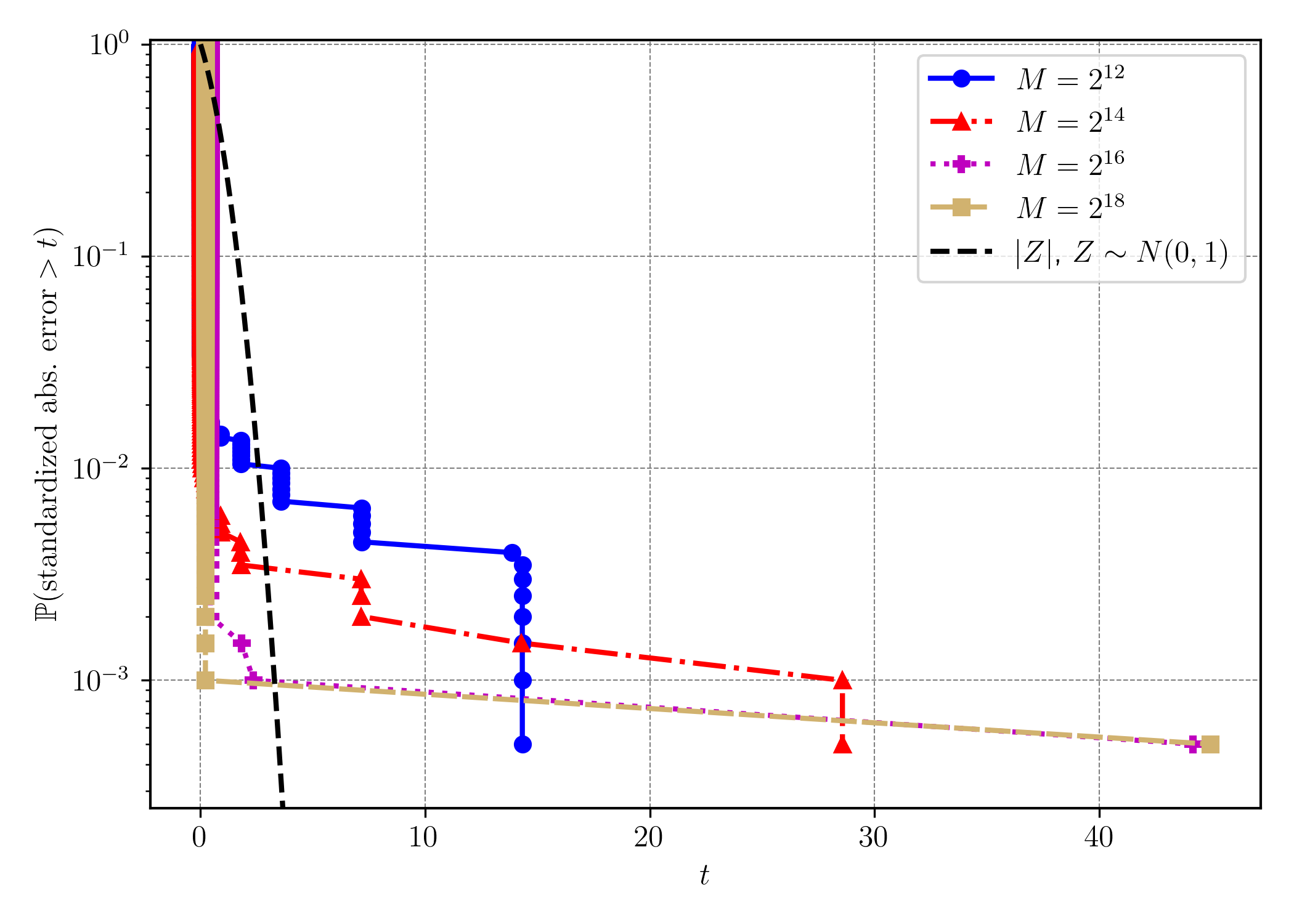}
	}
	\subfloat[Outer loop]{%
		\includegraphics[width=0.4\textwidth]{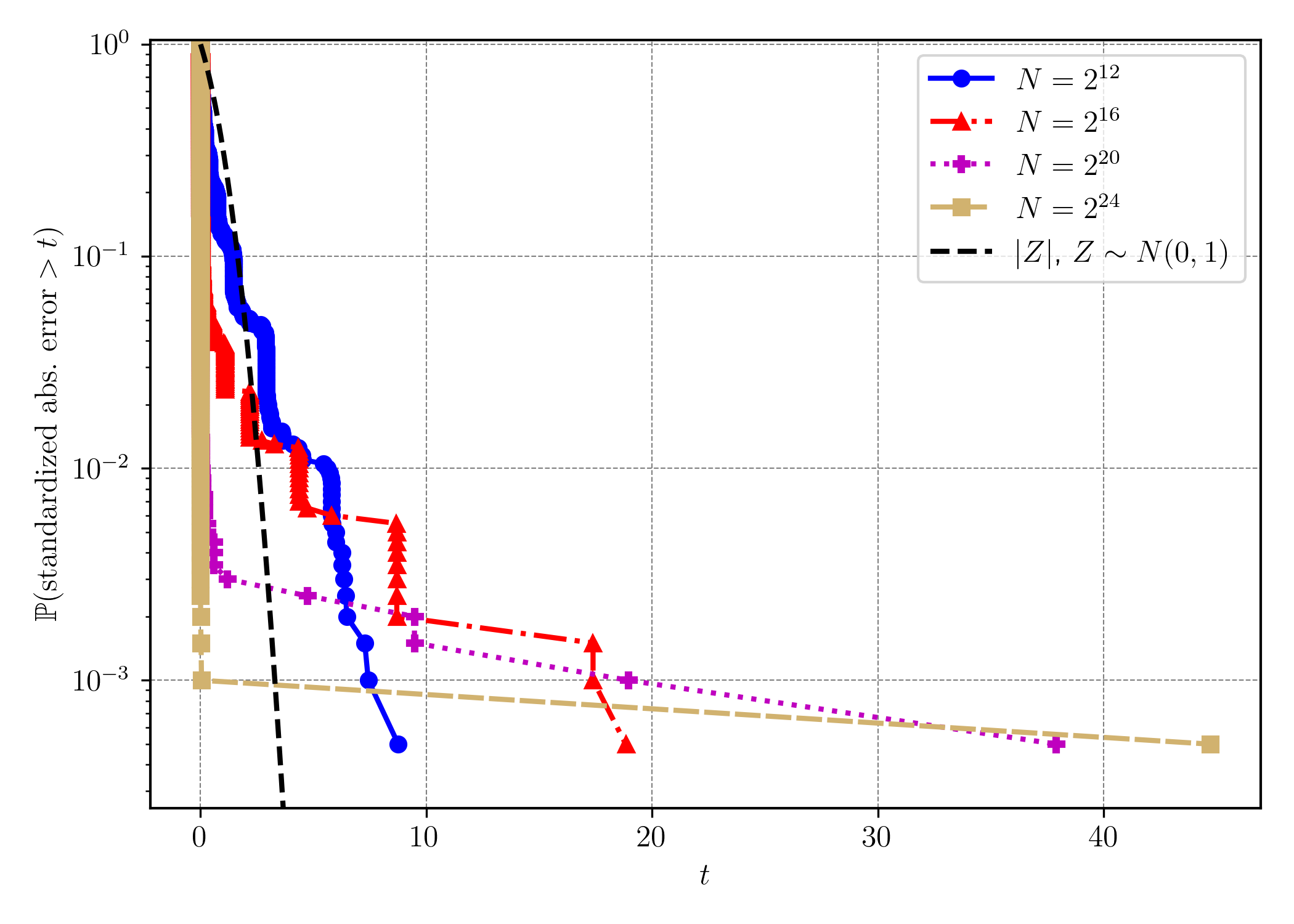}
	}
	\caption{Polynomial example: CCDF of the standardized absolute errors in dimension 10 over $J=2000$ randomizations for the inner loop  with fixed outer sample $\bs{\bar{y}}\coloneqq(1/2,\ldots,1/2)^{\scriptscriptstyle\mskip-1mu\top\mskip-2mu}$ (A) and for the outer loop with exact inner integration (B). The distribution of standardized absolute errors produces increasingly large outliers as the number of Sobol' points in the inner loop $M$ and in the outer loop $N$ increase. The standardized absolute errors of the standard normal are presented as a reference.}
    \label{fig:pol.ccdf}
\end{figure}
\FloatBarrier

\subsection{Example 2: EIG example}
We select a simple example adapted from~\cite{Hua13} as follows:
\begin{equation}\label{def:forward.model.hm}
    G_i(\bs{\theta},\xi_i)=\frac{1}{d_{\theta}}\sum_{j=1}^{d_{\theta}}\left[\theta_j^3\xi_i^2+\theta_je^{-|a_j-\xi_i|}\right], \quad 1\leq i\leq d_y,
\end{equation}
where $(a_1,\ldots,a_{d_{\theta}})\in\mathbb{R}^{d_{\theta}}$ is fixed and $(\xi_1,\ldots,\xi_{d_{y}})\in\mathbb{R}^{d_{y}}$ is the experiment design which we also consider fixed for our purposes. Moreover, the parameters of interest $\bs{\theta}$ are sampled from a uniform distribution, that is, $\bs{\theta}\sim\mathcal{U}([0,1]^{d_{\theta}})$ and we choose $\bs{\Sigma}_{\bs{\varepsilon}}=\sigma_{\varepsilon}^2\bs{I}_{(d_1-d_2)\times (d_1-d_2)}$, where 
$\sigma_{\varepsilon}^2=0.01$ 
and $\bs{I}_{(d_1-d_2)\times (d_1-d_2)}$ is the identity matrix in dimension $d_1-d_2$.  Alternatively, we write~\eqref{def:forward.model.hm} as
\begin{equation}\label{def:forward.model.hm.uhc}
    G_i(\bs{z},\xi_i)=\frac{1}{d_{2}}\sum_{j=1}^{d_{2}}\left[z_j^3\xi_i^2+z_je^{-|a_j-\xi_i|}\right], \quad 1\leq i\leq d_1-d_2,
\end{equation}
and set
\begin{equation}\label{def:forward.model.hm.uhc.vec}
    \bs{G}(\bs{z})\equiv\bs{G}(\bs{z},\bs{\xi})\coloneqq\begin{pmatrix}
        G_1(\bs{z},\xi_1)\\\vdots\\G_{d_1-d_2}(\bs{z},\xi_{d_1-d_2})
    \end{pmatrix}.
\end{equation}
The function defined in~\eqref{def:forward.model.hm.uhc.vec} trivially satisfies Assumption~\ref{asu:model} for fixed $(a_1,\ldots,a_{d_{\theta}})\in\mathbb{R}^{d_{\theta}}$ and $(\xi_1\ldots,\xi_{d_{y}})\in\mathbb{R}^{d_{y}}$.
For this example, we choose the parameter dimension $d_{\theta}=d_2\in\{1,2,3,4\}$ and the data dimension $d_{y}=d_1-d_2=2$; thus, $d_1\in\{3,4,5,6\}$. Figure~\ref{fig:EIG} presents the inner RMSE with the outer samples fixed at $\bs{\bar{y}}=(1/2,\ldots,1/2)^{\scriptscriptstyle\mskip-1mu\top\mskip-2mu}\in[0,1]^{d_1}$ against a high-fidelity reference solution computed with $R=51$ and $M=2^{26}$. We see that the mRQMC estimator achieved convergence rates better than 1.5 in dimensions one and two (Figure~\ref{fig:EIG} Panel (A) and (B)), whereas the RQMC estimator is limited to the rate 1.5, consistent with our theoretical findings. In dimensions three (Panel (C)) and four (Panel (D)), this effect is less pronounced. 

\begin{figure}[ht]
	\subfloat[$d_{2}=1$]{%
		\includegraphics[width=0.42\textwidth]{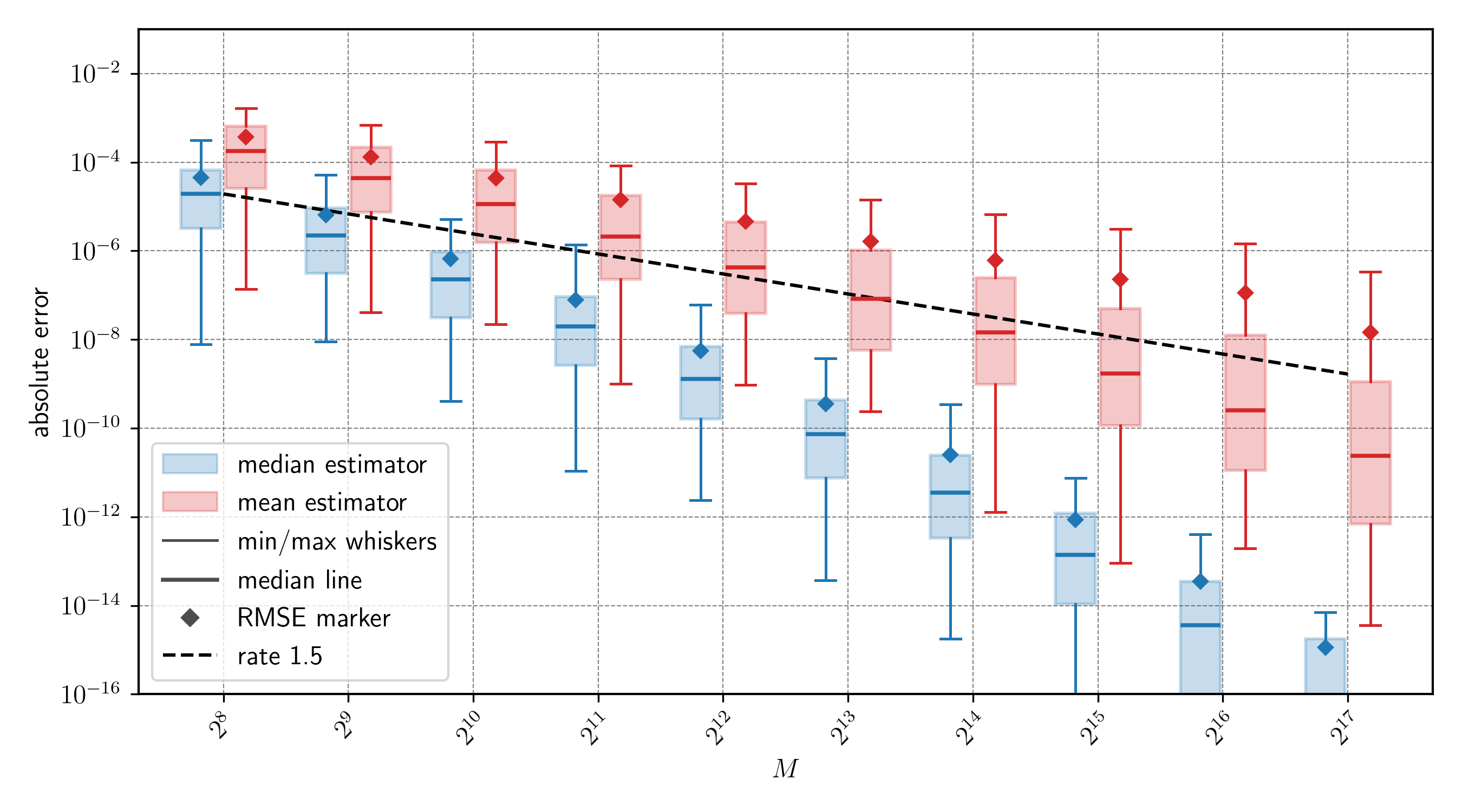}
	}
    \subfloat[$d_{2}=2$]{%
		\includegraphics[width=0.42\textwidth]{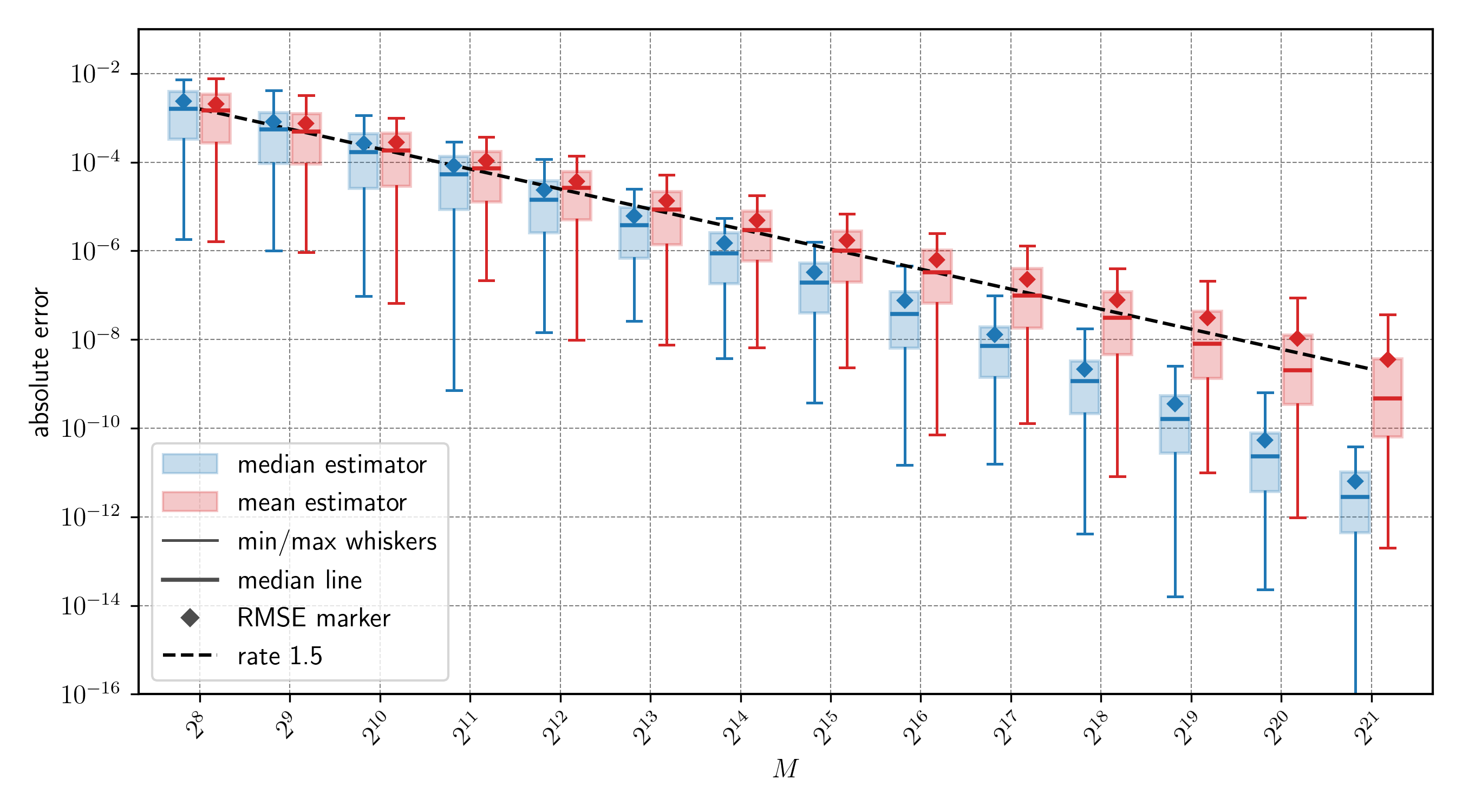}
	}\\
	\subfloat[$d_{2}=3$]{%
		\includegraphics[width=0.42\textwidth]{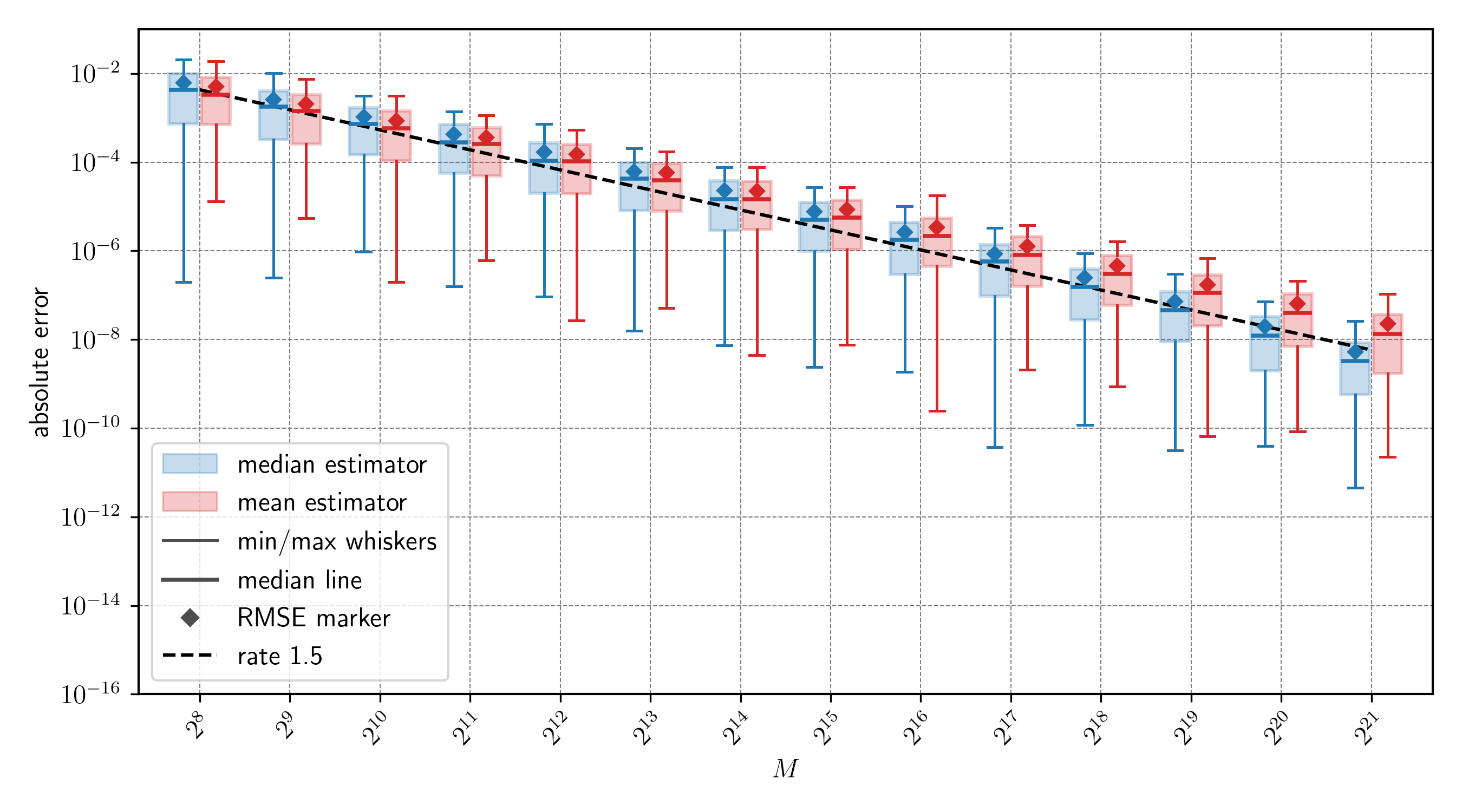}
	}
    \subfloat[$d_{2}=4$]{%
		\includegraphics[width=0.42\textwidth]{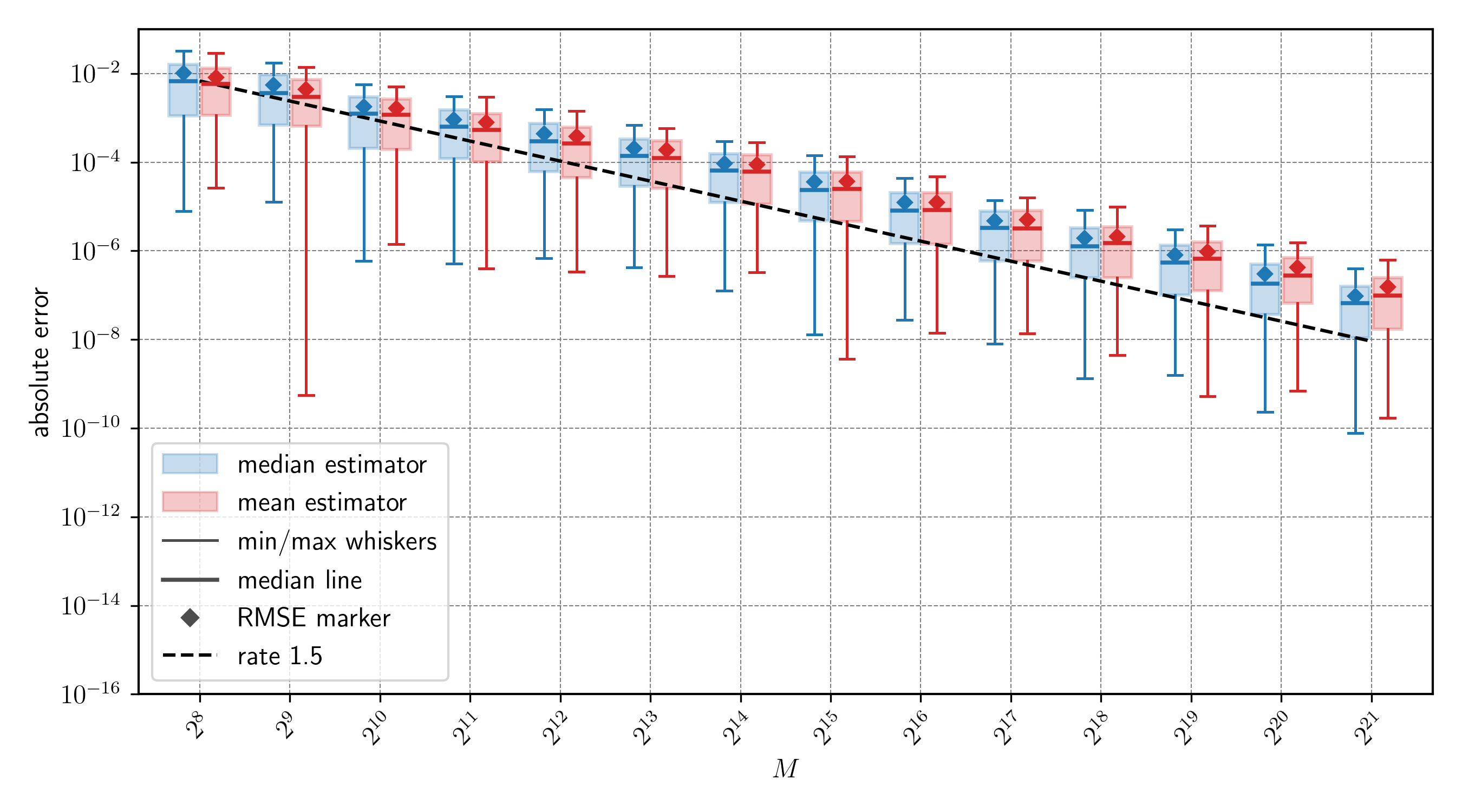}
	}
	\caption{EIG Example: Inner integration error for $J=1000$ replicates of the DLRQMC and the DLmRQMC estimators in various dimensions. The mRQMC estimator achieves convergence rates better than rate 1.5 in one and two dimensions (Panels (A) and (B)). In three dimensions (Panel (C)), the DLmRQMC estimator begins to achieve a convergence better than 1.5 for large sample sizes. In four dimensions (Panel (D)), the DLmRQMC estimator error is only slightly smaller than the DLRQMC estimator error.}
    \label{fig:EIG}
\end{figure}

In Figure~\ref{fig:EIG.outer} we present the outer RMSE against a high-fidelity reference solution computed with $J=1000$, $S=49$ and $N=2^{25}$, where the inner integral is computed using $R=1$ and $M=64$ scrambled Sobol' points that are fixed across all runs to isolate the outer error. The outer variables were obtained using the inverse CDF (left column) and the Gaussian importance sampling scheme (right column). Without importance sampling, both the DLRQMC and the DLmRQMC methods only achieve a convergence rate of one. This is because the composite integrand $f\circ I[g]$ in this case has a boundary singularity and thus displays infinite Hardy--Krause variation. Note that this convergence rate is still better than the rate $0.5$ achieved by DLMC methods. If we use the inverse CDF with importance sampling as described in Section~\ref{sec:eig}, the convergence rate of the DLRQMC method improves to $1.5$ and the DLmRQMC estimator achieves an even better convergence rate in small dimensions. Here we selected an importance sampling parameter $\gamma=(2,2)^{\scriptscriptstyle\mskip-1mu\top\mskip-2mu}$. As the dimension increases, both estimators were limited to rate one, indicating preasymptotic behavior.

\begin{figure}[ht]
	\subfloat[$d_{1}=3$]{%
		\includegraphics[width=0.42\textwidth]{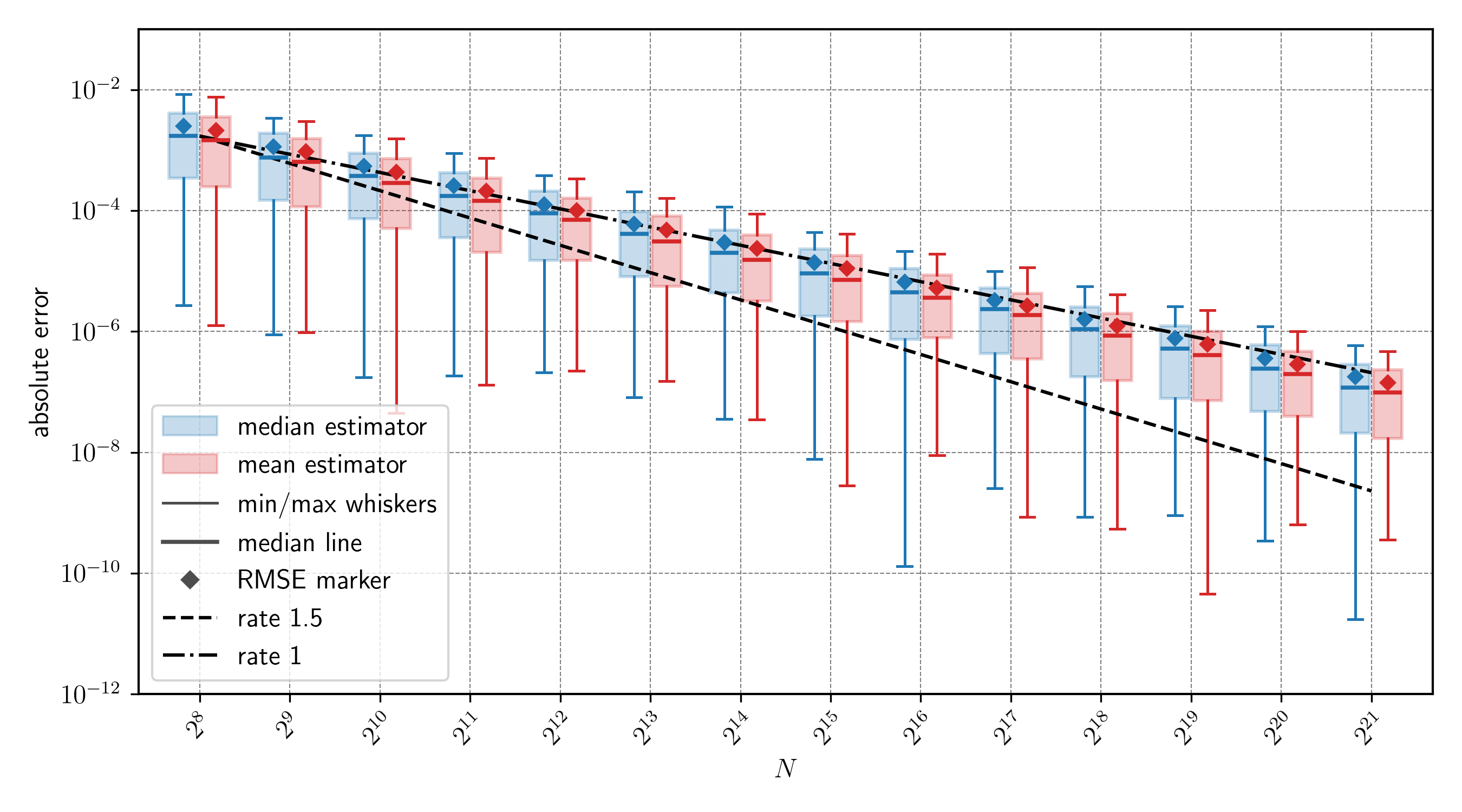}
	}
    	\subfloat[$d_{1}=3$ with IS]{%
		\includegraphics[width=0.42\textwidth]{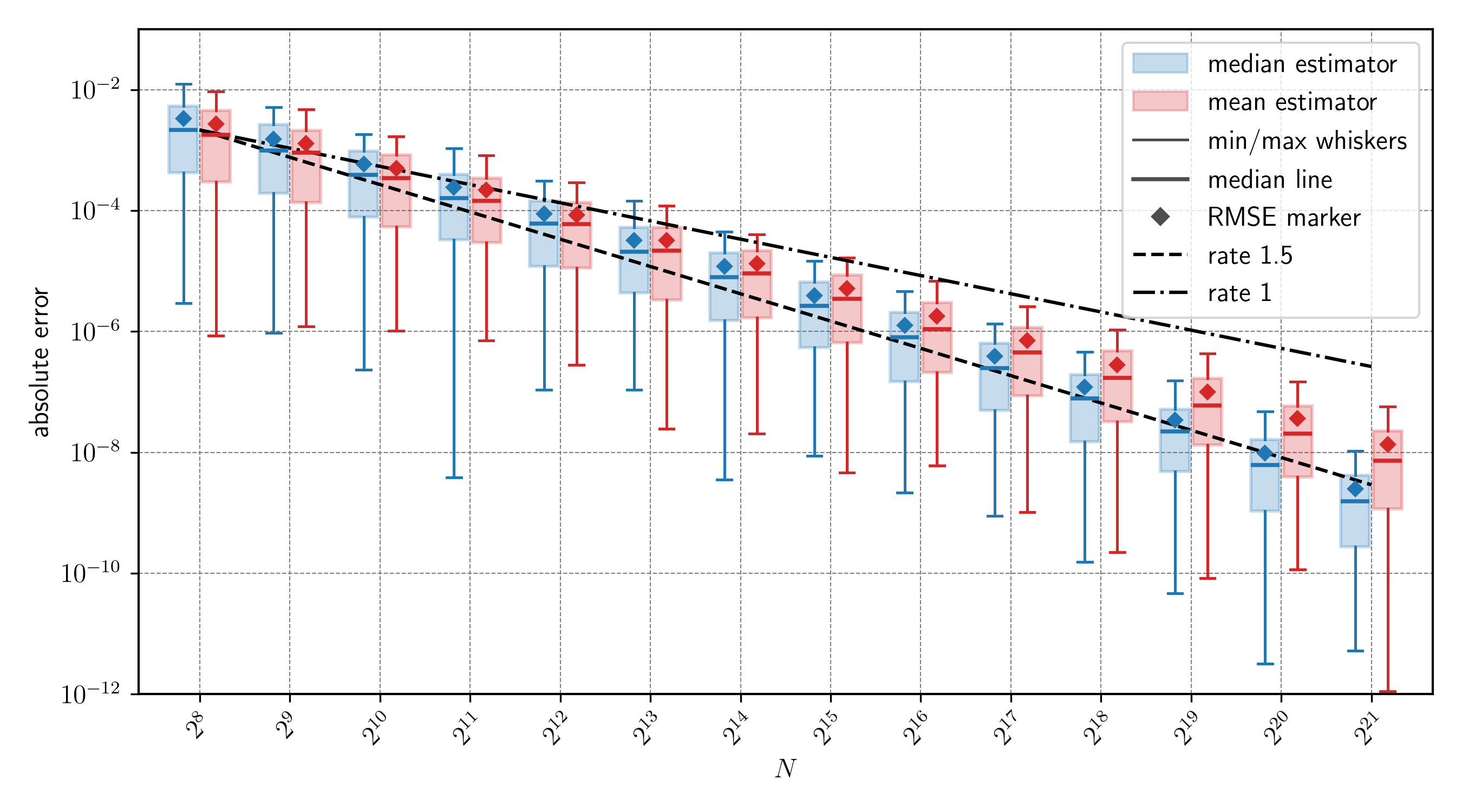}
	}
    \\
    \subfloat[$d_{1}=4$]{%
		\includegraphics[width=0.42\textwidth]{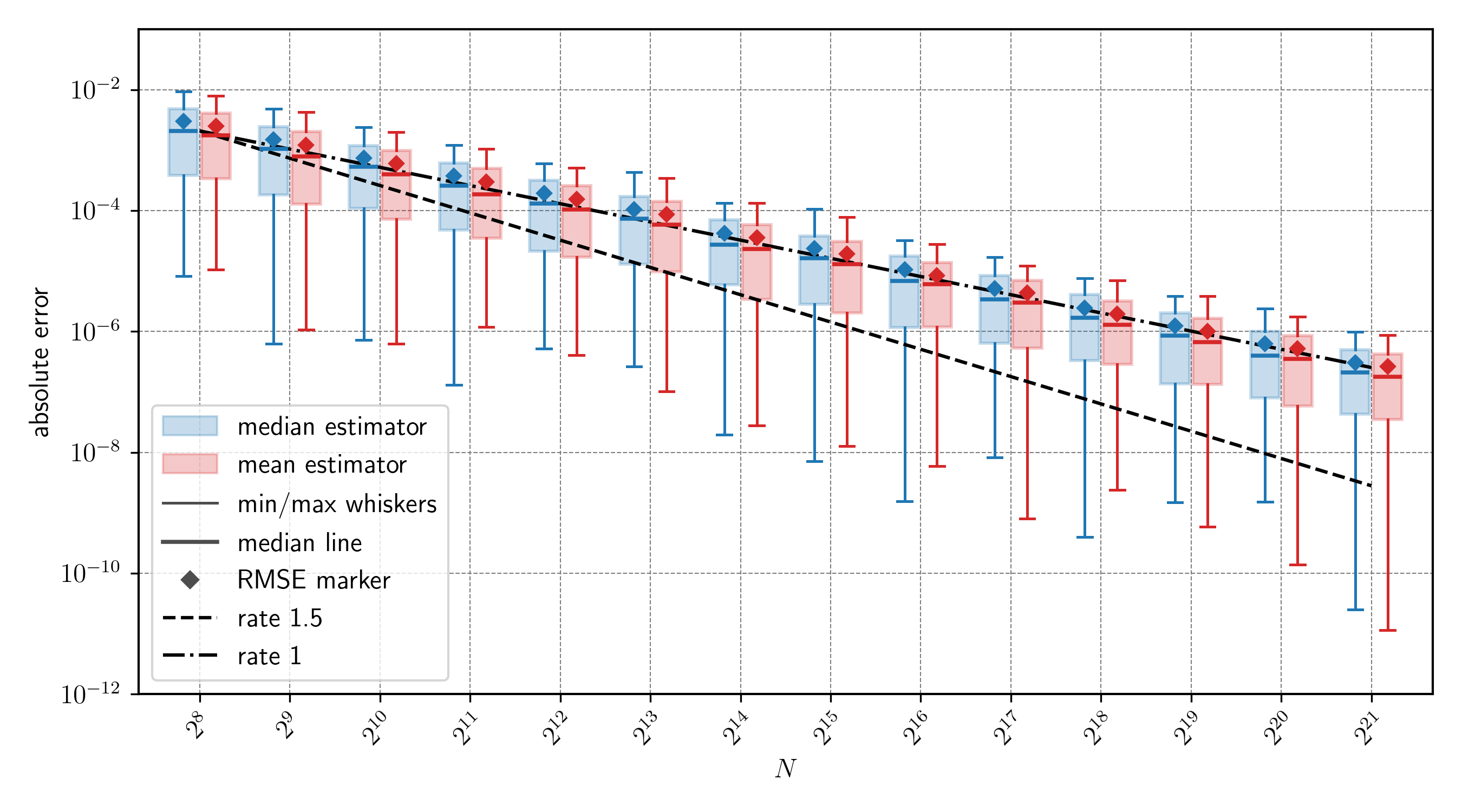}
	}
    \subfloat[$d_{1}=4$ with IS]{%
		\includegraphics[width=0.42\textwidth]{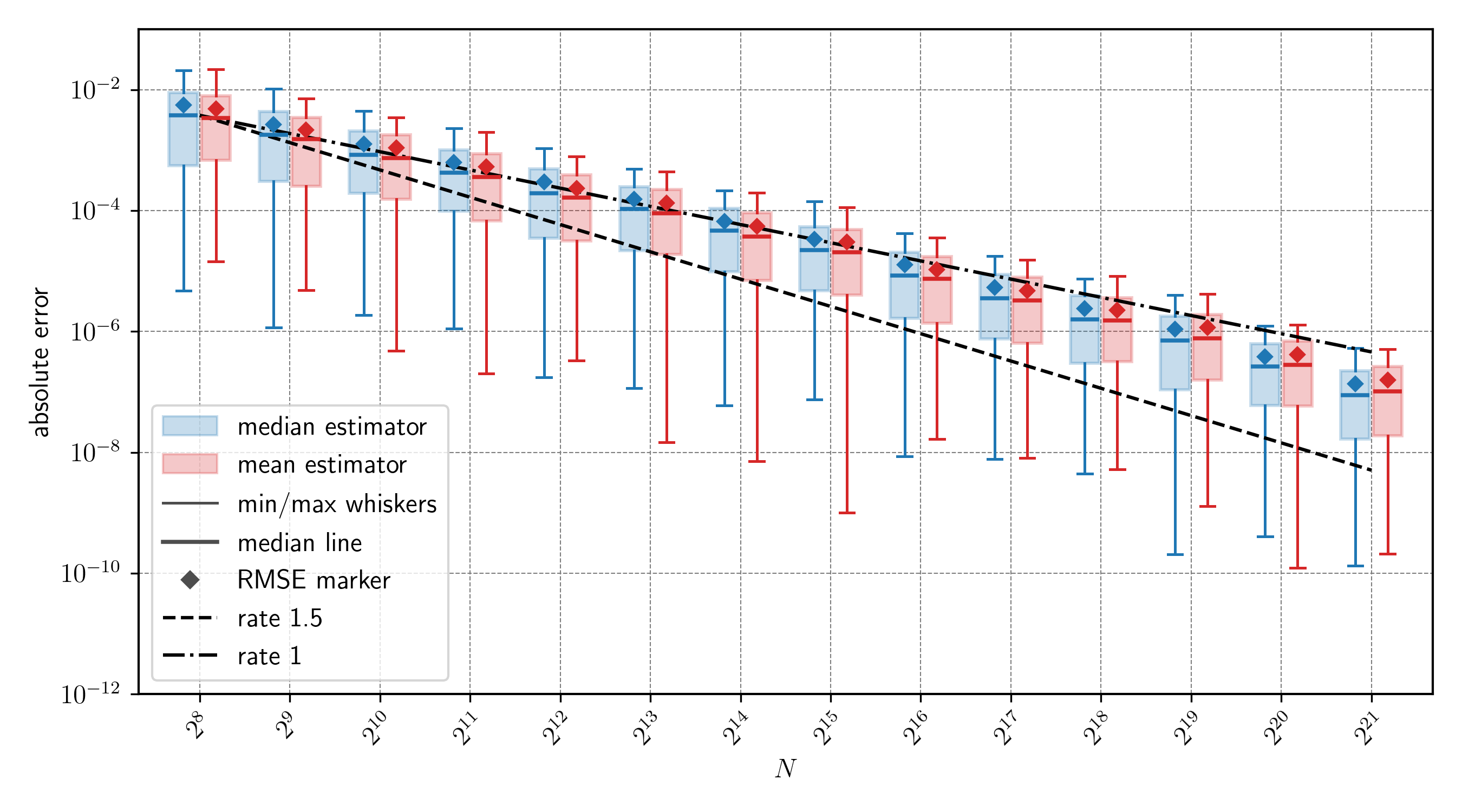}
	}
    \\
	\subfloat[$d_{1}=5$]{%
		\includegraphics[width=0.42\textwidth]{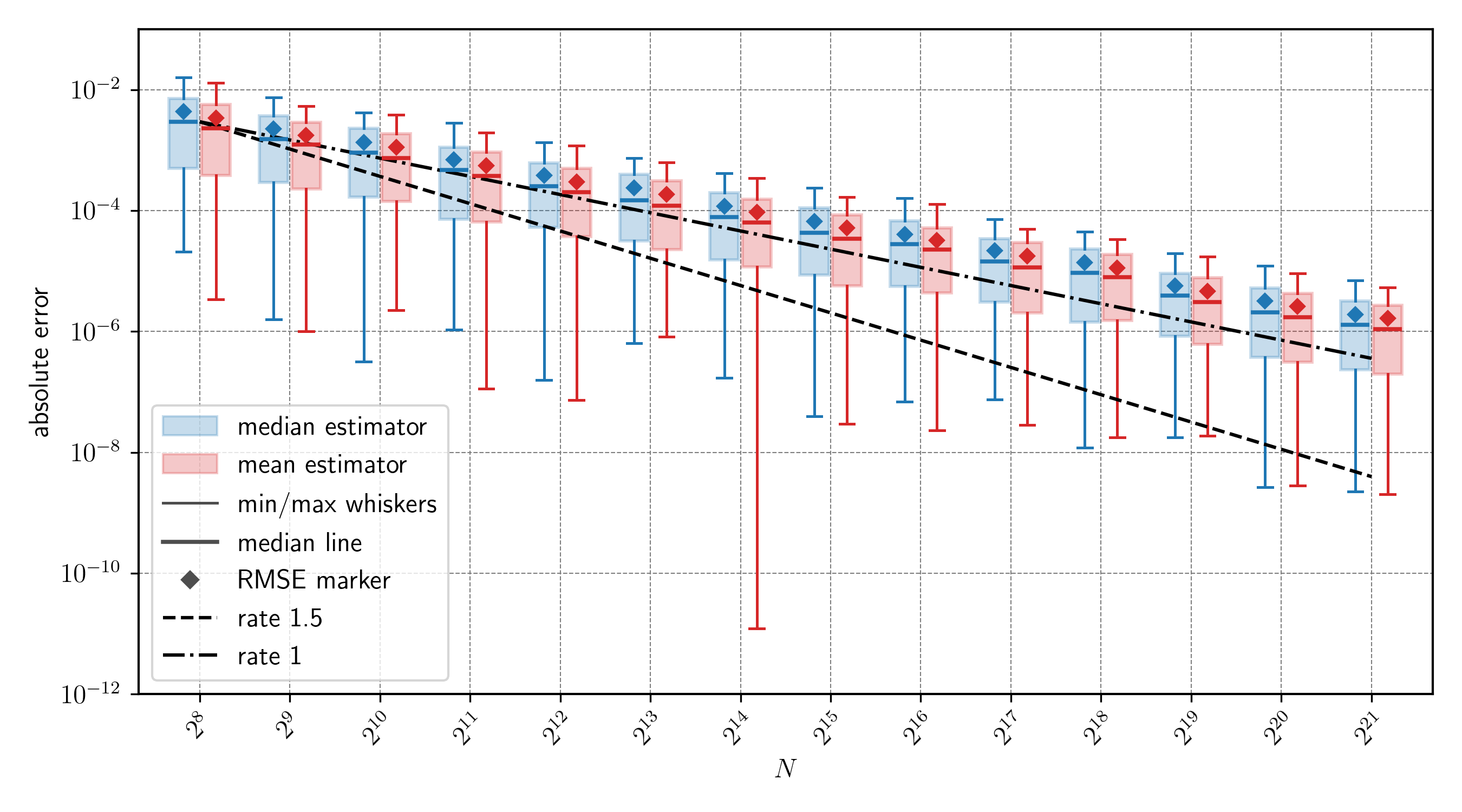}
	}
	\subfloat[$d_{1}=5$ with IS]{%
		\includegraphics[width=0.42\textwidth]{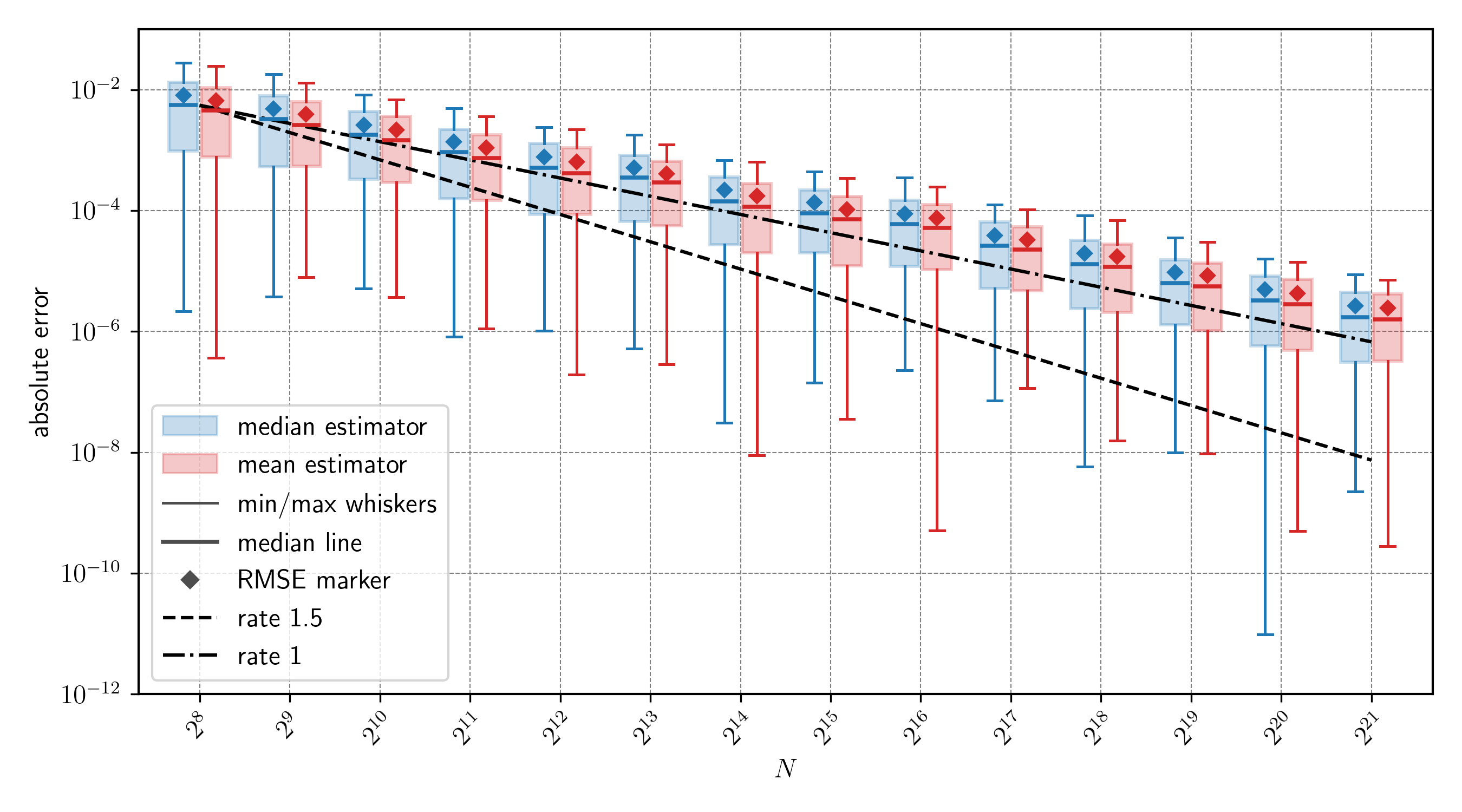}
	}\\
    \subfloat[$d_{1}=6$]{%
		\includegraphics[width=0.42\textwidth]{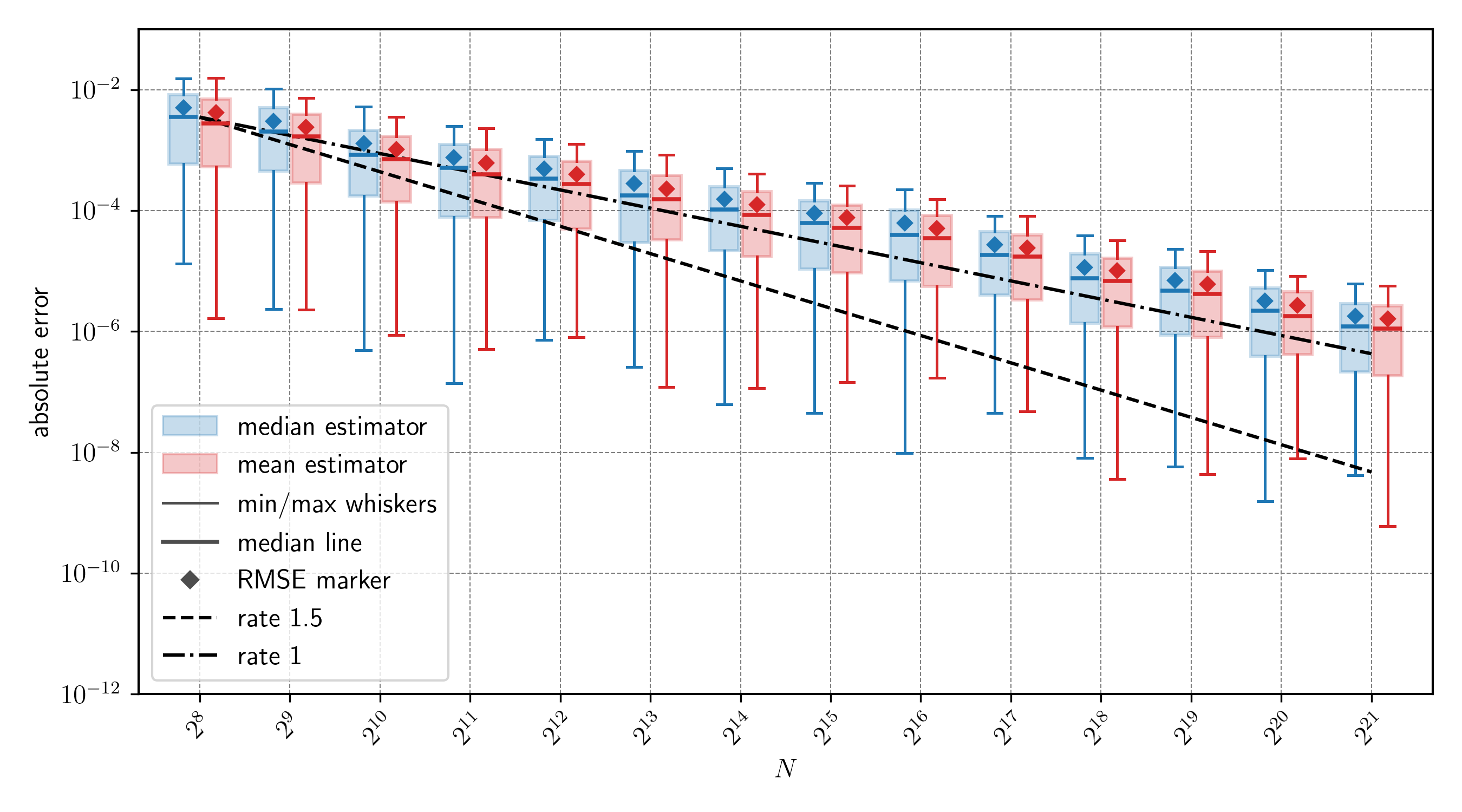}
	}
    \subfloat[$d_{1}=6$ with IS]{%
		\includegraphics[width=0.42\textwidth]{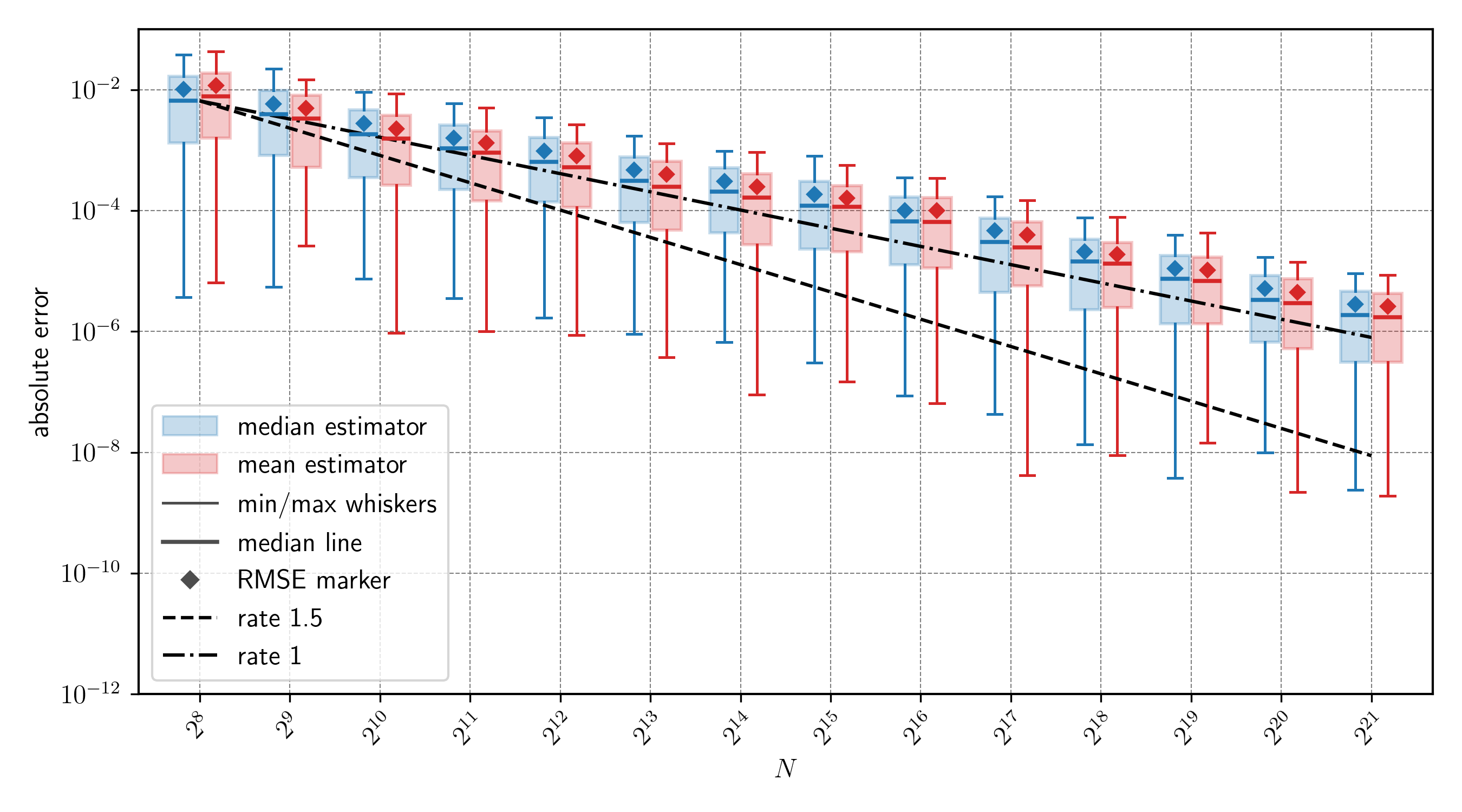}
	}
	\caption{EIG Example: Outer integration error for $J=1000$ replicates of the DLRQMC and the DLmRQMC estimators in various dimensions with and without importance sampling (IS). Both estimators achieve a convergence rate of one without IS (left column). With IS (right column), better rates can be achieved in low dimensions for the sample sizes considered.}
    \label{fig:EIG.outer}
\end{figure}

In Figure~\ref{fig:eig.ccdf.inner} we demonstrate that increasing the number of Sobol' points leads to more extreme outliers for the inner loop. We used $J=2000$ replicates of the inner mRQMC estimators with $R=1$. Fewer large outliers are observed as the dimension $d_2$ increases from one to four.
\begin{figure}[ht]
	\subfloat[$d_2=1$]{%
		\includegraphics[width=0.4\textwidth]{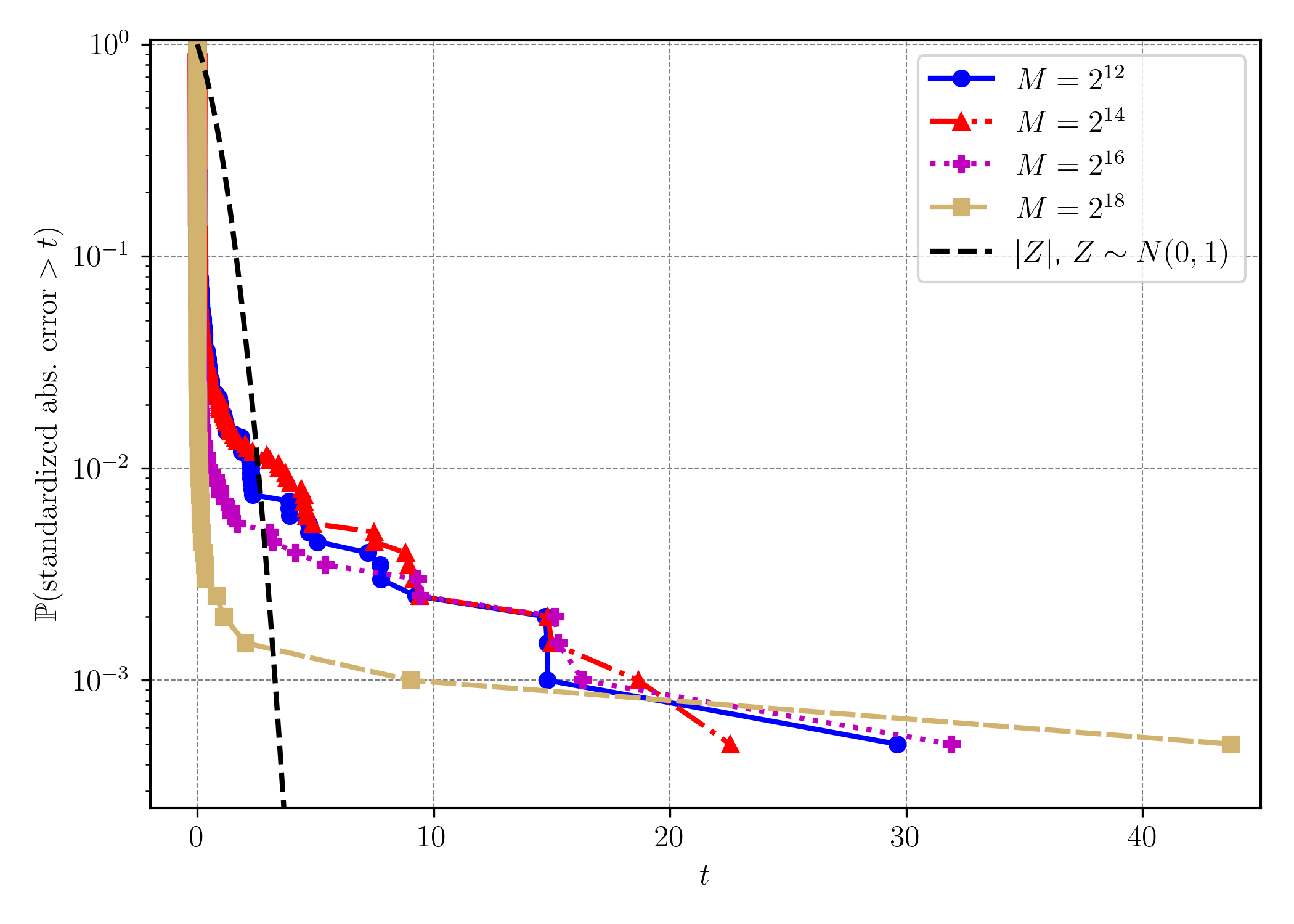}
	}
	\subfloat[$d_2=2$]{%
		\includegraphics[width=0.4\textwidth]{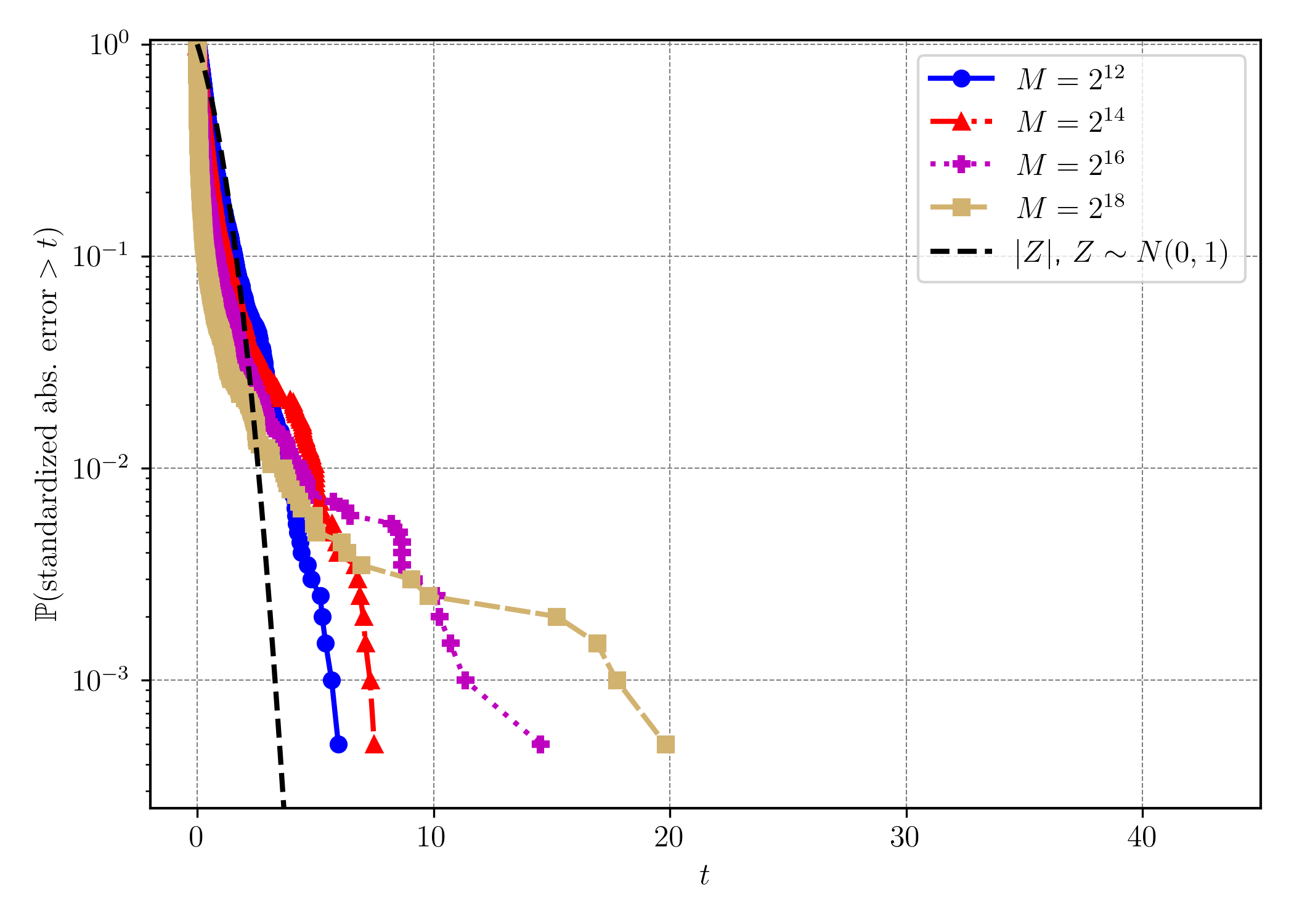}
	}\\
	\subfloat[$d_2=3$]{%
		\includegraphics[width=0.4\textwidth]{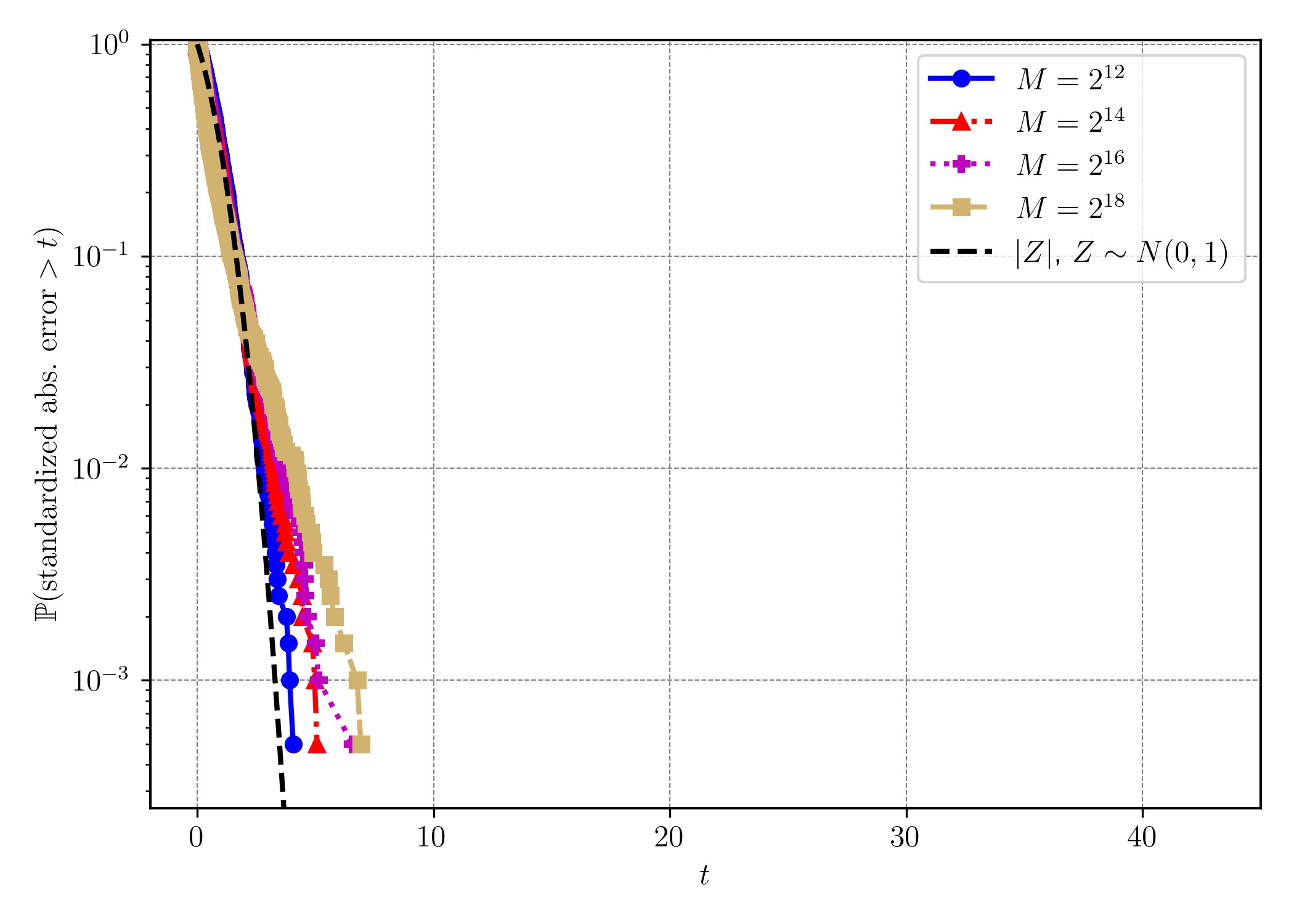}
	}
	\subfloat[$d_2=4$]{%
		\includegraphics[width=0.4\textwidth]{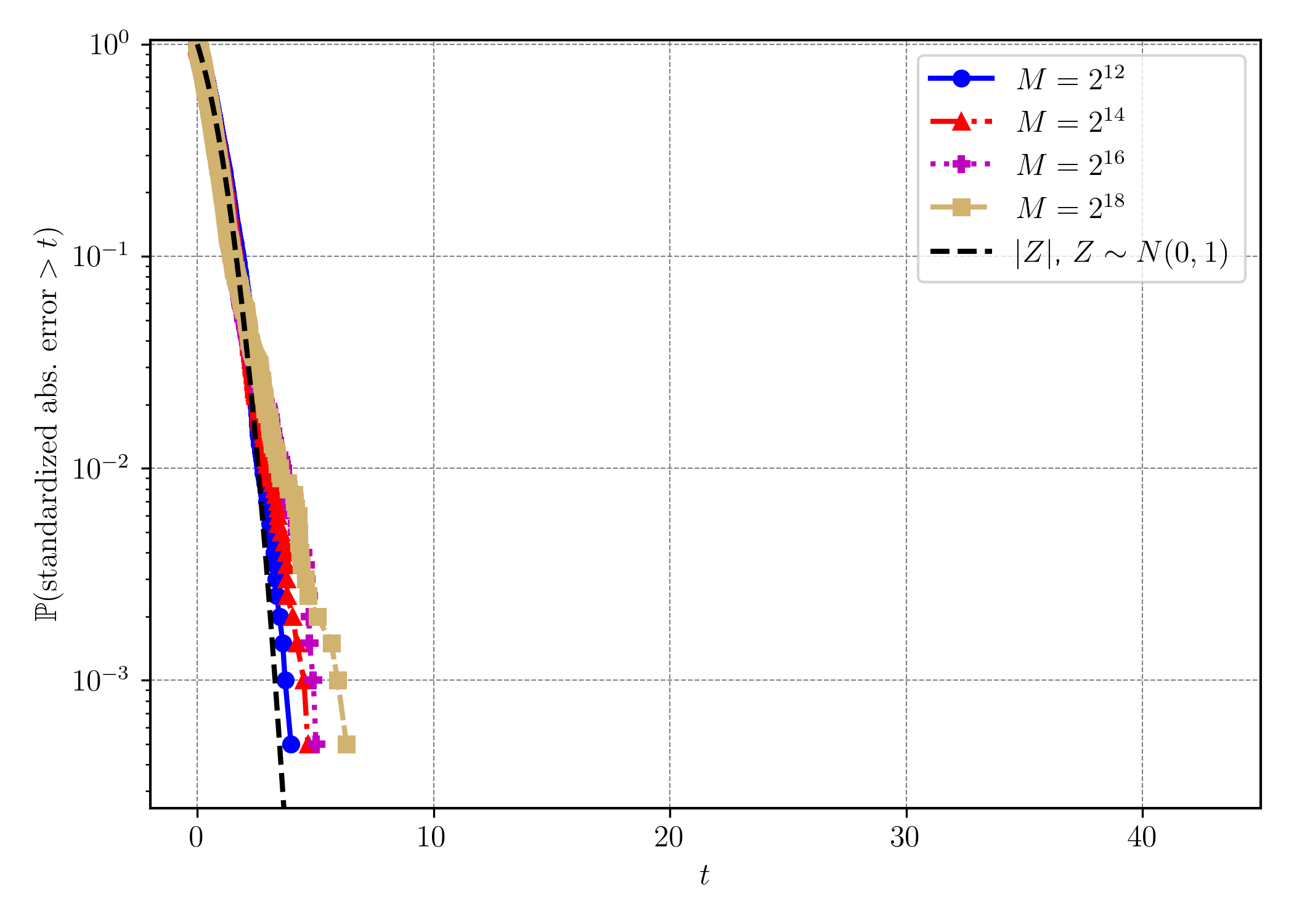}
	}
	\caption{EIG example: CCDF of the standardized absolute errors in dimensions $d_2\in\{1,2,3,4\}$ over $J=2000$ randomizations for the inner loop. The distribution of standardized absolute errors produces increasingly large outliers as the number of Sobol' points in the inner loop $M$ increases. The standardized absolute errors of the standard normal are presented as a reference.}
    \label{fig:eig.ccdf.inner}
\end{figure}
A similar behavior can be observed for the outer error with importance sampling. 

In Figure~\ref{fig:eig.ccdf.outer} we present $J=2000$ replicates of the outer mRQMC estimator with importance sampling for $S=1$ (Panel (B), (D), (F), and (H)). Without importance sampling, no such increase of large outliers is observed for an increasing number of Sobol' points in any dimension $d_1$ between three and six. These results indicate that the median-based estimators will outperform the mean estimators for the inner error and outer error with importance sampling, consistent with our theoretical analysis and error convergence results. Without importance sampling, there are not enough large outliers for the median to remove and hence outperform the mean. As the integrand dimension increases, the mean and median based estimators also display more and more similar behavior, consistent with these findings.

Demonstrating the convergence behavior of the total error of the double-loop estimators in the asymptotic or near-asymptotic regime would require an increasingly large number of inner and outer samples. However, the above results indicate that substantial improvements can be gained by replacing the mean with the median in both the inner and the outer RQMC estimators in the nested setting.

\begin{figure}[ht]
	\subfloat[$d_{1}=3$]{%
		\includegraphics[width=0.4\textwidth]{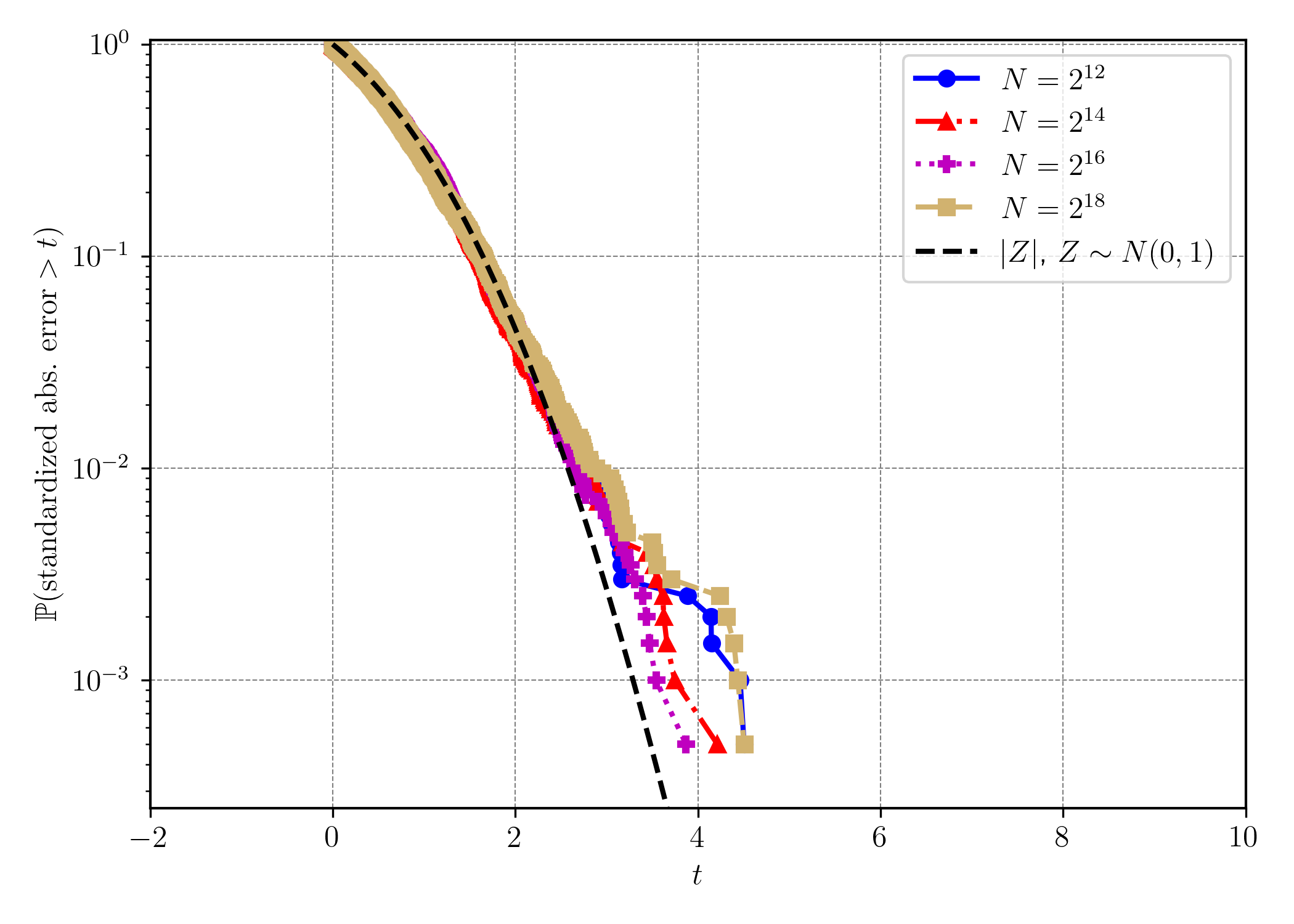}
	}
    	\subfloat[$d_{1}=3$ with IS]{%
		\includegraphics[width=0.4\textwidth]{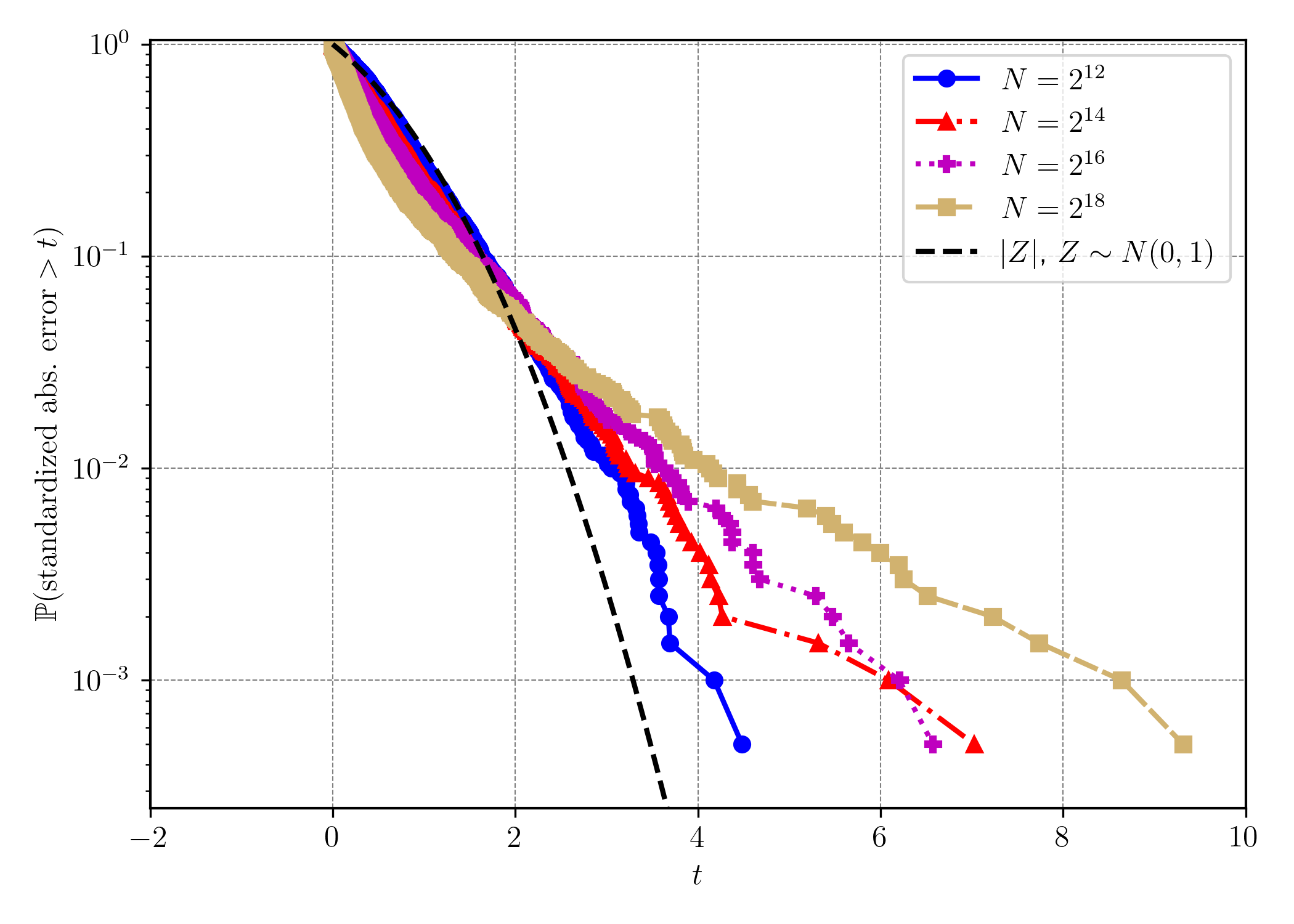}
	}
    \\
    \subfloat[$d_{1}=4$]{%
		\includegraphics[width=0.4\textwidth]{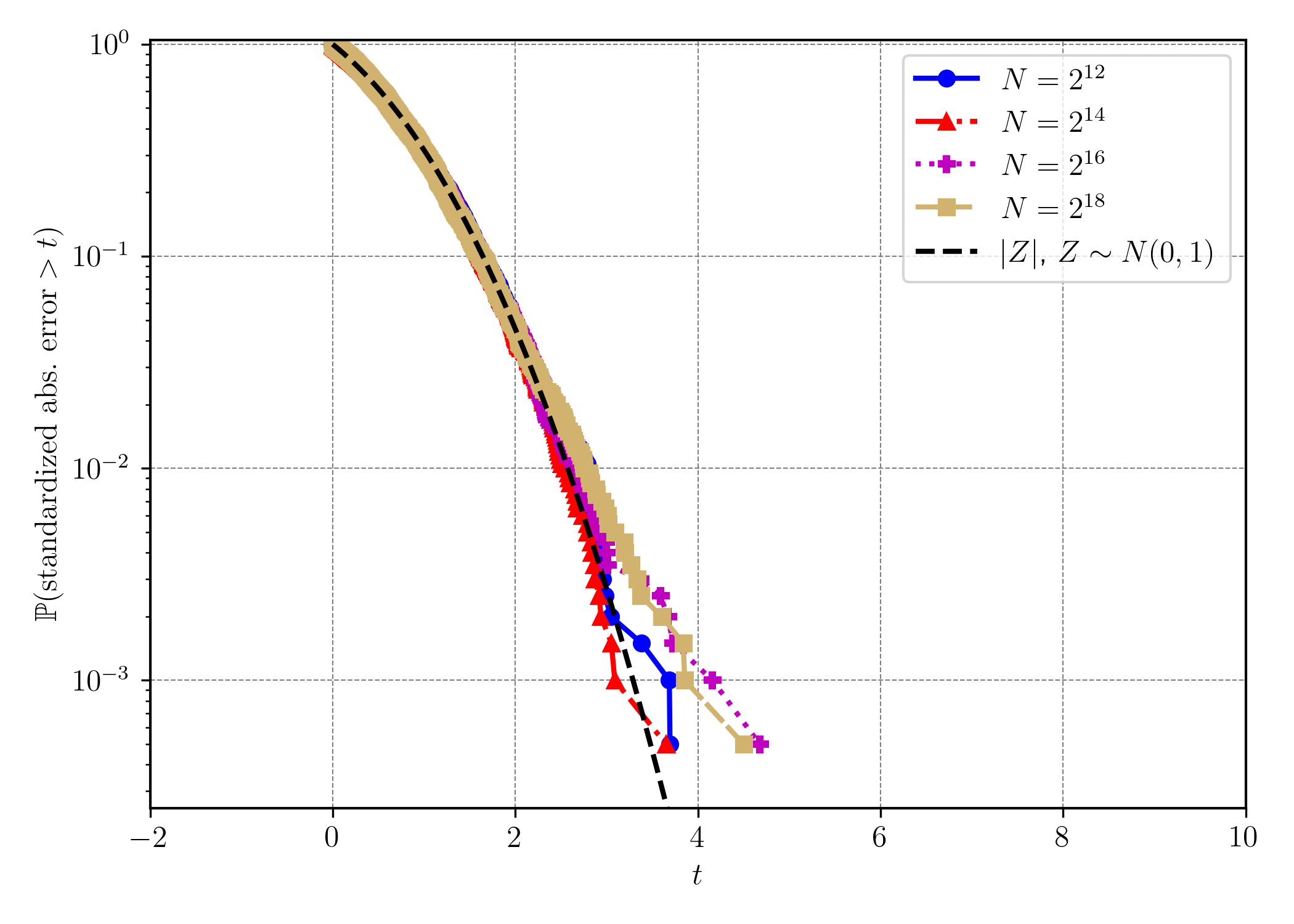}
	}
    \subfloat[$d_{1}=4$ with IS]{%
		\includegraphics[width=0.4\textwidth]{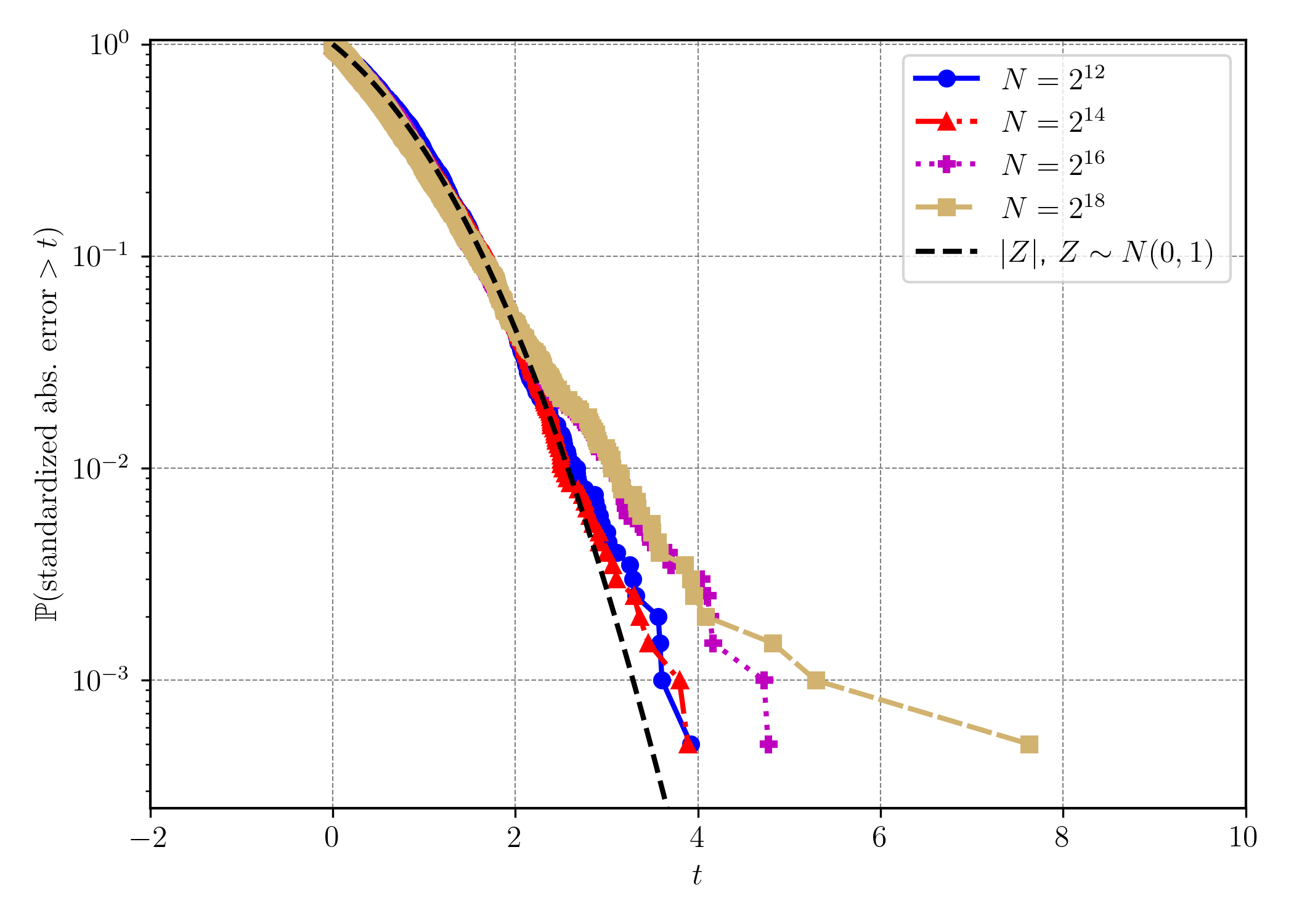}
	}
    \\
	\subfloat[$d_{1}=5$]{%
		\includegraphics[width=0.4\textwidth]{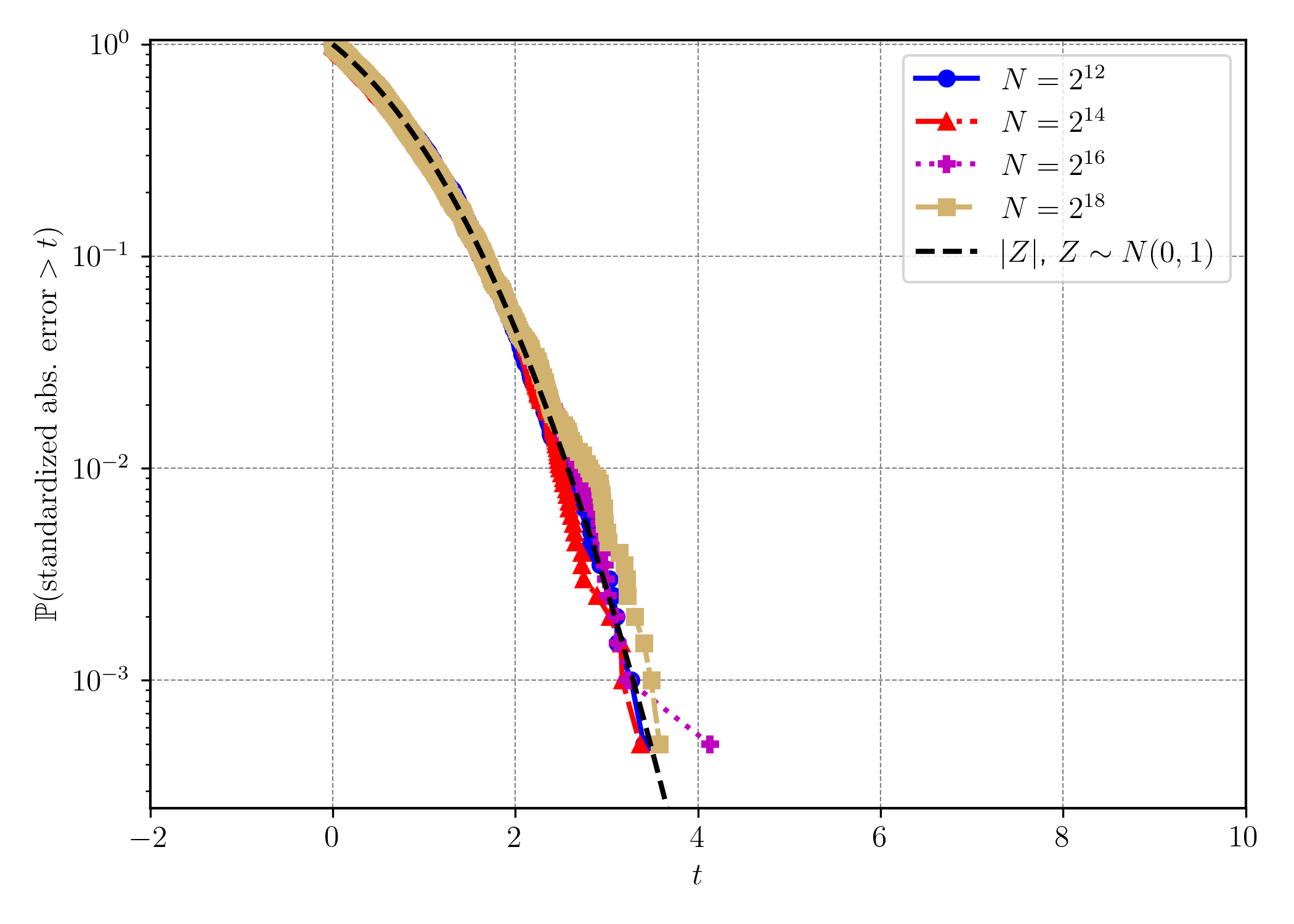}
	}
	\subfloat[$d_{1}=5$ with IS]{%
		\includegraphics[width=0.4\textwidth]{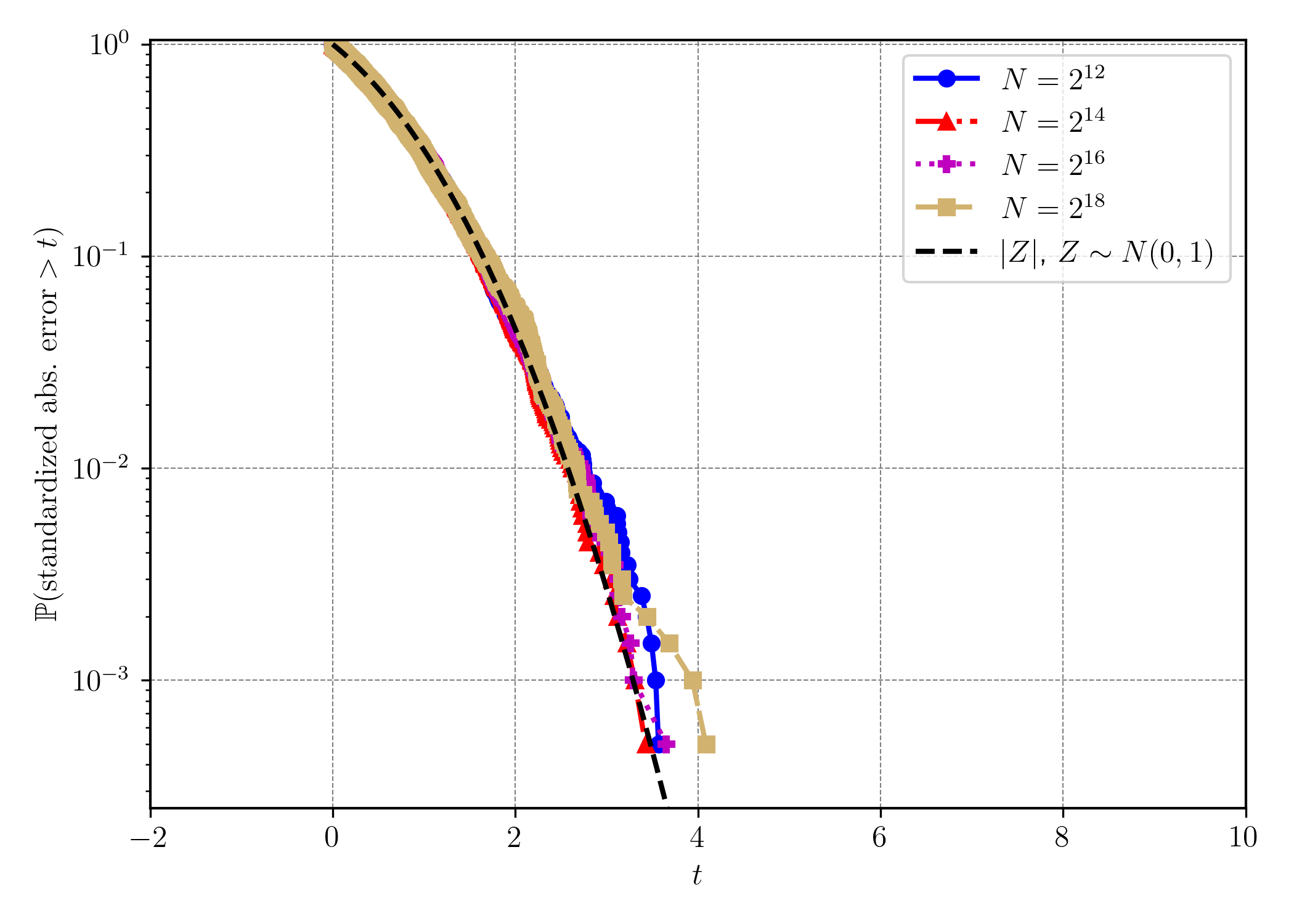}
	}\\
    \subfloat[$d_{1}=6$]{%
		\includegraphics[width=0.4\textwidth]{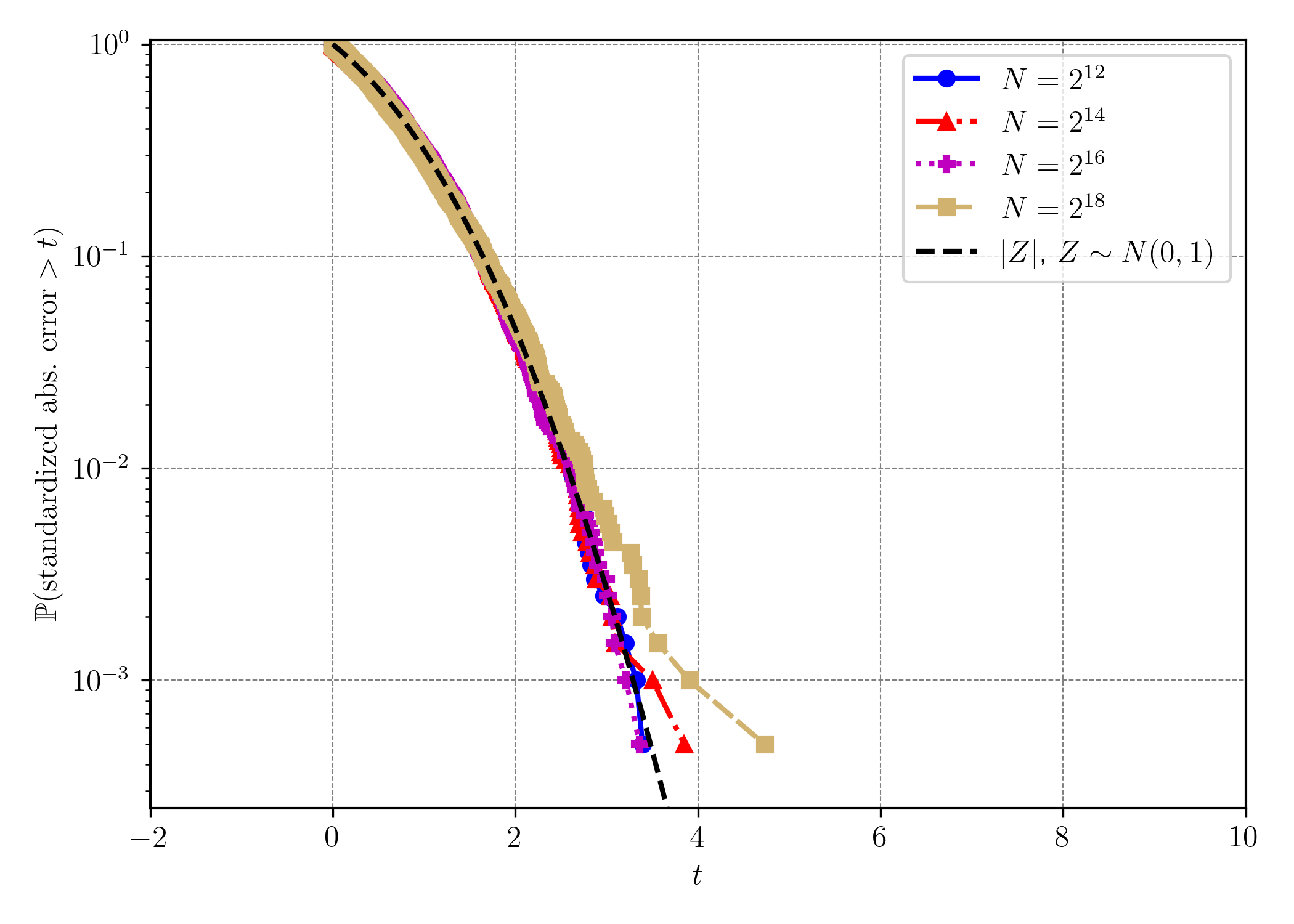}
	}
    \subfloat[$d_{1}=6$ with IS]{%
		\includegraphics[width=0.4\textwidth]{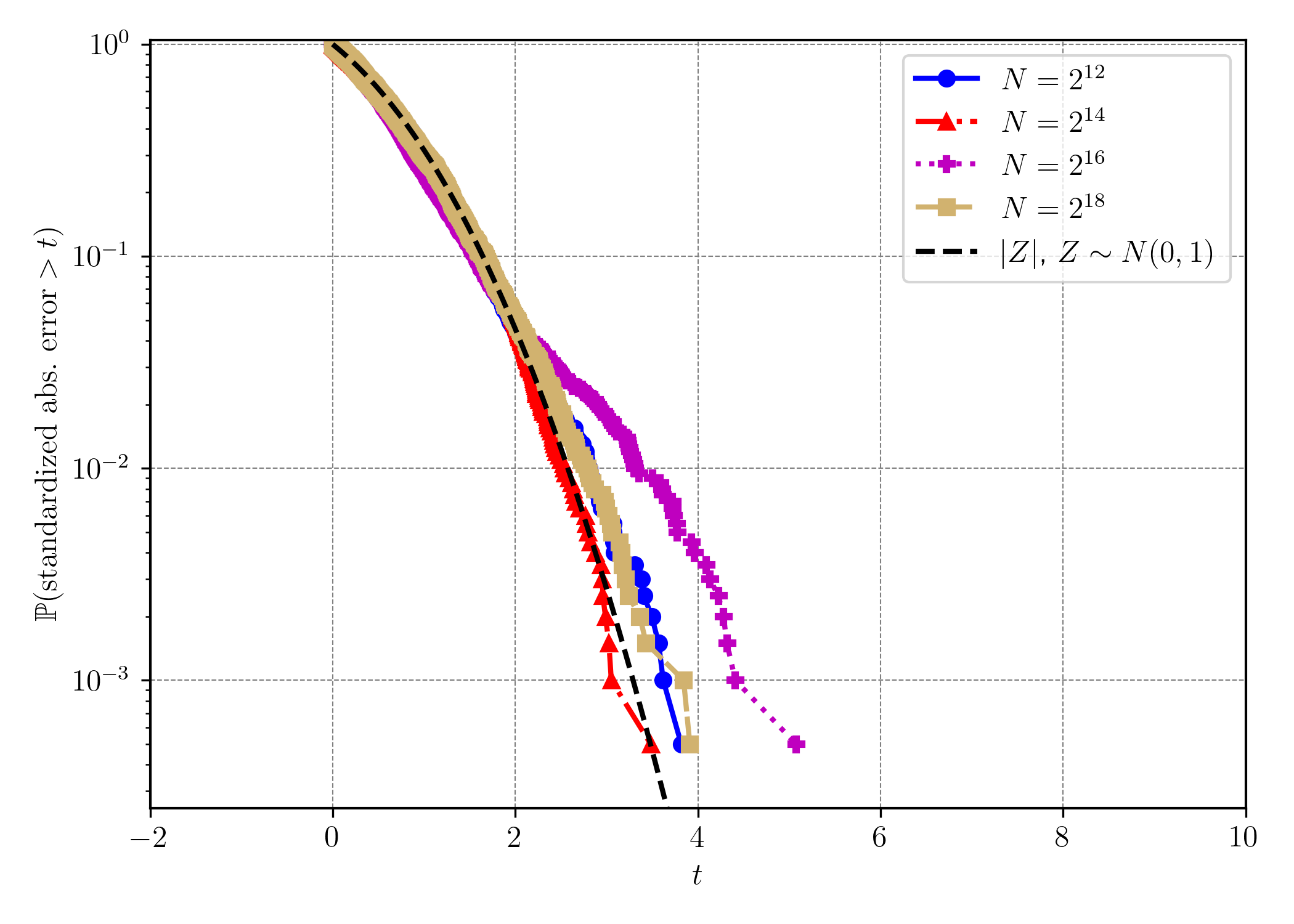}
	}
	\caption{EIG example: CCDF of the standardized absolute errors in dimensions $d_1\in\{3,4,5,6\}$ over $J=2000$ randomizations for the outer loop without importance sampling (IS, left column) and with IS (right column).}
    \label{fig:eig.ccdf.outer}
\end{figure}
\FloatBarrier

\section{Conclusion}
\label{sec:conclusion}
In this paper, we proposed a double-loop median randomized quasi-Monte Carlo estimator suitable for nested integrals and analyze its asymptotic behavior, based on the works of Pan and Owen~\cite{Pan23, Pan24, Pan25}. We demonstrated that the root mean squared error of our estimator displays a super-polynomial convergence rate similar to that of single-loop estimators for non-nested integrals. Then, we proposed an importance sampling scheme for the outer integrand in the expected information gain (EIG) of an experiment to make this integrand suitable for our estimator and compared this method to the standard approach using the inverse cumulative distribution function. Finally, we numerically demonstrated convergence rates of the inner loop of our estimator consistent with super-polynomial convergence for an EIG application. For the outer loop, we observed convergence rates greater than 1.5, that is, superior convergence to that attainable by standard randomized quasi-Monte Carlo methods. These results demonstrate the advantage of taking the sample median over the randomizations, rather than the sample mean, in both the outer and the inner loop.

An important avenue for future research is the extension of the proposed importance sampling scheme to parameters of interest with Gaussian priors and how the presented importance sampling scheme behaves in the small-noise or large-data limit, where a concentration of measure effect often leads to numerical underflow in the inner integral estimation. Moreover, it is worth exploring whether our importance sampling scheme for the outer integrand may be combined with importance sampling for the inner integrand based on the Laplace approximation. The analysis we presented is for general integrands and can be applied in various areas; however, the particular regularity assumptions we made are geared towards the EIG application. Another possible research avenue would be to extend the present analysis to other applications.

\section*{Acknowledgments}
This publication is based upon work supported by the King Abdullah University of Science and Technology (KAUST) Office of Sponsored Research (OSR) under Award No.~OSR-2019-CRG8-4033 and the Alexander von Humboldt Foundation.

\section*{Declarations}
This manuscript was prepared with the assistance of Anthropic Claude, OpenAI ChatGPT and Grammarly as
editing and structural-review tools; all mathematical content, numerical experiments, and final wording decisions are the responsibility of the authors.

\appendix

\subfile{./appendix.tex}

\footnotesize
\printbibliography
\end{document}

%% file: abstract.tex
Nested integrals $\int f\left(\int g(\bs{y},\bs{x})\di{}\bs{x}\right)\di{}\bs{y}$ where $\bs{y}\in[0,1]^{d_1}$ and $\bs{x}\in[0,1]^{d_2}$ frequently arise in Bayesian experimental design, medical decision making, and computational finance. Statistical estimators for nested integrals are more computationally demanding than for standard integrals. Recent studies demonstrate that the median randomized quasi-Monte Carlo (mRQMC) method achieves super-polynomial convergence rates for analytic integrands. We extend mRQMC methods to nested integrations and propose a novel double-loop mRQMC estimator with super-polynomial convergence rate if both the inner and outer integrands are analytic. Established estimators for nested integrals, such as double-loop or multilevel Monte Carlo, are by contrast limited to algebraic convergence rates. However, expected information gain (EIG) applications from experimental design typically do not fulfill this condition due to a lack of regularity in the outer integrand. We address this through a theoretically-grounded importance sampling regularization scheme. Although super-polynomial convergence rates remain achievable for the inner integral, only the algebraic rate $\alpha+3/2-\epsilon$ for any $\epsilon>0$ can be theoretically guaranteed for the outer integral in the EIG, where $\alpha\in\mathbb{N}\cup\{0\}$ denotes the order of the mixed derivatives of the regularized outer integrand. Our approach shows substantial improvements over double-loop randomized quasi-Monte Carlo methods based on the sample mean in moderately high-dimensional settings for a simple polynomial example. Moreover, we demonstrate improved convergence behavior of the median-based estimator for the more challenging setting of EIG estimation.
\\
\textbf{Keywords:} median randomized quasi-Monte Carlo $\cdot$ nested integration $\cdot$ expected information gain $\cdot$ importance sampling
\\
\textbf{AMS subject classifications:}
62F15 
$\cdot$
65C05 
$\cdot$
65D30 
$\cdot$
65D32 

%% file: appendix.tex
\section{Proof of Lemma Maximum difference between two sequences}\label{app:max.diff}
In this section we provide the proof of Lemma~\ref{lem:max.diff}.
\begin{proof}
We introduce
\begin{equation}
    \hat{A}^{(r)}\coloneqq \sigma_{A,\mathrm{asc}}\left(A^{(i_r)}\right), \quad \text{for $1\leq r\leq R$}
\end{equation}
for appropriately chosen indices $i_1,\ldots,i_R$ such that $(\hat{A}^{(r)})_{r=1}^R$ is in ascending order and similarly also
\begin{equation}
    \hat{B}^{(r)}\coloneqq \sigma_{B,\mathrm{asc}}\left(B^{(i_r)}\right), \quad \text{for $1\leq r\leq R$}
\end{equation}
such that $(\hat{B}^{(r)})_{r=1}^R$ is in ascending order for possibly different indices $i_1,\ldots,i_R$. Moreover, the order of the absolute differences in the maximum is not important. Therefore, we keep the sequence $\hat{A}$ as it is, and aim to construct a permutation $\sigma\in\mathcal{S}$ of the sequence $\hat{B}$ that reduces the maximum over the absolute differences compared to $\sigma_{B,\mathrm{asc}}$. If we encounter a contradiction, that means that the permutations that sorts the sequences in ascending order already minimize the maximum over the absolute differences. Let
\begin{equation}
    r^{\ast}\coloneqq \argmax_{r=1,\ldots,R}|\hat{A}^{(r)}-\hat{B}^{(r)}|.
\end{equation}
We only need to consider permutations $\sigma$ for which $\sigma(\hat{B}^{(r^{\ast})})\neq \hat{B}^{(r^{\ast})}$, as the absolute differences would otherwise still include the old maximum and therefore not reduce the new maximum. We therefore consider $\sigma$ such that $\sigma(\hat{B}^{(r^{\ast})})= \hat{B}^{(k)}$, where $k\neq r^{\ast}$. If
\begin{equation}
    |\hat{A}^{(r^{\ast})}-\hat{B}^{(k)}|\geq|\hat{A}^{(r^{\ast})}-\hat{B}^{(r^{\ast})}|,
\end{equation}
then the maximum difference for the permutation $\sigma$ is not reduced compared to the old maximum. We thus assume that
\begin{equation}\label{eq:asu.r.ast}
    |\hat{A}^{(r^{\ast})}-\hat{B}^{(k)}|<|\hat{A}^{(r^{\ast})}-\hat{B}^{(r^{\ast})}|,
\end{equation}
and will now demonstrate that we can always find another absolute difference for which it holds that
\begin{equation}
|\hat{A}^{(i)}-\hat{B}^{(j)}|\geq|\hat{A}^{(r^{\ast})}-\hat{B}^{(r^{\ast})}|
\end{equation}
for some $1\leq i,j\leq R$. Assume for now that $k>r^{\ast}$. It follows from~\eqref{eq:asu.r.ast} that $\hat{A}^{(r^{\ast})}>\hat{B}^{(r^{\ast})}$. This is demonstrated as follows. We immediately see that $\hat{A}^{(r^{\ast})}\neq\hat{B}^{(r^{\ast})}$, otherwise \eqref{eq:asu.r.ast} implies that $|\hat{A}^{(r^{\ast})}-\hat{B}^{(k)}|<0$. Assume that $\hat{A}^{(r^{\ast})}<\hat{B}^{(r^{\ast})}$. We know that $\hat{B}^{(r^{\ast})}\leq \hat{B}^{(k)}$ and thus $\hat{A}^{(r^{\ast})}<\hat{B}^{(r^{\ast})}\leq \hat{B}^{(k)}$, a contradiction to~\eqref{eq:asu.r.ast} and thus $\hat{A}^{(r^{\ast})}>\hat{B}^{(r^{\ast})}$ under the assumption that $k>r^{\ast}$.  We now consider a term $|\hat{A}^{(i)}-\hat{B}^{(j)}|$ such that $i>r^{\ast}$ and $j\leq r^{\ast}$. It follows that 
\begin{equation}
\hat{B}^{(j)}\leq \hat{B}^{(r^{\ast})}<\hat{A}^{(r^{\ast})}\leq \hat{A}^{(i)}.
\end{equation}
From this, it follows that 
\begin{equation}
|\hat{A}^{(i)}-\hat{B}^{(j)}|\geq |\hat{A}^{(r^{\ast})}-\hat{B}^{(r^{\ast})}|
\end{equation}
and we have thus not reduced the maximum over the absolute differences. Such indices $i,j$ always exist, as $r^{\ast}<k\leq R$ by assumption, and thus there exists $i>r^{\ast}$. 

Moreover, among all terms $|\hat{A}^{(r^{\ast}+1)}-\hat{B}^{(j_{r^{\ast}+1})}|,\ldots,|\hat{A}^{(R)}-\hat{B}^{(j_{R})}|$, one of the indices $j_{r^{\ast}+1},\ldots,j_{R}$ must be smaller or equal to $r^{\ast}$, as there are only $R-r^{\ast}$ indices greater than $r^{\ast}$ and the index $k>r^{\ast}$ already appears in the term $|\hat{A}^{(r^{\ast})}-\hat{B}^{(k)}|$. 

By a similar argument, it follows that if $k<r^{\ast}$, then $\hat{A}^{(r^{\ast})}<\hat{B}^{(r^{\ast})}$. Considering a term $|\hat{A}^{(i)}-\hat{B}^{(j)}|$ such that $i<r^{\ast}$ and $j\geq r^{\ast}$, it follows that 
\begin{equation}
\hat{A}^{(i)}\leq \hat{A}^{(r^{\ast})}<\hat{B}^{(r^{\ast})}\leq \hat{B}^{(j)}
\end{equation}
and thus again that
\begin{equation}
|\hat{A}^{(i)}-\hat{B}^{(j)}|\geq |\hat{A}^{(r^{\ast})}-\hat{B}^{(r^{\ast})}|.
\end{equation}
We thus demonstrated that the maximum over the absolute differences of sequences in ascending order cannot be reduced by any permutation.
\end{proof}

\section{Uniform analyticity of \texorpdfstring{$g$}{inner integrand}}\label{app:uniform.analyticity}
Here, we provide the proof of Lemma~\ref{Lem:asu.2}.
\begin{proof}
 We fix $\bs{y}\in[0,1]^{d_1}$ and write
 \begin{equation}
     g_{\bs{y}}(\bs{x})\coloneqq g(\bs{y},\bs{x}),
 \end{equation}
 and similarly also
  \begin{equation}
     \Lambda_{\bs{y}}(\bs{x})\coloneqq \Lambda(\bs{y},\bs{x}),
 \end{equation}
 for $\Lambda$ as in~\eqref{def:Lambda}. Let $\lVert\bs{\alpha}\rVert_1=n$. It follows that
\begin{equation}
    \frac{\partial^{\lVert\bs{\alpha}\rVert_1}g_{\bs{y}}(\bs{x})}{\partial x_1^{\alpha_1}\cdots\partial x_{d_2}^{\alpha_{d_2}}}=\frac{\partial^{n}g_{\bs{y}}(\bs{x})}{\partial x_{i_1}\cdots\partial x_{i_n}}
\end{equation}
    for appropriately chosen indices $i_1,\ldots,i_n\in\{1,\ldots,d_2\}$ with repetition. Next, we apply Fa\`{a} di Bruno's formula to obtain
    \begin{equation}
        \frac{\partial^{n}}{\partial x_{i_1}\cdots\partial x_{i_n}}g_{\bs{y}}(\bs{x})=\sum_{\pi}(-1)^{|\pi|}e^{-\Lambda_{\bs{y}}}\prod_{b\in\pi}\frac{\partial^{|b|}\Lambda_{\bs{y}}(\bs{x})}{\prod_{j\in b}\partial x_{i_j}},
    \end{equation}
    where $\pi$ runs through the set of all partitions of $\{1,\ldots,n\}$ and $b$ runs through the list of all the blocks of partition $\pi$. It immediately follows that
        \begin{equation}\label{eq:Faa.di.Bruno.likelihood}
        \left|\frac{\partial^{n}}{\partial x_{i_1}\cdots\partial x_{i_n}}g_{\bs{y}}(\bs{x})\right|\leq\sum_{\pi}e^{-\Lambda_{\bs{y}}}\prod_{b\in\pi}\left|\frac{\partial^{|b|}\Lambda_{\bs{y}}(\bs{x})}{\prod_{j\in b}\partial x_{i_j}}\right|.
    \end{equation}
    To demonstrate the bound on the partial derivatives of $\Lambda_{\bs{y}}$, let $\bs{\beta}\in\mathbb{N}_{0}^{d_2}$ such that $\left\lVert \bs{\beta}\right\rVert_1=|b|$. Moreover, let
    \begin{equation}\label{def:v}
        \bs{v}_{\bs{y}}(\bs{x})\coloneqq \bs{\Sigma}_{\bs{\varepsilon}}^{-\frac{1}{2}}\left(\bs{G}(\bs{y}_1)-\bs{G}(\bs{x})\right)+\bs{\Phi}^{-1}(\bs{y}_2)
    \end{equation}
    so that
    \begin{align}
        \Lambda_{\bs{y}}(\bs{x})\coloneqq \frac{1}{2}\left\lVert \bs{v}_{\bs{y}}(\bs{x})\right\rVert^2=\frac{1}{2}\left\langle \bs{v}_{\bs{y}}(\bs{x}),\bs{v}_{\bs{y}}(\bs{x}) \right\rangle.
    \end{align}
    It follows from the multivariate general Leibniz rule that
    \begin{equation}
        \frac{\partial^{\left\lVert\bs{\beta} \right\rVert_1}\Lambda_{\bs{y}}(\bs{x})}{\partial x_1^{\beta_1}\cdots\partial x_{d_2}^{\beta_{d_2}}}=\frac{1}{2}\sum_{\bs{\nu}\leq\bs{\beta}}\binom{\bs{\beta}}{\bs{\nu}}\left\langle \frac{\partial^{\left\lVert\bs{\nu}\right\rVert_1}\bs{v}_{\bs{y}}(\bs{x})}{\partial x_1^{\nu_1}\cdots\partial x_{d_2}^{\nu_{d_2}}},\frac{\partial^{\left\lVert\bs{\beta}-\bs{\nu}\right\rVert_1}\bs{v}_{\bs{y}}(\bs{x})}{\partial x_1^{\beta_1-\nu_1}\cdots\partial x_{d_2}^{\beta_{d_2}-\nu_{d_2}}} \right\rangle,
    \end{equation}
    which we write in condensed notation as follows:
        \begin{equation}\label{eq:general.Leibniz.rule}
        \partial^{\bs{\beta}}\Lambda_{\bs{y}}=\frac{1}{2}\sum_{\bs{\nu}\leq\bs{\beta}}\binom{\bs{\beta}}{\bs{\nu}}\left\langle \partial^{\bs{\nu}}\bs{v}_{\bs{y}},\partial^{\bs{\beta}-\bs{\nu}}\bs{v}_{\bs{y}} \right\rangle,
    \end{equation}
understanding that all derivatives are with respect to the variable $\bs{x}\in[0,1]^{d_2}$. Moreover, it follows from~\eqref{def:v} that
    \begin{equation}
        \partial^{\bs{\beta}}\bs{v}_{\bs{y}}=-\bs{\Sigma}_{\bs{\varepsilon}}^{-\frac{1}{2}}\partial^{\bs{\beta}}\bs{G}
    \end{equation}
   for all $|\bs{\beta}|\geq 1$ and thus it follows from Assumption~\ref{asu:model} that there exist constants $A_{\bs{G}}<\infty$ and $\eta_{\bs{G}}>0$ such that
\begin{equation}\label{eq:bound.asu.5}
   \left\lVert\partial^{\bs{\beta}}\bs{v}_{\bs{y}}\right\rVert\leq\left\lVert \bs{\Sigma}_{\bs{\varepsilon}}^{-\frac{1}{2}} \right\rVert \frac{A_{\bs{G}}|b|!}{\eta_{\bs{G}}^{|b|}}
\end{equation}
for all $\bs{x}\in[0,1]^{d_2}$, $\bs{y}\in[0,1]^{d_1}$, and $\bs{\beta}\in\mathbb{N}_0^{d_2}$ such that $|\bs{\beta}|\geq 1$. Thus, only the terms in~\eqref{eq:general.Leibniz.rule} where $\bs{\nu}=\bs{0}$ and where $\bs{\nu}=\bs{\beta}$ depend on $\bs{y}$. For those cases, it follows that $\binom{\bs{\beta}}{\bs{\nu}}=1$ and thus, by symmetry, together they contribute a factor $\langle \bs{v}_{\bs{y}},\partial^{\bs{\beta}}\bs{v}_{\bs{y}}\rangle$. It follows from the Cauchy--Schwarz inequality that
\begin{align}\label{eq:bound.extreme}
    \left|\left\langle \bs{v}_{\bs{y}},\partial^{\bs{\beta}}\bs{v}_{\bs{y}}\right\rangle\right|{}&\leq \left\lVert \bs{v}_{\bs{y}} \right\rVert \left\lVert \partial^{\bs{\beta}}\bs{v}_{\bs{y}} \right\rVert\nonumber\\
    {}&\leq \left\lVert \bs{v}_{\bs{y}} \right\rVert \left\lVert \bs{\Sigma}_{\bs{\varepsilon}}^{-\frac{1}{2}} \right\rVert\frac{ A_{\bs{G}}|b|!}{\eta_{\bs{G}}^{|b|}}.
\end{align}
For the remaining terms in~\eqref{eq:general.Leibniz.rule}, it follows again from the Cauchy--Schwarz inequality and the bound in~\eqref{eq:bound.asu.5} that
\begin{equation}\label{eq:bound.middle}
    \left|\left\langle \partial^{\bs{\nu}}\bs{v}_{\bs{y}},\partial^{\bs{\beta}-\bs{\nu}}\bs{v}_{\bs{y}}\right\rangle\right|\leq \left\lVert \bs{\Sigma}_{\bs{\varepsilon}}^{-\frac{1}{2}} \right\rVert^2\frac{ A_{\bs{G}}^2\lVert\bs{\nu}\rVert_1!\lVert\bs{\beta}-\bs{\nu}\rVert_1!}{\eta_{\bs{G}}^{|b|}},
\end{equation}
where $|b|=\lVert\bs{\nu}\rVert_1+\lVert\bs{\beta}-\bs{\nu}\rVert_1$. We now combine the bounds in~\eqref{eq:bound.extreme} and~\eqref{eq:bound.middle} to obtain that
\begin{align}\label{eq:bound.gen.Leibniz}
    {}&\left|\frac{1}{2}\sum_{\bs{\nu}\leq\bs{\beta}}\binom{\bs{\beta}}{\bs{\nu}}\left\langle \partial^{\bs{\nu}}\bs{v}_{\bs{y}},\partial^{\bs{\beta}-\bs{\nu}}\bs{v}_{\bs{y}}\right\rangle\right|\nonumber\\
    {}&\leq\left|\left\langle \bs{v}_{\bs{y}},\partial^{\bs{\beta}}\bs{v}_{\bs{y}}\right\rangle\right|+\frac{1}{2}\sum_{\bs{0}<\bs{\nu}<\bs{\beta}}\binom{\bs{\beta}}{\bs{\nu}}\left|\left\langle \partial^{\bs{\nu}}\bs{v}_{\bs{y}},\partial^{\bs{\beta}-\bs{\nu}}\bs{v}_{\bs{y}}\right\rangle\right|\nonumber\\
    {}&\leq \left\lVert \bs{v}_{\bs{y}} \right\rVert \left\lVert \bs{\Sigma}_{\bs{\varepsilon}}^{-\frac{1}{2}} \right\rVert\frac{ A_{\bs{G}}|b|!}{\eta_{\bs{G}}^{|b|}}+\frac{1}{2}\left\lVert \bs{\Sigma}_{\bs{\varepsilon}}^{-\frac{1}{2}} \right\rVert^2\frac{ A_{\bs{G}}^2}{\eta_{\bs{G}}^{|b|}}\sum_{\bs{0}<\bs{\nu}<\bs{\beta}}\binom{\bs{\beta}}{\bs{\nu}}\lVert\bs{\nu}\rVert_1!\lVert\bs{\beta}-\bs{\nu}\rVert_1!.
\end{align}
For convenience, we simply add the two terms $\bs{\nu}=\bs{0}$ and $\bs{\nu}=\bs{\beta}$ to the sum in the above expression to obtain that
\begin{align}
    \sum_{\bs{0}<\bs{\nu}<\bs{\beta}}\binom{\bs{\beta}}{\bs{\nu}}\lVert\bs{\nu}\rVert_1!\lVert\bs{\beta}-\bs{\nu}\rVert_1!{}&\leq \sum_{\bs{\nu}\leq\bs{\beta}}\binom{\bs{\beta}}{\bs{\nu}}\lVert\bs{\nu}\rVert_1!\lVert\bs{\beta}-\bs{\nu}\rVert_1!\nonumber\\
    {}&=\sum_{j=0}^{|b|}\left(\sum_{\left\lVert\bs{\nu}\right\rVert_1=j}\binom{\bs{\beta}}{\bs{\nu}}\lVert\bs{\nu}\rVert_1!\lVert\bs{\beta}-\bs{\nu}\rVert_1!\right)\nonumber\\
    {}&=\sum_{j=0}^{|b|}\left(j!(|b|-j)!\sum_{\left\lVert\bs{\nu}\right\rVert_1=j}\binom{\bs{\beta}}{\bs{\nu}}\right).
\end{align}
Now, we use Vandermonde's identity to find that
\begin{align}
    \sum_{j=0}^{|b|}\left(j!(|b|-j)!\sum_{\left\lVert\bs{\nu}\right\rVert_1=j}\binom{\bs{\beta}}{\bs{\nu}}\right){}&=\sum_{j=0}^{|b|}\left(j!(|b|-j)!\binom{|b|}{j}\right)\nonumber\\
    {}&=\sum_{j=0}^{|b|}|b|!\nonumber\\
    {}&=(|b|+1)|b|!\nonumber\\
    {}&\leq 2^{|b|}|b|!.
\end{align}
Substituting into~\eqref{eq:bound.gen.Leibniz} yields
\begin{align}\label{eq:comb.bound}
    \left|\frac{1}{2}\sum_{\bs{\nu}\leq\bs{\beta}}\binom{\bs{\beta}}{\bs{\nu}}\left\langle \partial^{\bs{\nu}}\bs{v}_{\bs{y}},\partial^{\bs{\beta}-\bs{\nu}}\bs{v}_{\bs{y}}\right\rangle\right|{}&\leq\left\lVert \bs{v}_{\bs{y}} \right\rVert \left\lVert \bs{\Sigma}_{\bs{\varepsilon}}^{-\frac{1}{2}} \right\rVert\frac{ A_{\bs{G}}|b|!}{\eta_{\bs{G}}^{|b|}}+\frac{1}{2}\left\lVert \bs{\Sigma}_{\bs{\varepsilon}}^{-\frac{1}{2}} \right\rVert^2\frac{ A_{\bs{G}}^2}{\eta_{\bs{G}}^{|b|}}2^{|b|}|b|!\nonumber\\
    {}&=\left(\left\lVert \bs{v}_{\bs{y}} \right\rVert K+2^{|b|-1}K^2\right)\frac{|b|!}{\eta_{\bs{G}}^{|b|}},
\end{align}
where
\begin{equation}
    K\coloneqq \left\lVert \bs{\Sigma}_{\bs{\varepsilon}}^{-\frac{1}{2}} \right\rVert A_{\bs{G}}.
\end{equation}
We simplify the expression in~\eqref{eq:comb.bound} further as follows:
\begin{align}\label{eq:bound.tilde}
    \left(\left\lVert \bs{v}_{\bs{y}} \right\rVert K+2^{|b|-1}K^2\right)\frac{|b|!}{\eta_{\bs{G}}^{|b|}}{}&\leq\left(\left\lVert \bs{v}_{\bs{y}} \right\rVert K+K^2\right)\frac{2^{|b|}|b|!}{\eta_{\bs{G}}^{|b|}}\nonumber\\
    {}&\leq\left(1+\left\lVert \bs{v}_{\bs{y}} \right\rVert \right)\left(K+K^2\right)\frac{2^{|b|}|b|!}{\eta_{\bs{G}}^{|b|}}\nonumber\\
    {}&=\left(1+\left\lVert \bs{v}_{\bs{y}} \right\rVert \right)\tilde{K}\frac{|b|!}{\tilde{\eta}_{\bs{G}}^{|b|}},
\end{align}
where $\tilde{K}\coloneqq K+K^2$ and $\tilde{\eta}_{\bs{G}}\coloneqq \eta_{\bs{G}}/2$. We substitute~\eqref{eq:bound.tilde} into~\eqref{eq:comb.bound},~\eqref{eq:comb.bound} into~\eqref{eq:bound.gen.Leibniz},~\eqref{eq:bound.gen.Leibniz} into~\eqref{eq:general.Leibniz.rule}, and finally~\eqref{eq:general.Leibniz.rule} into~\eqref{eq:Faa.di.Bruno.likelihood} to obtain that
\begin{align}\label{eq:Faa.di.Bruno.bound}
        \left|\frac{\partial^{n}}{\partial x_{i_1}\cdots\partial x_{i_n}}g_{\bs{y}}(\bs{x})\right|{}&\leq\sum_{\pi}e^{-\Lambda_{\bs{y}}}\prod_{b\in\pi}\left(1+\left\lVert \bs{v}_{\bs{y}} \right\rVert \right)\tilde{K}\frac{|b|!}{\tilde{\eta}_{\bs{G}}^{|b|}}\nonumber\\
        {}&=\sum_{k=1}^{n}\sum_{|\pi|=k}e^{-\Lambda_{\bs{y}}}\prod_{b\in\pi}\left(1+\left\lVert \bs{v}_{\bs{y}} \right\rVert \right)\tilde{K}\frac{|b|!}{\tilde{\eta}_{\bs{G}}^{|b|}}\nonumber\\
        {}&=\frac{1}{\tilde{\eta}_{\bs{G}}^{n}}\sum_{k=1}^{n}\left(1+\left\lVert \bs{v}_{\bs{y}} \right\rVert \right)^k\tilde{K}^ke^{-\Lambda_{\bs{y}}}\sum_{|\pi|=k}\prod_{b\in\pi}|b|!.
    \end{align}
    Let $0\leq\tau\coloneqq \left\lVert \bs{v}_{\bs{y}} \right\rVert\in\mathbb{R}_{\geq 0}$. It follows from $\Lambda_{\bs{y}}=\frac{1}{2}\left\lVert \bs{v}_{\bs{y}} \right\rVert^2$ that
    \begin{align}\label{eq:k.term.bound}
        \left(1+\left\lVert \bs{v}_{\bs{y}} \right\rVert \right)^k\tilde{K}^ke^{-\Lambda_{\bs{y}}}{}&=\left(1+\left\lVert \bs{v}_{\bs{y}} \right\rVert \right)^k\tilde{K}^ke^{-\frac{1}{2}\left\lVert \bs{v}_{\bs{y}} \right\rVert^2}\nonumber\\
        {}&=\left(1+\tau \right)^k\tilde{K}^ke^{-\frac{1}{2}\tau^2}.
    \end{align}
    Then, we find that
    \begin{align}\label{eq:sup.max}
        \left(1+\tau \right)^k\tilde{K}^ke^{-\frac{1}{2}\tau^2}{}&\leq 2^k\max\{1,\tau^k\}\tilde{K}^ke^{-\frac{1}{2}\tau^2}\nonumber\\
        {}&\leq 2^k\tilde{K}^k\sup_{\tau\geq 0}\max\{1,\tau^k\}e^{-\frac{1}{2}\tau^2}\nonumber\\
        {}&= 2^k\tilde{K}^k\max\left\{1,\left(\frac{k}{e}\right)^{\frac{k}{2}}\right\}.
    \end{align}
    In the above derivation, we used that the supremum over $\tau^{k}e^{-\tau^2/2}$ is attained at $\tau=\sqrt{k}$ independent of $\bs{y}\in[0,1]^{d_1}$. It follows from the series representation of the exponential function that
    \begin{align}
        \left(\frac{k}{e}\right)^{\frac{k}{2}}\leq \sqrt{k!}.
    \end{align}
 For $k\geq1$ it follows that $\max\{1,(k/e)^{k/2}\}\leq \sqrt{k!}$ and thus it follows from~\eqref{eq:sup.max} that
        \begin{align}\label{eq:series.representation}
        \left(1+\tau \right)^k\tilde{K}^ke^{-\frac{1}{2}\tau^2}{}&\leq  2^k\tilde{K}^k\sqrt{k!}\nonumber\\
        {}&\leq 2^k\tilde{K}^kk!.
    \end{align}
    Substituting~\eqref{eq:series.representation} into~\eqref{eq:k.term.bound} and subsequently into~\eqref{eq:Faa.di.Bruno.bound}, we find that
    \begin{align}
        \left|\frac{\partial^{n}}{\partial x_{i_1}\cdots\partial x_{i_n}}g_{\bs{y}}(\bs{x})\right|{}&\leq \frac{1}{\tilde{\eta}_{\bs{G}}^{n}}\sum_{k=1}^{n}2^k\tilde{K}^kk!\sum_{|\pi|=k}\prod_{b\in\pi}|b|!.
    \end{align}
    We conclude by using the identity
    \begin{equation}
        \sum_{|\pi|=k}\prod_{b\in\pi}|b|!=L(n,k),
    \end{equation}
    where
        \begin{equation}
        L(n,k)=\frac{n!}{k!}\binom{n-1}{k-1},
    \end{equation}
    is the Lah number~\cite{Rio02} to obtain
    \begin{align}
        \left|\frac{\partial^{n}}{\partial x_{i_1}\cdots\partial x_{i_n}}g_{\bs{y}}(\bs{x})\right|{}&\leq \frac{n!}{\tilde{\eta}_{\bs{G}}^{n}}\sum_{k=1}^{n}\binom{n-1}{k-1}\left(2\tilde{K}\right)^k\nonumber\\
        {}&=\frac{n!}{\tilde{\eta}_{\bs{G}}^{n}}2\tilde{K}\sum_{k=0}^{n-1}\binom{n-1}{k}\left(2\tilde{K}\right)^k\nonumber\\
        {}&=\frac{n!}{\tilde{\eta}_{\bs{G}}^{n}}2\tilde{K}\left(1+2\tilde{K}\right)^{n-1}\nonumber\\
        {}&\leq\frac{n!}{\tilde{\eta}_{\bs{G}}^{n}}2\tilde{K}\left(1+2\tilde{K}\right)^{n}\nonumber\\
        {}&=\frac{A_gn!}{\eta_{g}^{n}},
    \end{align}
    where $A_g\coloneqq\max\{1,2\tilde{K}\}$ and $\eta_g\coloneqq\tilde{\eta}_{\bs{G}}/(1+2\tilde{K})$. The case where $n=0$ follows from $g_{\bs{y}}\leq 1$.
\end{proof}

\section{Weighted inner-error bound}\label{app:weighted}
We provide the proof of Lemma~\ref{lem:weighted}.
\begin{proof}
    First, note that $I[g(\bs{y},\cdot)]\geq g^{\ast}(\bs{y})>0$ for all $\bs{y}\in\Omega_{>0}$ as the integration domain $[0,1]^{d_2}$ is compact. Next, assume that $I[g(\bs{y},\cdot)]\geq 2b$ for $0<b<1/2$. It follows that
    \begin{equation}
        \left|\frac{K}{I[g(\bs{y},\cdot)]}-1\right|\leq \frac{b}{I[g(\bs{y},\cdot)]}\leq \frac{1}{2},
    \end{equation}
    and thus it follows that $K/I[g(\bs{y},\cdot)]\in[1/2,3/2]$. From Assumption~\ref{asu:weighted.compatibility}, it follows that
    \begin{align}
        \left|f\left(K\right)-f\left(I[g(\bs{y},\cdot)]\right)\right|{}&\leq L_{f}\left|\log(K)-\log\left(I[g(\bs{y},\cdot)]\right)\right|\nonumber\\
        {}&=L_{f}\left|\log\left(\frac{K}{I[g(\bs{y},\cdot)]}\right)\right|.
    \end{align}
    We use the inequality $|\log(t)|\leq 2|t-1|$, which holds for all $t\geq 1/2$, to obtain that:
        \begin{align}
     L_{f}\left|\log\left(\frac{K}{I[g(\bs{y},\cdot)]}\right)\right|{}&\leq L_{f}2\left|\frac{K}{I[g(\bs{y},\cdot)]}-1\right|\nonumber\\
     {}&\leq \frac{2L_{f}b}{I[g(\bs{y},\cdot)]}.
    \end{align}
    It follows from Assumption~\ref{asu:weighted.compatibility} that
    \begin{align}\label{eq:2b}
        w(\bs{y})\left|f(K)-f\left(I[g(\bs{y},\cdot)]\right)\right|{}&\leq \frac{2L_{f}bA_{\rho}(g^{\ast}(\bs{y}))^{\rho}}{I[g(\bs{y},\cdot)]}\nonumber\\
        {}&\leq \frac{2L_{f}bA_{\rho}(I[g(\bs{y},\cdot)])^{\rho}}{I[g(\bs{y},\cdot)]}\nonumber\\
        {}&= 2L_{f}bA_{\rho}(I[g(\bs{y},\cdot)])^{\rho-1}\nonumber\\
        {}&\leq 2L_{f}bA_{\rho}(2b)^{\rho-1}\nonumber\\
        {}&=L_{f}A_{\rho}(2b)^{\rho},
    \end{align}
    where we used that $0<\rho<1$ and thus $\rho-1<0$. Next, we assume that $I[g(\bs{y},\cdot)]<2b$. It follows that $g^{\ast}(\bs{y})\leq I[g(\bs{y},\cdot)]< 2b$ and $K\leq I[g(\bs{y},\cdot)]+b< 3b$; thus, it follows that $K,I[g(\bs{y},\cdot)]\in[g^{\ast}(\bs{y}),3b]$. Therefore,
    \begin{align}
        \left|f(K)-f\left(I[g(\bs{y},\cdot)]\right)\right|{}&\leq L_{f}\left|\log(K)-\log\left(I[g(\bs{y},\cdot)]\right)\right|\nonumber\\
        {}&\leq L_{f}\log\left(\frac{3b}{g^{\ast}(\bs{y})}\right).
    \end{align}
    It follows that
        \begin{align}\label{eq:g.log.g}
        w(\bs{y})\left|f(K)-f\left(I[g(\bs{y},\cdot)]\right)\right|\leq A_{\rho}\left(g^{\ast}(\bs{y})\right)^{\rho}L_{f}\log\left(\frac{3b}{g^{\ast}(\bs{y})}\right).
    \end{align}
    Let $u\coloneqq g^{\ast}(\bs{y})\in(0,2b]$ and $\psi(u)\coloneqq u^{\rho}\log(3b/u)$. It follows that
    \begin{equation}\label{eq:u.log.1}
        \psi'(u)=u^{\rho-1}\left(\rho\log\left(\frac{3b}{u}\right)-1\right)
    \end{equation}
    and
    \begin{equation}
        \psi''(u)=(\rho-1)u^{\rho-2}\left(\rho\log\left(\frac{3b}{u}\right)-1\right)-\rho u^{\rho-2}.
    \end{equation}
    We have that $\psi'(u^{\ast})=0$ for $u^{\ast}=3be^{-1/\rho}$ and $\psi''(u^{\ast})=-\rho(3b)^{\rho-2}e^{-(\rho-2)/\rho}<0$, so $\psi$ is maximized at $u^{\ast}.$ Moreover, we have that $e^{-1/\rho}\leq e^{-1}<2/3$ and thus $0<u^{\ast}\leq 2b$, so the maximizer is in the admissible range. It follows that
    \begin{align}\label{eq:u.log.2}
        \psi(u)\leq \psi(u^{\ast})=\frac{(3b)^{\rho}}{e\rho},
    \end{align}
    and thus it follows from~\eqref{eq:g.log.g} that
    \begin{equation}\label{eq:3b}
        w(\bs{y})\left|f(K)-f\left(I[g(\bs{y},\cdot)]\right)\right|\leq L_{f}A_{\rho}\frac{(3b)^{\rho}}{e\rho}.
    \end{equation}
   Combining the results in~\eqref{eq:2b} and~\eqref{eq:3b} and setting
       \begin{equation}
        C_{\rho}\coloneqq L_{f}A_{\rho}\max\left\{2^{\rho},\frac{3^{\rho}}{e\rho}\right\}.
    \end{equation}
   completes the proof.
\end{proof}

\section{Weighted inner-error bound for the EIG}\label{app:weighted.EIG}
Here, we provide the proof of Lemma~\ref{lem:weighted.compatibility.EIG}.
\begin{proof}
    Let $\bs{z}\coloneqq\bs{\Phi}^{-1}(\bs{y}_2)$ and $t\coloneqq\left\lVert \bs{z}\right\rVert$ for every $\bs{y}_2\in(0,1)^{d_1-d_2}$. It follows that
    \begin{align}\label{eq:bound.w.1}
        w_{\bs{\gamma}}(\bs{y}_2){}&=\prod_{i=1}^{d_1-d_2}\gamma_ie^{-\frac{1}{2}z_i^2(\gamma_i^2-1)}\nonumber\\
        {}&\leq e^{-\frac{1}{2}t^2(\gamma_{\rm{min}}^2-1)}\prod_{i=1}^{d_1-d_2}\gamma_i.
    \end{align}
    Next, we demonstrate that
    \begin{equation}\label{eq:bound.exponent}
        -\frac{1}{2}t^2(\gamma_{\rm{min}}^2-1)\leq -\frac{\rho}{2}\left(U+\gamma_{\rm{\max}}t\right)^2+\frac{\rho^2U^2\gamma_{\rm{max}}^2}{2((\gamma_{\rm{min}}^2-1)-\rho\gamma_{\rm{max}}^2)}+\frac{\rho U^2}{2}
    \end{equation}
    for all $t\geq 0$. Let
    \begin{align}
        \psi(t){}&\coloneqq -\frac{1}{2}t^2(\gamma_{\rm{min}}^2-1)+\frac{\rho}{2}\left(U+\gamma_{\rm{\max}}t\right)^2\nonumber\\
        {}&=-\left(\frac{(\gamma_{\rm{min}}^2-1)-\rho\gamma_{\rm{max}}^2}{2}\right)t^2+\rho U\gamma_{\rm{max}}t+\frac{\rho U^2}{2}.
    \end{align}
    It follows that
    \begin{equation}
        \psi'(t)=-\left((\gamma_{\rm{min}}^2-1)-\rho\gamma_{\rm{max}}^2\right)t+\rho U\gamma_{\rm{max}}
    \end{equation}
    and
        \begin{equation}
        \psi''(t)=-\left((\gamma_{\rm{min}}^2-1)-\rho\gamma_{\rm{max}}^2\right).
    \end{equation}
    For $\rho<\rho^{\ast}=(\gamma_{\rm{min}}^2-1)/\gamma_{\rm{max}}^2$ $\psi''<0$ and for
    \begin{equation}
        t^{\ast}\coloneqq \frac{\rho U\gamma_{\rm{max}}}{(\gamma_{\rm{min}}^2-1)-\rho\gamma_{\rm{max}}^2}
    \end{equation}
    we have that $\psi'(t^{\ast})=0$,
    and thus, $\psi$ attains its maximum at $t^{\ast}$. It follows that $\psi(t)\leq \psi(t^{\ast})$, that is,
    \begin{equation}
        -\left(\frac{(\gamma_{\rm{min}}^2-1)-\rho\gamma_{\rm{max}}^2}{2}\right)t^2+\rho U\gamma_{\rm{max}}t+\frac{\rho U^2}{2}\leq \frac{\rho^2U^2\gamma_{\rm{max}}^2}{2\left((\gamma_{\rm{min}}^2-1)-\rho\gamma_{\rm{max}}^2\right)}+\frac{\rho U^2}{2},
    \end{equation}
    which is equivalent to~\eqref{eq:bound.exponent}. It follows that
    \begin{align}\label{eq:bound.w.2}
        e^{-\frac{1}{2}t^2(\gamma_{\rm{min}}^2-1)}\prod_{i=1}^{d_1-d_2}\gamma_i{}&\leq e^{-\frac{\rho}{2}\left(U+\gamma_{\rm{\max}}t\right)^2}e^{\frac{\rho^2U^2\gamma_{\rm{max}}^2}{2((\gamma_{\rm{min}}^2-1)-\rho\gamma_{\rm{max}}^2)}+\frac{\rho U^2}{2}}\prod_{i=1}^{d_1-d_2}\gamma_i\nonumber\\
        {}&\eqqcolon e^{-\frac{\rho}{2}\left(U+\gamma_{\rm{\max}}t\right)^2}A_{\rho}
    \end{align}
    for
    \begin{equation}
        A_{\rho}=e^{\frac{\rho^2U^2\gamma_{\rm{max}}^2}{2((\gamma_{\rm{min}}^2-1)-\rho\gamma_{\rm{max}}^2)}+\frac{\rho U^2}{2}}\prod_{i=1}^{d_1-d_2}\gamma_i.
    \end{equation}
    Finally, it follows from
    \begin{align}
        \sup_{\bs{x}\in[0,1]^{d_2}}\Lambda_{\bs{\gamma}}(\bs{y},\bs{x}){}&=\sup_{\bs{x}\in[0,1]^{d_2}}\frac{1}{2}\left\lVert \bs{\Sigma}_{\bs{\varepsilon}}^{-\frac{1}{2}}(\bs{G}(\bs{y}_1)-\bs{G}(\bs{x}))+\bs{D}_{\bs{\gamma}}\bs{z}\right\rVert^2\nonumber\\
        {}&\leq \sup_{\bs{x}\in[0,1]^{d_2}}\frac{1}{2}\left(\left\lVert \bs{\Sigma}_{\bs{\varepsilon}}^{-\frac{1}{2}}(\bs{G}(\bs{y}_1)-\bs{G}(\bs{x}))\right\rVert+\left\lVert \bs{D}_{\bs{\gamma}}\bs{z}\right\rVert\right)^2\nonumber\\
        {}&\leq\frac{1}{2}\left(U+\left\lVert \bs{D}_{\bs{\gamma}}\bs{z}\right\rVert\right)^2\nonumber\\
        {}&\leq\frac{1}{2}\left(U+\gamma_{\rm{\max}}t\right)^2\nonumber\\
    \end{align}
    that
    \begin{align}\label{eq:bound.w.3}
        e^{-\frac{\rho}{2}\left(U+\gamma_{\rm{\max}}t\right)^2}{}&\leq e^{-\rho\sup_{\bs{x}}\Lambda_{\bs{\gamma}}(\bs{y},\bs{x})}\nonumber\\
        {}&=\inf_{\bs{x}\in[0,1]^{d_2}}e^{-\rho\Lambda_{\bs{\gamma}}(\bs{y},\bs{x})}\nonumber\\
        {}&=\left(g^{\ast}(\bs{y})\right)^{\rho}.
    \end{align}
    Combining~\eqref{eq:bound.w.1} with~\eqref{eq:bound.w.2} and~\eqref{eq:bound.w.3} completes the proof.
\end{proof}

\section{Mixed derivatives of the composite integrand with importance sampling}\label{app:der.int.is}
We provide the proof of Lemma~\ref{lem:mixed.derivatives.is}. 
\begin{proof}
To demonstrate the boundary behavior of the integrand with importance sampling, we consider $d_1=2$ for simplicity. Thus, let $\bs{y}=(y_1,y_2)$, and
\begin{align}\label{eq:int.is}
    h_{\rm{IS}}(\bs{y}){}&=\log\left(\int_{[0,1]}e^{-\Lambda_{\gamma}(\bs{y},x)}\di{}x\right)w_{\gamma}(y_2)\nonumber\\
    {}&=\log\left(\int_{[0,1]}e^{-\frac{1}{2}\left(\frac{1}{\sigma_{\bs{\varepsilon}}}( G(y_1)-G(x))+\gamma\Phi^{-1}(y_2)\right)^2}\di{}x\right)\gamma e^{-\frac{1}{2}(\Phi^{-1}(y_{2}))^2(\gamma^2-1)}.
\end{align}
Then, it follows from the product rule that
\begin{align}\label{eq:part.int.is}
    \frac{\partial^2}{\partial y_1 y_2}h_{\rm{IS}}(\bs{y}){}&=\frac{\partial}{\partial y_1}\frac{\frac{\partial}{\partial y_2}\int_{[0,1]}e^{-\Lambda_{\gamma}(\bs{y},x)}\di{}x}{\int_{[0,1]}e^{-\Lambda_{\gamma}(\bs{y},x)}\di{}x}w_{\gamma}(y_2) +\frac{\frac{\partial}{\partial y_1}\int_{[0,1]}e^{-\Lambda_{\gamma}(\bs{y},x)}\di{}x}{\int_{[0,1]}e^{-\Lambda_{\gamma}(\bs{y},x)}\di{}x}\frac{\partial}{\partial y_2}w_{\gamma}(y_2).
\end{align}
For the first term in~\eqref{eq:part.int.is}, it follows from the quotient rule that
\begin{align}\label{eq:part.int.is.I}
    {}&\frac{\partial}{\partial y_1}\frac{\frac{\partial}{\partial y_2}\int_{[0,1]}e^{-\Lambda_{\gamma}(\bs{y},x)}\di{}x}{\int_{[0,1]}e^{-\Lambda_{\gamma}(\bs{y},x)}\di{}x}w_{\gamma}(y_2)\nonumber\\
    {}&=\left(\frac{\frac{\partial^2}{\partial y_1\partial y_2}\int_{[0,1]}e^{-\Lambda_{\gamma}(\bs{y},x)}\di{}x}{\int_{[0,1]}e^{-\Lambda_{\gamma}(\bs{y},x)}\di{}x}-\frac{\frac{\partial}{\partial y_1}\int_{[0,1]}e^{-\Lambda_{\gamma}(\bs{y},x)}\di{}x\frac{\partial}{\partial y_2}\int_{[0,1]}e^{-\Lambda_{\gamma}(\bs{y},x)}\di{}x}{\left(\int_{[0,1]}e^{-\Lambda_{\gamma}(\bs{y},x)}\di{}x\right)^2}\right)w_{\gamma}(y_2).
\end{align}
Again, for the first part in~\eqref{eq:part.int.is.I}, it follows that
\begin{align}\label{eq:first.term.quotient}
    {}&\frac{\frac{\partial^2}{\partial y_1\partial y_2}\int_{[0,1]}e^{-\Lambda_{\gamma}(\bs{y},x)}\di{}x}{\int_{[0,1]}e^{-\Lambda_{\gamma}(\bs{y},x)}\di{}x}\nonumber\\
    {}&=\frac{\int_{[0,1]}\left(-\frac{\partial^2}{\partial y_1\partial y_2}\Lambda_{\gamma}(\bs{y},x)+\frac{\partial}{\partial y_1}\Lambda_{\gamma}(\bs{y},x)\frac{\partial}{\partial y_2}\Lambda_{\gamma}(\bs{y},x)\right) e^{-\Lambda_{\gamma}(\bs{y},x)}\di{}x}{\int_{[0,1]}e^{-\Lambda_{\gamma}(\bs{y},x)}\di{}x}.
\end{align}
Next, it follows for $\Lambda_{\gamma}(\bs{y},x)=\frac{1}{2}\left(\frac{1}{\sigma_{\bs{\varepsilon}}}( G(y_1)-G(x))+\gamma\Phi^{-1}(y_2)\right)^2$ that
\begin{equation}
    \left|\frac{\partial}{\partial y_1}\Lambda_{\gamma}(\bs{y},x)\right|=\sqrt{2\Lambda_{\gamma}(\bs{y},x)}\frac{1}{\sigma_{\bs{\varepsilon}}}\left|\frac{\rm{d}}{{\rm{d}}y_1}G(y_1)\right|.
\end{equation}
Note that
\begin{align}\label{eq:inverse.function.theorem}
    \frac{\di{}}{\di{} y_2}\Phi^{-1}(y_2){}&=\left(\Phi^{\prime}(\Phi^{-1}(y_2))\right)^{-1},\nonumber\\
{}&=\sqrt{2\pi}e^{\frac{1}{2}(\Phi^{-1}(y_2))^2},
\end{align}
and thus
\begin{align}
    \left|\frac{\partial}{\partial y_2}\Lambda_{\gamma}(\bs{y},x)\right|{}&=\sqrt{2\Lambda_{\gamma}(\bs{y},x)}\gamma\frac{\rm{d}}{{\rm{d}}y_2}\Phi^{-1}(y_2)\nonumber\\
    {}&=\sqrt{2\Lambda_{\gamma}(\bs{y},x)}\gamma\sqrt{2\pi}e^{\frac{1}{2}(\Phi^{-1}(y_2))^2}.
\end{align}
Finally, it follows that
\begin{align}
    \frac{\partial^2}{\partial y_1\partial y_2}\Lambda_{\gamma}(\bs{y},x){}&=\frac{1}{\sigma_{\bs{\varepsilon}}}\frac{\rm{d}}{{\rm{d}}y_1}G(y_1)\gamma\sqrt{2\pi}e^{\frac{1}{2}(\Phi^{-1}(y_2))^2}.
\end{align}
From Assumption~\ref{asu:model}, it follows that there exists $L_G<\infty$ such that $G$ is Lipschitz continuous with Lipschitz constant $L_G$. Thus, it follows that there exists $T<\infty$ independent of $y_2$ such that
\begin{align}
    \sqrt{2\Lambda_{\gamma}(\bs{y},x)}{}&\leq \frac{1}{\sigma_{\bs{\varepsilon}}}L_G+\gamma\left|\Phi^{-1}(y_2)\right|\nonumber\\
    {}&\leq T(1+\left|\Phi^{-1}(y_2)\right|),
\end{align}
and that
\begin{equation}
    \left|\frac{\rm{d}}{{\rm{d}}y_1}G(y_1)\right|\leq \frac{A_G}{\eta_G}.
\end{equation}
It follows from the above derivations that
\begin{align}\label{eq:first.term.mixed}
    \left|\frac{\partial^2}{\partial y_1\partial y_2}\Lambda_{\gamma}(\bs{y},x)\right|\leq \frac{1}{\sigma_{\bs{\varepsilon}}}\frac{A_G}{\eta_G}\gamma\sqrt{2\pi}e^{\frac{1}{2}(\Phi^{-1}(y_2))^2},
\end{align}
and that
\begin{equation}\label{eq:first.term.product}
    \left|\frac{\partial}{\partial y_1}\Lambda_{\gamma}(\bs{y},x)\frac{\partial}{\partial y_2}\Lambda_{\gamma}(\bs{y},x)\right|\leq T^2(1+\left|\Phi^{-1}(y_2)\right|)^2\frac{1}{\sigma_{\bs{\varepsilon}}}\frac{A_G}{\eta_G}\gamma\sqrt{2\pi}e^{\frac{1}{2}(\Phi^{-1}(y_2))^2}.
\end{equation}
In particular, the bounds in~\eqref{eq:first.term.mixed} and in~\eqref{eq:first.term.product} are independent of $x$, and thus it follows for the first term in~\eqref{eq:part.int.is.I} that
\begin{align}\label{eq:first.term.first.term}
    {}&\frac{\int_{[0,1]}\left(-\frac{\partial^2}{\partial y_1\partial y_2}\Lambda_{\gamma}(\bs{y},x)+\frac{\partial}{\partial y_1}\Lambda_{\gamma}(\bs{y},x)\frac{\partial}{\partial y_2}\Lambda_{\gamma}(\bs{y},x)\right) e^{-\Lambda_{\gamma}(\bs{y},x)}\di{}x}{\int_{[0,1]}e^{-\Lambda_{\gamma}(\bs{y},x)}\di{}x}\nonumber\\
    {}&\leq \frac{1}{\sigma_{\bs{\varepsilon}}}\frac{A_G}{\eta_G}\gamma\sqrt{2\pi}e^{\frac{1}{2}(\Phi^{-1}(y_2))^2}\left(1+T^2(1+\left|\Phi^{-1}(y_2)\right|)^2\right).
\end{align}
Similarly, it follows for the second term in~\eqref{eq:part.int.is.I} that
\begin{align}\label{eq:first.term.second.term}
    {}&\frac{\frac{\partial}{\partial y_1}\int_{[0,1]}e^{-\Lambda_{\gamma}(\bs{y},x)}\di{}x\frac{\partial}{\partial y_2}\int_{[0,1]}e^{-\Lambda_{\gamma}(\bs{y},x)}\di{}x}{\left(\int_{[0,1]}e^{-\Lambda_{\gamma}(\bs{y},x)}\di{}x\right)^2}\nonumber\\
    {}&\leq T^2(1+\left|\Phi^{-1}(y_2)\right|)^2\frac{1}{\sigma_{\bs{\varepsilon}}}\frac{A_G}{\eta_G}\gamma\sqrt{2\pi}e^{\frac{1}{2}(\Phi^{-1}(y_2))^2}.
\end{align}
Combining~\eqref{eq:first.term.first.term} and~\eqref{eq:first.term.second.term}, we see that the term in~\eqref{eq:part.int.is.I} is bounded as follows:
\begin{align}\label{eq:f.first.term}
    {}&\frac{\partial}{\partial y_1}\frac{\frac{\partial}{\partial y_2}\int_{[0,1]}e^{-\Lambda_{\gamma}(\bs{y},x)}\di{}x}{\int_{[0,1]}e^{-\Lambda_{\gamma}(\bs{y},x)}\di{}x}w_{\gamma}(y_2)\nonumber\\
    {}&\leq \frac{1}{\sigma_{\bs{\varepsilon}}}\frac{A_G}{\eta_G}\gamma\sqrt{2\pi}e^{\frac{1}{2}(\Phi^{-1}(y_2))^2}\left(1+2T^2(1+\left|\Phi^{-1}(y_2)\right|)^2\right)w_{\gamma}(y_2)\nonumber\\
    {}&=\frac{1}{\sigma_{\bs{\varepsilon}}}\frac{A_G}{\eta_G}\sqrt{2\pi}\left(1+2T^2(1+\left|\Phi^{-1}(y_2)\right|)^2\right) \gamma^2e^{-\frac{1}{2}(\Phi^{-1}(y_{2}))^2(\gamma^2-2)}
\end{align}
for $w_{\gamma}(y_2)=\gamma e^{-\frac{1}{2}(\Phi^{-1}(y_{2}))^2(\gamma^2-1)}$. This term has no singularity at the boundary for $\gamma>\sqrt{2}$, a condition more strict than $\gamma>1$, which was necessary to ensure that the integrand itself has no boundary singularity. For last term in~\eqref{eq:part.int.is}, it follows that
\begin{align}
    {}&\frac{\frac{\partial}{\partial y_1}\int_{[0,1]}e^{-\Lambda_{\gamma}(\bs{y},x)}\di{}x}{\int_{[0,1]}e^{-\Lambda_{\gamma}(\bs{y},x)}\di{}x}\frac{\partial}{\partial y_2}w_{\gamma}(y_2)\nonumber\\
    {}&\leq T(1+\left|\Phi^{-1}(y_2)\right|)\frac{1}{\sigma_{\bs{\varepsilon}}}\frac{A_G}{\eta_G}\frac{\partial}{\partial y_2}w_{\gamma}(y_2)\nonumber\\
    {}&\leq T(1+\left|\Phi^{-1}(y_2)\right|)\frac{1}{\sigma_{\bs{\varepsilon}}}\frac{A_G}{\eta_G}\gamma \left|\Phi^{-1}(y_{2})\right|\sqrt{2\pi}e^{\frac{1}{2}(\Phi^{-1}(y_2))^2}(\gamma^2-1)e^{-\frac{1}{2}(\Phi^{-1}(y_{2}))^2(\gamma^2-1)}\nonumber\\
    {}&= T(1+\left|\Phi^{-1}(y_2)\right|)\frac{1}{\sigma_{\bs{\varepsilon}}}\frac{A_G}{\eta_G}\gamma \left|\Phi^{-1}(y_{2})\right|\sqrt{2\pi}(\gamma^2-1)e^{-\frac{1}{2}(\Phi^{-1}(y_{2}))^2(\gamma^2-2)}.
\end{align}
This term also has no singularity at the boundary for $\gamma>\sqrt{2}=\sqrt{\alpha+2}$ for $\alpha=0$, similar to the term in~\eqref{eq:f.first.term}. The above derivation demonstrates the case for $d_1=2$ and $\bs{\beta}=(1,1)$. We now provide the argument for arbitrary $d_1$, arbitrary $\alpha\in\mathbb{N}\cup\{0\}$, and $\bs{\beta}\in\{0,1,\ldots,\alpha+1\}^{d_1}$.

First, we write $\bs{z}=\bs{\Phi}^{-1}(\bs{y}_2)$ with components $z_i=\Phi^{-1}(y_{2,i})$, $1\leq i\leq d_1-d_2$. Next, we write $\bs{\beta}=(\bs{\beta}_1,\bs{\beta}_2)$, where $\bs{\beta}_1$ is a multi-index in $\mathbb{N}_0^{d_2}$ and $\bs{\beta}_2$ is a multi-index in $\mathbb{N}_0^{d_1-d_2}$ and abbreviate
\begin{equation}
   J(\bs{y})\coloneqq\int_{[0,1]^{d_2}}e^{-\Lambda_{\bs{\gamma}}(\bs{y},\bs{x})}\di{}\bs{x},
\end{equation}
and
\begin{equation}
   \ell(\bs{y})\coloneqq\log \left(J(\bs{y})\right),
\end{equation}
such that
\begin{equation}
 h_{\rm{IS}}=\ell\,w_{\bs{\gamma}}.
\end{equation}
Let $\Pi$ denote a polynomial in $\lVert\bs{z}\rVert_{\infty}$ with coefficients depending only on $d_1$, $d_2$, $\bs{\gamma}$, $A_{\bs{G}}$, $\eta_{\bs{G}}$ and $\bs{\Sigma}_{\bs{\varepsilon}}$. We aim to show that:
\begin{equation}\label{eq:is.master}
   \left|\partial^{\bs{\beta}}h_{\rm{IS}}(\bs{y})\right|\leq\Pi(\bs{z})\prod_{i=1}^{d_1-d_2}e^{-\frac{1}{2}z_i^2\left(\gamma_i^2-1-\beta_{2,i}\right)}.
\end{equation}

We begin by differentiating $\frac{\di{}}{\di{}y_{2,i}}\Phi^{-1}(y_{2,i})=\sqrt{2\pi}e^{\frac{1}{2}z^2}$ repeatedly and, using the chain rule, obtain by induction that
\begin{equation}\label{eq:phiinv.k}
   \frac{\di{}^{k}}{\di{}y_{2,i}^{k}}\Phi^{-1}(y_{2,i})=(2\pi)^{\frac{k}{2}}P_k(z)e^{\frac{k}{2}z^2},
\end{equation}
where
\begin{equation}
 P_1\equiv 1,\quad P_{k+1}=P_k'+kzP_k.
\end{equation}
For $k=1$,~\eqref{eq:phiinv.k} follows from~\eqref{eq:inverse.function.theorem} with $P_1\equiv 1$. Assume that~\eqref{eq:phiinv.k} holds for some $k\geq 1$. Its right-hand side depends on $y_{2,i}$ only through $z$, so the chain rule gives
\begin{align}
   \frac{\di{}^{k+1}}{\di{}y_{2,i}^{k+1}}\Phi^{-1}(y_{2,i})
   {}&=\frac{\di{}}{\di{}z}\left((2\pi)^{\frac{k}{2}}P_k(z)e^{\frac{k}{2}z^2}\right)\frac{\di{}z}{\di{}y_{2,i}}\nonumber\\
   {}&=(2\pi)^{\frac{k}{2}}\left(P_k'(z)+kzP_k(z)\right)e^{\frac{k}{2}z^2}\sqrt{2\pi}e^{\frac{1}{2}z^2}\nonumber\\
   {}&=(2\pi)^{\frac{k+1}{2}}\left(P_k'(z)+kzP_k(z)\right)e^{\frac{k+1}{2}z^2},
\end{align}
which demonstrates~\eqref{eq:phiinv.k} at order $k+1$ with $P_{k+1}=P_k'+kzP_k$, completing the induction. It follows that $\deg (P_k')=\deg (P_k)-1$ while $\deg(kzP_k)=\deg (P_k)+1$, so the second term dominates and $\deg (P_{k+1})=\deg (P_k)+1$; with $\deg (P_1)=0$ this gives $\deg (P_k)=k-1$. The leading coefficients satisfy $c_{k+1}=kc_k$ with $c_1=1$, hence $c_k=(k-1)!$. Each differentiation with respect to $y_{2,i}$ thus contributes exactly one factor $e^{\frac{1}{2}z_i^2}$, together with polynomial growth.

 Writing $w_{\bs{\gamma}}=\prod_i\gamma_ie^{-\frac{1}{2}z_i^2(\gamma_i^2-1)}$ and applying~\eqref{eq:phiinv.k} coordinatewise, the same induction gives
\begin{equation}\label{eq:weight.k}
   \partial^{\bs{\kappa}}_{\bs{y}_2}w_{\bs{\gamma}}(\bs{y}_2)=w_{\bs{\gamma}}(\bs{y}_2)\prod_{i=1}^{d_1-d_2}(2\pi)^{\frac{\kappa_i}{2}}R_{\kappa_i,i}(z_i)e^{\frac{\kappa_i}{2}z_i^2},
\end{equation}
where
\begin{equation}
 R_{0,i}\equiv 1,\quad R_{k+1,i}=R_{k,i}'+\left(k+1-\gamma_i^2\right)zR_{k,i},
\end{equation}
with $\deg (R_{k,i})\leq k$.

Next, let $\bs{v}(\bs{y},\bs{x})\coloneqq\bs{\Sigma}_{\bs{\varepsilon}}^{-\frac{1}{2}}(\bs{G}(\bs{y}_1)-\bs{G}(\bs{x}))+\bs{D}_{\bs{\gamma}}\bs{z}$, so that $\Lambda_{\bs{\gamma}}=\frac{1}{2}\lVert\bs{v}\rVert^2$ and, by Assumption~\ref{asu:model} and the compactness of $[0,1]^{d_2}$, it follows that $\lVert\bs{v}\rVert\leq\Pi(\bs{z})$ uniformly in $\bs{x}$ and $\bs{y}_1$. Derivatives of $\bs{v}$ with respect to $\bs{y}_1$ are bounded by Assumption~\ref{asu:model} and carry no exponential factor, whereas a derivative in $y_{2,i}$ acts only on $\gamma_iz_i$ and contributes $e^{\frac{1}{2}z_i^2}$ by~\eqref{eq:phiinv.k}. Expanding $\partial^{\bs{\beta}}\frac{1}{2}\langle\bs{v},\bs{v}\rangle$ by the multivariate general Leibniz rule as in the proof of Lemma~\ref{Lem:asu.2}, and noting that the exponential factors of the two arguments of the inner product result in $\prod_ie^{\frac{\beta_{2,i}}{2}z_i^2}$ in every term, we obtain that
\begin{equation}\label{eq:Lambda.beta}
   \sup_{\bs{x}\in[0,1]^{d_2}}\left|\partial^{\bs{\beta}}\Lambda_{\bs{\gamma}}(\bs{y},\bs{x})\right|\leq\Pi(\bs{z})\prod_{i=1}^{d_1-d_2}e^{\frac{\beta_{2,i}}{2}z_i^2}.
\end{equation}

 By Fa\`{a} di Bruno's formula applied to $e^{-\Lambda_{\bs{\gamma}}}$, we have that $\partial^{\bs{b}}e^{-\Lambda_{\bs{\gamma}}}=\mathcal{P}_{\bs{b}}e^{-\Lambda_{\bs{\gamma}}}$, where $\mathcal{P}_{\bs{b}}$ is a polynomial in the derivatives $\{\partial^{\bs{\rho}}\Lambda_{\bs{\gamma}}\}_{\bs{\rho}\leq\bs{b}}$ in which every monomial has total order $\bs{b}$; and hence $|\mathcal{P}_{\bs{b}}|\leq\Pi(\bs{z})\prod_ie^{\frac{b_{2,i}}{2}z_i^2}$ by~\eqref{eq:Lambda.beta}. Consequently, it follows that $\partial^{\bs{b}}J=\int\mathcal{P}_{\bs{b}}e^{-\Lambda_{\bs{\gamma}}}\di{}\bs{x}$ and
\begin{equation}
   \frac{\left|\partial^{\bs{b}}J\right|}{J}\leq\sup_{\bs{x}\in[0,1]^{d_2}}\left|\mathcal{P}_{\bs{b}}\right|\leq\Pi(\bs{z})\prod_{i=1}^{d_1-d_2}e^{\frac{b_{2,i}}{2}z_i^2},
\end{equation}
by a similar argument as in~\eqref{eq:first.term.first.term}. Applying Fa\`{a} di Bruno's formula to the outer function $\log$, whose $k$th derivative at $J$ is $(-1)^{k-1}(k-1)!J^{-k}$, it follows that
\begin{equation}
   \partial^{\bs{\mu}}\ell=\sum_{\pi}(-1)^{\left|\pi\right|-1}\left(\left|\pi\right|-1\right)!\prod_{b\in\pi}\frac{\partial^{\bs{b}}J}{J},
\end{equation}
and from this, it follows that
\begin{equation}\label{eq:ell.mu}
   \left|\partial^{\bs{\mu}}\ell(\bs{y})\right|\leq\Pi(\bs{z})\prod_{i=1}^{d_1-d_2}e^{\frac{\mu_{2,i}}{2}z_i^2}.
\end{equation}
For $\bs{\mu}=\bs{0}$ the same bound holds as $J\leq 1$ and $J\geq e^{-\sup_{\bs{x}}\Lambda_{\bs{\gamma}}}\geq e^{-\Pi(\bs{z})}$, resulting in $|\ell|\leq\Pi(\bs{z})$. Finally, applying the multivariate general Leibniz rule to $h_{\rm{IS}}=\ell\,w_{\bs{\gamma}}$ and combining~\eqref{eq:ell.mu} with~\eqref{eq:weight.k}, every term of
\begin{equation}
   \partial^{\bs{\beta}}h_{\rm{IS}}=\sum_{\bs{\mu}\leq\bs{\beta}}\binom{\bs{\beta}}{\bs{\mu}}\partial^{\bs{\mu}}\ell\;\partial^{\bs{\beta}-\bs{\mu}}w_{\bs{\gamma}}
\end{equation}
has an exponential factor 
\begin{equation}
\prod_{i=1}^{d_1-d_2}e^{\frac{\mu_{2,i}}{2}z_i^2}e^{\frac{\beta_{2,i}-\mu_{2,i}}{2}z_i^2}e^{-\frac{1}{2}z_i^2(\gamma_i^2-1)}=\prod_{i=1}^{d_1-d_2}e^{-\frac{1}{2}z_i^2(\gamma_i^2-1-\beta_{2,i})},
\end{equation}
independently of $\bs{\mu}$, which yields~\eqref{eq:is.master}. If $\gamma_i^2>1+\beta_{2,i}$ for every $1\leq i\leq d_1-d_2$, the right-hand side of~\eqref{eq:is.master} tends to zero as $|z_i|\to\infty$, that is, as $y_{2,i}$ approaches $0$ or $1$; hence $\partial^{\bs{\beta}}h_{\rm{IS}}$ is bounded on $(0,1)^{d_1}$ and extends continuously to $[0,1]^{d_1}$, and the extension is zero for $y_{2,i}\in\{0,1\}$. Taking $\gamma_i>\sqrt{\alpha+2}$ for all $1\leq i\leq d_1-d_2$ demonstrates the claim for every $\bs{\beta}\in\{0,1,\ldots,\alpha+1\}^{d_1}$.
\end{proof}